\documentclass[12pt]{amsart}

\usepackage[dvipsnames]{xcolor}
\usepackage{overpic, float}
\usepackage{eucal} % chjange mathcal font
\usepackage[linktocpage, pdfborder={0 0 .3}]{hyperref} % linktocpage (link at the page number)

\usepackage{cancel}
\usepackage{cleveref}
\usepackage{soul}
\usepackage{pict2e} % overic arrow
\usepackage{mathtools}

\usepackage{leftidx} % upper left index
\usepackage{stmaryrd}

\usepackage{tikz-cd}
\usepackage{tikz}
\usetikzlibrary{shapes.geometric, arrows}
\usetikzlibrary{fit,calc,positioning}
\usepackage{contour}
\contourlength{1.2pt}

\setstcolor{red}

\newtheorem{lemma}{Lemma}[section]
\newtheorem{proposition}[lemma]{Proposition}
\newtheorem{theorem}[lemma]{Theorem}
\newtheorem{claim}[lemma]{Claim}

\newtheorem{remark}[lemma]{Remark}
\newtheorem{conjecture}[lemma]{Conjecture}
\newtheorem{corollary}[lemma]{Corollary}
\newtheorem{definition}[lemma]{Definition}

\newtheorem{thm}{Theorem}

\newcommand{\R}{\mathbb{R}}
\newcommand{\C}{\mathbb{C}}
\newcommand{\Z}{\mathbb{Z}}
\renewcommand{\H}{\mathbb{H}}

\newcommand{\TT}{\mathcal{T}}
\newcommand{\PP}{\mathcal{P}}
\newcommand{\BB}{\mathcal{B}}
\newcommand{\MM}{\mathcal{M}}
\newcommand{\HH}{\mathcal{H}}
\newcommand{\QQ}{\mathcal{Q}}

\newcommand{\RR}{\mathcal{R}}
\newcommand{\LL}{\mathcal{L}}
\newcommand{\CC}{\mathcal{C}}

\newcommand{\QQQ}{\mathbf{Q}}
\newcommand{\HHH}{\mathbf{H}}
\newcommand{\NNN}{\mathbf{N}}
\newcommand{\RRR}{\mathbf{R}}

\newcommand{\Ep}{{\rm Ep}}

\newcommand{\Hol}{{\rm Hol}}
\newcommand{\gr}{{\rm gr}}
\newcommand{\Gr}{{\rm Gr}}
\renewcommand{\Im}{{\rm Im}}

\newcommand{\ML}{{\rm ML}}
\newcommand{\PML}{{\rm PML}}
\newcommand{\PSL}{{\rm PSL}}
\newcommand{\length}{{\rm length}}
\newcommand{\dev}{{\rm dev}}
\newcommand{\QF}{{\rm QF}}

 \newcommand{\Label}[1]{\label{#1}\textcolor{green}{\tiny #1} } 
\renewcommand{\Label}[1]{\label{#1}}

\newcommand{\CP}{\mathbb{C}{\rm P}}

\newcommand{\col}{\colon}
\newcommand{\minus}{\setminus}
\newcommand{\bdr}{\partial}

\newcommand{\ti}{\tilde}

\newcommand{\ep}{\epsilon}
\newcommand{\gam}{\gamma}

\newcommand{\kap}{\kappa}
\newcommand{\lam}{\lambda}

\newcommand{\Del}{\Delta}

\DeclareRobustCommand{\rchi}{{\mathpalette\irchi\relax}}
\newcommand{\irchi}[2]{\raisebox{\depth}{$#1\chi$}} % inner command, used by \rchi

\newcommand{\Qed}[1]{\nopagebreak[4]{\tiny \hfill\fbox{\ref{#1}} \linebreak }\pagebreak[2]}

\definecolor{dblue}{cmyk}{1,.5, 0,.1}
\definecolor{arsenic}{rgb}{0.23, 0.27, 0.49}

 \definecolor{calpolypomonagreen}{rgb}{0.12, 0.3, 0.17}
 
  \definecolor{darkbyzantium}{rgb}{0.6, 0.3, 0.4}
 \definecolor{azure}{rgb}{0.0, 0.5, 1.0}
 \definecolor{cittingcolor}{cmyk}{60,0,10,0}
  
 \hypersetup{pdfborder=0 0 .3}
	\hypersetup{
    colorlinks=true,
   linkcolor=darkbyzantium,
    filecolor=magenta,      
    urlcolor=NavyBlue,
    citecolor = calpolypomonagreen,
    pdftitle={Overleaf Example},
    pdfpagemode=FullScreen,
     }

 \title[\today]{Intersection number of holonomy varieties of $\CP^1$-structures}

\author{Shinpei Baba}
\address{University of Osaka}

\email{sb.sci@osaka-u.ac.jp}
\date{\today}

\begin{document}
\begin{abstract} 
Let $\Sigma$ be a closed orientable surface of genus at least two, and let $X, Y$ be distinct marked Riemann surface structures on $\Sigma$, possibly with opposite orientations.
 In this paper, we show that there are (exactly) countably infinite pairs of $\CP^1$-structures on $X$ and on $Y$ sharing holonomy $\pi_1(\Sigma) \to \PSL_2\C$.
   \end{abstract}
\maketitle

\setcounter{tocdepth}{1}
\tableofcontents

\section{Introduction}
Let $\Sigma$ be a closed orientable surface of genus $g$ at least two. 
A {\sf quasi-Fuchsian} representation $\rho\col \pi_1(\Sigma) \to \PSL_2\C$ is a ``typical" discrete and faithful representation, such that the limit set of the discrete subgroup $\Im \rho$ is a Jordan curve on $\CP^1$.
Quasi-Fuchsian representations are important for the study of hyperbolic three-manifolds.
It has various equivalent definitions (such as convex cocompact representations), and also generalizations to representations into other Lie groups. (For other rank-one Lie groups, see Bowditch \cite{Bowditch95}.  See also \cite{KapovichLeeb18IsometryGroupsOfSpaces}).  
In this paper, we investigate another type of generalization of quasi-Fuchsian representations within the framework of $\pi_1(\Sigma) \to \PSL_2\C$. 

Let $S$ be the surface $\Sigma$ with a fixed orientation, and $S^\ast$ be $\Sigma$ with the opposite orientation. 
Let $\TT$ be the Teichmüller space of $S$, the deformation space of all marked Riemann surface structures on $S$.
Similarly let $\TT^\ast$ be the Teichmüller space of $S^\ast$.
Given a quasi-Fuchsian representation $\rho\col \pi_1(S) \to \PSL_2\C$, let $\Lambda$ be the limit set of $\Im \rho$. 
Then its complement $\CP^1 \minus \Lambda$ is a union of disjoint topological open disks $\Omega^+$ and $\Omega^-$, and the quotients $\Omega^+/\Im \rho$ and $\Omega^-/\Im \rho$ have marked Riemann surface structures on $S$ and $S^\ast$.
Bers' simultaneous uniformization theorem (\cite{Bers60}) asserts that,  for every pair of Riemann surface structures $X$ on $S$ and $Y$ on $S^\ast$, there is unique quasi-Fuchsian representation $\rho\col \pi_1(S) \to \PSL_2\C$, such that $X$ and $Y$ are biholomorphic to $\Omega^+/\Im \rho$ and $\Omega^-/\Im \rho$.
By this correspondence, the space $\QF$ of quasi-Fuchsian representations, {\it the quasi-Fuchsian space}, is indeed biholomorphic to $\TT \times \TT^\ast$ (\cite{Bers60}, see also \cite{Hubbard06, EarleKra06}).  
In particular, the quasi-Fuchsian representation $\rho$ is the unique transversal intersection point of the Bers' slices $\{X\} \times \TT^\ast$ and $\TT \times \{T\}$ in $\QF$.

The quotient surfaces $\Omega^+/\Im \rho$  and $\Omega^-/\Im \rho$ have not only Riemann surface structures but also have {\it $\CP^1$-structures (or complex projective structures)} on $S$ and $S^\ast$, which correspond to holomorphic quadratic differentials on these Riemann surfaces.
From a viewpoint of $\CP^1$-structures, the simultaneous uniformization theorem can equivalently be stated as follows, without the notion of quasi-Fuchsian representations:
Given a pair of Riemann surface structures  $X$ on $S$ and $Y$ in $S^\ast$, there is a unique pair of $\CP^1$-structure $C_X$ on $X$ and a $\CP^1$-structure $C_Y$ on $Y$ such that 
\begin{itemize}
\item the holonomy representation $\pi_1(\Sigma) \to \PSL_2\C$ of $C_X$ coincides with the holonomy representation $\pi_1(\Sigma) \to \PSL_2\C$ of $C_Y$ (up to conjugation), and
\item the developing maps $\ti{S} \to \CP^1$ of $C_X$ and $\ti{S}^\ast \to \CP^1$ $C_Y$ are injective, where $\ti{S}$ and $\ti{S}^\ast$ are the universal covers of $S$ and $S^\ast$, respectively. 
\end{itemize}

In this paper,  we consider a more general realization problem of a pair of Riemann surface structures $X$ and $Y$ on either $S$ or $S^\ast$ by a pair of $\CP^1$-structures $C_X$ and $C_Y$ sharing holonomy. 
In this general setting without the restriction of the injectivity and the orientation, we show that there are infinitely many realizing pairs: 
\begin{thm}\Label{InfinitelyManyPairs}
Let  $X, Y \in \TT \cup \TT^\ast$ with $X \neq Y$.
Then, there are exactly countably infinitely many distinct pairs $(C_i^X, C_i^Y)_{i = 1}^\infty$ of $\CP^1$-structures  $C^X_i$ on $X$ and $C^Y_i$ and $Y$,  such that the holonomy $\pi_1(S) \to \PSL_2\C$ of $C_i^X$ coincides with the holonomy of $C_i^Y$ for each $i = 1, 2, \dots$. (\Cref{InfiniteIsomodromicPairs}.)
\end{thm}
Note that $\TT \cup \TT^\ast$ is the space of all marked Riemann surface structures on $S$ (without fixing the orientation). 
The orientations of $X$ and $Y$ can be either the same or the opposite, in contrast to Bers' theorem. 

Next we interpret \Cref{InfinitelyManyPairs} in 
the {\sf $\PSL_2\C$-character variety} $\rchi$ of $\Sigma$,  the space of representations divided by conjugation: 
$$\{\pi_1(S) \to \PSL_2\C\} \sslash \PSL_2\C.$$
In this affine algebraic variety  $\chi$,
  there are various half-dimensional (closed real or complex smooth analytic) subvarieties with geometric significance. 
It has been important to understand the intersection of such half-dimensional subvarieties (Faltings \cite{Faltings83}, Dumas-Wolf \cite{Dumas-Wolf08}, Tanigawa \cite{Tanigawa-97}). 

Here we shall consider the intersection of {\it holonomy varieties}. 
For a Riemann surface structure $X$ on $\Sigma$, 
the set $\PP_X$ of $\CP^1$-structures on $X$ is identified with the complex affine vector space ${\rm QD}(X) \cong \C^{3g -3}$ of holomorphic quadratic differentials on $X$. 
Then, the deformation space $\PP_X$ properly embeds into the $\PSL_2\C$-character variety of $\Sigma$ by the holonomy map (Poincar\'e \cite{Poincare884}, Kapovich \cite{Kapovich-95}, Gallo Kapovich Marden \cite{Gallo-Kapovich-Marden}, see also \cite{Tanigawa99, Dumas18HolonomyLimit}). 
 Its image is a smooth complex analytic subvariety of $\rchi$, and it is called the {\sf holonomy variety} of $X$,  and we denote it by $\rchi_X$.

\begin{thm}\Label{InfiniteIntersection}
For all distinct $X, Y \in \TT \cup \TT^\ast,$
the intersection 
$$\rchi_X \cap \rchi_Y$$ 
is an infinite discrete closed subset of the character variety $\rchi$. 
Thus their algebraic intersection number is infinite.
\end{thm}
As $\rchi_X$ and $\rchi_Y$ are smooth analytic subvarieties, the signs of the (isolated) intersection points are all positive; thus in total they have infinite algebraic intersection number. 
In fact, the points of the intersection 
$\rchi_X \cap \rchi_Y$ bijectively correspond to the infinite pairs  $(C_i^X, C_i^Y)_{i = 1}^\infty$ in \Cref{InfinitelyManyPairs}.

The {\it Fuchsian space} is a component of the real (analytic) slice of $\rchi$ which consists of discrete and faithful representations $\pi_1(S) \to \PSL_2\R$ respecting the orientation of $S$.  
 \Cref{InfiniteIntersection} can be regarded as reminiscent of Tanigawa's theorem stating that, for every $X \in \TT$, the holonomy variety $\rchi_X$ intersects the Fuchsian space in a discrete set (\cite[Theorem 12]{Faltings83})  infinitely many times   (\cite{Tanigawa-97}).

Last we relate our main theorem to the deformation space of isomonodromic pairs of $\CP^1$-structures.
Consider the space $\BB$ of (ordered) pairs of distinct $\CP^1$-structures on $\Sigma$ sharing holonomy. 
Then the quasi-Fuchsian space is biholomorphically identified with a connected component of $\BB$ unique up to switching the ordering of paired $\CP^1$-structures.
Some pairs in  $\BB$ with opposite orientations appear as the ideal boundary of 3-dimensional hyperbolic cone manifolds of $2\pi$-multiple cone angles which are homeomorphic to $S \times (-1,1)$ (c.f. \cite{Bromberg-07}).

Let $$\Psi \col \BB \to (\TT \sqcup \TT^\ast)^2 \minus \Delta$$
be the uniformization map taking a pair $(C, D)$ in $\BB$ to the pair of the underlying Riemann surface structures of $C$ and $D$. 
Then the author previously proved that the analytic mapping $\Psi$ is a complete local branched covering map (\cite[Theorem A]{Baba_23}). 
This local property implies that every fiber of $\Psi$ is closed and discrete.
Our theorem provides the explicit cardinality of its fibers. 
\begin{thm}\Label{InfiniteFiber}
Every fiber of $\Psi$ is a (countable) infinite discrete closed subset of $\BB$. 
\end{thm}

The $\Psi$-fiber over $(X, Y)$ is exactly the infinite pairs $(C_i^X, C_i^Y)_{i = 1}^\infty$ in \Cref{InfinitelyManyPairs}. 
\Cref{InfiniteFiber} suggests a possibility of $\BB$ having infinitely many connected components, since each connected component must contain at least one fiber point, due to the completeness of the uniformization map $\Psi$ (the path lifting property).

\Cref{InfiniteIntersection} and \Cref{InfiniteFiber} follow immediately from \Cref{InfinitelyManyPairs}.
In fact, those three theorems are all equivalent. 

In our theorems, the ``infiniteness" is new. 
On the other hand, the discreteness in those theorems was proven by the author (\cite[Theorem C]{Baba_23}), and thus the cardinality has been known to be, at most, a countable set.
In this paper, we show that this upper bound is sharp by constructing infinitely many pairs.   
As for the lower bound, it has been known that the cardinality of  $\rchi_X \cap \rchi_Y$ is at least two if the orientations of $X$ and $Y$ are opposite and the cardinality of  $\rchi_X \cap \rchi_Y$ is at least one if the orientations of $X$ and $Y$ are the same (\cite[Corollary 12.7]{Baba_23}). 

The technical aspect of the proof concerns the strong asymptotic property of Teichmüller (geodesic) rays and grafting rays. 
Gupta, in his thesis (\cite{Gupta14, Gupta15}),  proved that, given every conformal grafting ray in the Teichmüller space, there is a Teichmüller ray asymptotic to it, as unparametrized rays. 
In his construction, typically those asymptotic rays have different base points. 

In contrast, we show such an asymptotic property, as parametrized rays,  for a certain family of pairs of a Teichüller ray and a grafting ray which share their base point.  
Moreover, this family has a uniform asymptotic rate (\Cref{SplitingLimitRay}). 

By \Cref{InfiniteFiber}, the intersection $\rchi_X \cap \rchi_Y$ is an infinite discrete set in the character variety $\chi$. 
 Thus it is interesting to understand how $\chi_X \cap \chi_Y$ distributes in $\chi$. 
The Morgan-Shalen boundary $\bdr \rchi$ is the compactification of $\chi$ by $\pi_1(S)$-actions on $\R$-trees, realizing projective limits of the translation lengths of closed curves by the representations $\pi_1(S) \to \PSL_2\C$. 
The Morgan-Scalen $\bdr \chi$ contains Thurston boundary of the Teichmüller space (\cite[\S 11.16]{Kapovich-01}).  
Perhaps one can generalize the construction of the  points in $\chi_X \cap \chi_Y$ in this paper and show that the intersection points are distributed in almost all directions: 

\begin{conjecture}
For all distinct $X, Y \in \TT \cup \TT^\ast,$ 
the accumulation set of the intersection $\chi_X \cap \chi_Y$ in the Morgan-Shalen boundary contains the Thurston boundary of the Teichmüller space in $\chi$. 
\end{conjecture}

\subsection{Idea of the proof}
We begin with outline the proof of \Cref{InfinitelyManyPairs} in the case where  $X, Y$ are Riemann surface structures on $S$,  i.e. the case where the realizing Riemann surfaces have the same orientation.  
Supposing that there are already $n$ isomonodromic pairs $(C_1^X, C_1^Y), \dots, (C_n^X, C_n^Y)$ of $\CP^1$-structures on $X$ and $Y$, we construct a new isomonodromic pair as follows. 

We take a ``generic'' Teichmüller (bifinte) geodesic $X_t ~(t\in \R)$ in $\TT$ which passes very close to $X$ and $Y$, so that its projection $[X_t]$ is dense in the Moduli space of Riemann surfaces in particular. 
Let $\rho_t\col \pi_1(S) \to \PSL_2\C$ be the discrete faithful representation uniformizing $X_t$, so that the marked hyperbolic surface $\H^2/ \Im \rho_t$ is conformally identified with the marked Riemann surface $X_t$.  

Clearly the mapping $\R \to \rchi$ defined by $t \mapsto \rho_t$ is a proper embedding.
Thus we can take sufficiently small $t < 0$ so that $\rho_t$ is sufficiently far from the $n$ holonomy representations of $(C_1^X, C_1^Y), \dots, (C_n^X, C_n^Y)$.
In addition,  using $2\pi$-grafting,  we can construct $\CP^1$-structures $C_{X'}, C_{Y'}$ with holonomy $\rho_t$ whose underlying Riemann surface structures $X'$ and $Y'$ are very close to the desired Riemann surface structures $X$ and $Y$: 
To achieve this construction, we prove and combine the following two approximations: the uniform asymptotic properties of certain infinite pairs of a Teichmüller ray and a corresponding grafting ray, associated with an accumulation point of the projection $[X_t]$ in the moduli space  (\Cref{SplitingLimitRay}), and corresponding uniform approximation of grafting rays by integral graftings (\Cref{IntegralGrafting}).
 
Then, by the completeness of $\Psi\col \BB \to (\TT \cup \TT^\ast)^2 \minus \Delta$, we can deform this pair $(C_{X'}, C_{Y'})$ to $C_X, C_Y$ in $\BB$ so that their underlying Riemann surface structures are exactly $X$ and $Y$.
As $\rho_t$ is sufficiently far from the holonomy representations of the already given pairs, we can conclude that the deformed new pair $(C_X, C_Y)$ realizing $(X, Y)$ is different from the $n$ pairs  $(C_1^X, C_1^Y), \dots, (C_n^X, C_n^Y)$ we already have. 

In the case where the orientations of $X$ and $Y$ are opposite, the proof is reduced to the case where $X$ and $Y$ have the same orientation by appropriate complex conjugation  (\S \ref{sOppositeOrieantations}).
We also give a short alternative proof, using Tanigawa's result and the properties of the uniformization map $\Psi$ (\S\ref{sAlternativeProof}).

\section{Acknowledgment}
I thank Brian Collier, Sebastian Heller, Takuro Mochizuki, Kyoji Saito for stimulating conversation and correspondence. 
 I also thank Subhojoy Gupta for discussions about Teichmülller rays. 
 
 This work is partially supported by Grant-in-Aid for Scientific Research 24K06737 and 23K22396.

\section{Preliminaries}
\subsection{Teichmüller rays}(See \cite[\S 11]{FarbMarglit12} for instance.)
The Teichmüller space $\TT$ of $S$ is the space of marked Riemann surface structures on $S$ up to isotopy. 
Given two (marked) Riemann surfaces $X, Y \in \TT$, let $K = K(X, Y)$ denote the infimum of the quasi-conformal dilatations $K_f $ among all quasi-conformal mappings $X \to Y$ preserving the marking of the surface. 
The {\sf Teichmüller distance} between $X$ and $Y$ on $\TT$ is given by 
$$d (X, Y) = \frac{1}{2} \log K, $$
which gives a Finsler metric on $\TT$, call the {\sf Teichmüller metric}.

A geodesic in $\TT$ in the Teichmüller metric is called a {\sf Teichmüller geodesic}. 
Then, given $X \in \TT$ and a measured foliation $V$ on $S$,  there is a Teichmüller geodesic $X_t$ with $X_0 = X$ along which $V$ ``conformally shrinks". 
Namely,  by Hubbard and Masur \cite{HubbardMasur79}, there is a flat surface $E = E(X, V)$ conformal to $X$ such that $V$ is the vertical measured foliation. 
In this paper, a {\sf flat surface} is an Euclidean surface with isolated cone singularities whose cone angles are $\Z_{\geq 3}$-multiples of $\pi$.
Then, we can obtain a ray of flat surfaces $E_t$ obtained by stretching in the horizontal direction by $e^{t/2}$ and shrinking in the horizontal direction by $e^{-t/2}$.
The conformal structure of $E_t$ gives the Teichmüller geodesic at unit speed. 

\subsection{$\CP^1$-structures}
(General references are \cite{Dumas-08}, \cite[Chapter 7]{Kapovich-01}, \cite[Chapter 14]{Goldman22GeometricStructuresOnManifolds}.) 
Recall that $\PSL_2\C$ is the automorphism group of $\CP^1$. 
Then, a {\sf $\CP^1$-structure} on a surface is a $(\CP^1, \PSL_2\C)$-structure.
Namely, an atlas of embedding open subsets covering $S$ into $\CP^1$ such that transition maps are given by elements in $\PSL_2\C.$
Clearly, each $\CP^1$-structure has a Riemann surface structure since transition maps preserve complex structure. 

$\CP^1$-structures can be viewed from various viewpoints, including the ones we discuss below.  
\subsubsection{Developing pairs}
A $\CP^1$-structure can equivalently be defined using a ``global coordinate" on the universal cover $\ti{S}$ of $S$.
Namely, a $\CP^1$-structure on $X$ is a pair $(f, \rho)$ of
\begin{itemize}
\item a local diffeomorphism $f \col \ti{S} \to \CP^1$ ({\sf developing map}) and 
\item a homomorphism $\rho\col \pi_1(S) \to \PSL_2\C$ ({\sf holonomy representation}), 
\end{itemize}
such that $f$ is {\it $\rho$-equivariant} (i.e. $f \gam = \rho(\gam) f$ for all $\gam \in \pi_1(S)$).
   This {\it developing pair} is defined up to $\PSL_2\C$. 
Namely
$(f, \rho)$ is equivalent to $(\alpha f,  \alpha \rho \alpha^{-1})$ for all $\alpha \in \PSL_2\C$.

\subsubsection{Schwarzian parametrization}
Next we explain an analytic viewpoint of a $\CP^1$-structure.
 A $\CP^1$-structure $C = (f, \rho)$ on $S$ corresponds to a holomorphic quadratic differential $q = \phi\, dz^2$ on a Riemann surface. 
The developing map $f$ induces a Riemann surface structure $X$ on $S$, and the Schwarzian derivative of $f$ induces a holomorphic quadratic differential on $X$. 
Thus a space of $\CP^1$-structures on a Riemann surface $X$ is an affine vector space of holomorphic quadratic differentials on $X$. 

There is a unique marked hyperbolic structure $\sigma$ uniformizing $X$.
A hyperbolic structure is, in particular, a $\CP^1$-structure, and, in this paper, we pick this hyperbolic structure to be the zero of this vector space--- in other words, when we take the Schwarzian derivative of $f$, the domain $\ti{X}$, the universal cover of $X$,  is identified with the upper half plane of $\C$ by the uniformization theorem.  

Let $\PP$ be the space of all marked $\CP^1$-structures on $S$. 
Let $\psi \col \PP \to \TT$ be the projection map which takes each $\CP^1$-structure to its complex structure. 

\subsubsection{Thurston's parametrization}(See \cite{Kamishima-Tan-92, Kulkani-Pinkall-94}, also \cite{Baba20ThurstonParameter}.) \Label{sThurstonParameters}
The Riemann sphere $\CP^1$ is the ideal boundary of the hyperbolic three-space $\H^3$. 
Moreover, $\PSL_2\C$ is the orientation-preserving isometry group of $\H^3$, extending the action on $\CP^1$. 

Through such relation,  a $\CP^1$-structure $C = (f, \rho)$ corresponds to a pair $(\sigma, L)$  of a hyperbolic structure $\sigma$ on $S$ and a measured (geodesic) lamination $L$ on $S$. 
This pair is called {\sf Thurston's parameters} of $C$. 

Let $\ti{L}$ be the  $\pi_1(S)$-invariant measured lamination on the universal cover $\H^2$ of $\sigma$. 
Then, this pair $(\sigma, L)$ give an $\rho$-equivariant ``locally convex" surface $\beta\col \H^2 \to \H^3$, called {\sf a bending map}, obtained by bending $\H^2$ along $\ti{L}$ by the angle given by the transversal measure of $\ti{L}$.
A bent surface is a particular type of pleated surface.
The developing map $f\col \ti{S} \to \CP^1$  corresponds to $\beta$ in a $\rho$-equivariant manner by certain ``locally'' well-defined nearest point projections from parts of $\CP^1$ to the pleated surface.

Let $\ML$ denote the space of all measured laminations on $S$. 
Then we have Thurston's parametrization of the deformation space
\begin{eqnarray}
\PP \cong \TT \times \ML, \Label{iThurston}
\end{eqnarray}\begingroup  
 by a canonical tangential homeomorphism. 
In particular, if $C \in \PP$ correponds to $(\sigma, L) \in \TT \times \ML$, the bending map $\H^2 \to \H^3$ is equivariant via the holonomy representation of $C$. 

    For each periodic leaf $\ell$ of $L$, there is a round cylinder $A_\ell$ in the grafted surface $\Gr_L \sigma$ foliated by circular closed curves such that the height of $A_\ell$ is the weight of $\ell$. 
    If there is more than one periodic leaf, then their corresponding round cylinders are disjoint. 
    The collapsing map  $\kap\col \Gr_L \sigma  \to \sigma$ collapses each cylinder $A_\ell$ to the closed geodesic $\ell$ on $\sigma$ and the restriction of $\kap$ to the complement of the cylinders is a $C^1$-diffeomorphism onto the complement of closed leaves 
    Therefore, there is a measured lamination $\LL$ on the grafted surface $\Gr_L \sigma$, such that 
    \begin{itemize}
\item the leaves of $\LL$ are circular,  and
\item $\kap$ takes $\LL$ to $L$ on $\sigma$.
\end{itemize}
    This lamination $\LL$ of the $\CP^1$-structure  $\Gr_L \sigma$ is called {\sf Thurston's lamination}.

\subsubsection{Thurston metric of a $\CP^1$-surface} \Label{sThurstonMetric}(\cite[S 2.1]{Tanigawa-97} \cite[\S 4.3]{Dumas-08}.) 
If a projective structure  $C \in \PP$ corresponds to $ (\sigma, L) \in \TT \times \ML$ by Thurston's parametrization (\ref{iThurston}),  then its grafting construction yields a conformal piecewise Hyperbolic/Euclidean-metric the $\CP^1$-structure $C = \Gr_L \sigma$ as follows: 
First suppose that $\ell$ is a geodesic loop on the hyperbolic surface $\sigma$ with weight $w > 0$.
  Then, as described above, the grafting of $\sigma$ along $\ell$ inserts a cylinder $A_\ell$ along $\ell$. 
 This cylinder  $A_\ell$ can be regarded as a Euclidean cylinder with geodesic boundary, such that the height is $w$ and the circumference is the hyperbolic length of $\ell$. 
  The $\Gr_\ell \sigma \minus A_\ell$ has a hyperbolic metric with geodesic boundary isometric to $\ell$, induced from the hyperbolic metric of $\sigma$. 
  Therefore, the hyperbolic metric and the Euclidean metric on  $\Gr_\ell \sigma$ are isometric along their boundary components, and we obtain a conformal metric on  $\Gr_\ell \sigma$.
  This is the Thurston metric on $\ Gr_\ell\sigma$.
  
  Suppose next that the measured lamination $L$ is a weighted (geodesic) multiloop $M$ on $\sigma$.
  Then we can similarly insert an Euclidean cylinder of height equal to the weight along each loop on $M$, and obtain a conformal hyperbolic/Euclidean structure on $\Gr_M \sigma$.
  
   For a general measured lamination $L$, let $\ell_i$ be a sequence of weighted loops on $\sigma$ converging to $L$ as $i \to \infty$.
   Then the Thurston metric on $\Gr_L \tau$ is the limit of Thurston metric on $\Gr_{\ell_i} \sigma$ as $i \to \infty$.
  \subsubsection{$2\pi$-grafting } 
A $2\pi$-grafting is, in general,  a cut-and-paste operation of a $\CP^1$-structure along an (admissible) loop or multiloop, and it creates a new $\CP^1$-structure preserving the holonomy representation.  
Grafting deformation was first developed by Maskit (\cite{Maskit-69}),  Hejhal (\cite{Hejhal-75}) Sullivan-Thurston (\cite{Sullivan-Thurston-83}).
It was further developed by Goldman (\cite{Goldman-87}). 
 We describe $2\pi$-grafting of a hyperbolic surface along a geodesic loop or multiloop, which we utilize in this paper. 

Let $\sigma$ is a hyperbolic surface, and $\ell$ is a geodesic loop on $\sigma$ with weight $w \in \Z_{> 0}$. 
Let $\rho\col \pi_1(S) \to \PSL_2\R$ be the discrete and faithful representation corresponding to $\sigma$ (unique up to $\PSL_2\R$), so that $\sigma$ is isometric to $\H^2/ \Im \rho.$
Note that the hyperbolic structure $\sigma$ is in particular a $\CP^1$-structure with holonomy $\rho$, and its developing map is a diffeomorphism onto  $\H^2$ as a subset of $\CP^1$. 
Then, by grafting $\sigma$ along $\ell$ with weight $2\pi w$, we obtain a $\CP^1$-structure on $S$ with the Fuchsian holonomy $\rho$.
Indeed, by bending a totally geodesic hyperbolic plane in $\H^3$ by angle $2\pi w$ along lifts of the geodesic loop $\ell$, the plane is {\it not} bent after all and it retains the same shape.
Therefore, as the bending map is holonomy equivariant,  $\Gr_{2\pi \ell} \sigma$ has the same holonomy $\rho$. (\Cref{fGrafting}.)
\begin{figure}
\begin{overpic}[scale=.15%, grid,tics=10
] {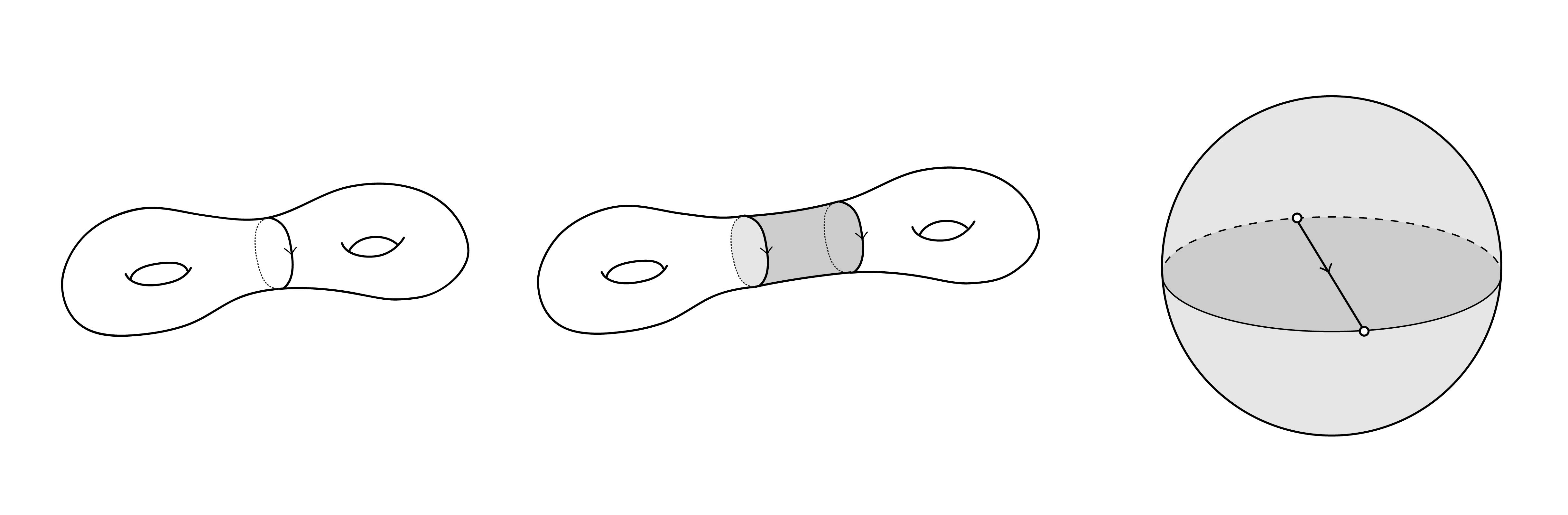} % figure file
 \put(10, 22){\textcolor{black}{\small \contour{white}{$\sigma$}}}   
  \put(19.5, 17){\textcolor{black}{\small \contour{white}{$\ell$}}}   
  \put(35, 22){\textcolor{black}{\small \contour{white}{$\Gr_{2\pi \ell} \sigma$}}}   
 \put(47, 21){\textcolor{darkgray}{\small \contour{white}{$A_{2\pi \ell}$}}}   
 \put(71, 28){\textcolor{black}{\small \contour{white}{$\CP^1$}}}   
 \put(80, 28){\textcolor{darkgray}{\small \contour{white}{$\Im\, \dev(A_{2\pi\ell})$}}} 
 % \put(, ){\textcolor{}{\tiny \contour{white}{$$}}}   
 %   \put( , ){\textcolor{}{$$}}  
 %   \put( , ){}  
%\put(, ){\color{}\vector(,){}} %{length}
      \end{overpic}
\caption{$2\pi$-Grafting. Middle: The cylinder $A_{2\pi \ell}$ inserted by grafting.  Right: the grafting cylinder develops onto $\CP^1$ minus the endpoints of the holonomy hyperbolic element. }\Label{fGrafting}
\end{figure}
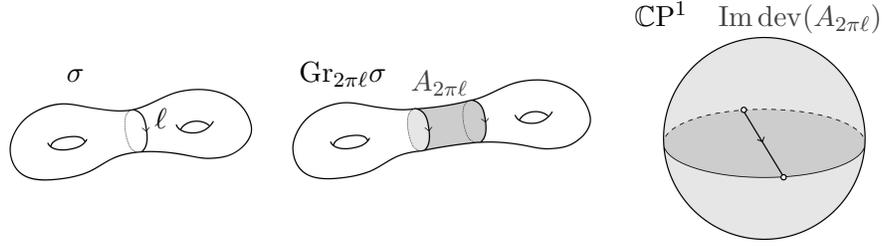

Suppose that $M$ is a $\Z_{> 0}$-weighted multiloop on a hyperbolic surface $\sigma$. 
Then, by $2\pi$-grafting $\sigma$ along $M$, we similarly obtain a new $\CP^1$-structure with the same holonomy $\rho$.  
 \subsubsection{Grafting rays}

Let $\sigma$ be a hyperbolic structure on $S$. 
Let $L$ be a measured geodesic lamination on $\sigma$.
Then we obtain a ray of $\CP^1$-structures $\Gr_{t L} \sigma$ corresponding to $(\sigma, t L), t \geq 0$ by Thurston's parametrization (\ref{iThurston}).
 This ray in the deformation space $\PP$ is called {\sf a (projective) grafting ray.}

Let $\gr_{L} \sigma \in \TT$ denote  $\psi (\Gr_{L} \sigma)$,  the complex structure of $\Gr_{L}\sigma$. 
Then, $\gr_{t L} \sigma, t \geq 0$ is called the conformal grafting ray from $\sigma$ in the direction of $L$.

\subsection{Epstein surfaces} (\cite{Epstein84}, see also \cite{Anderson98ProjectiveStructuresOnRiemannSurfacesAndDevelopingMaps}.)\Label{sEpsteinSurfaces}
Let $(f, \rho)$ be a developing pair of a projective structure $C$ on $S$, where $f \col \ti{S} \to \CP^1$ is its developing map.

Each point $x \in \H^3$ determines a unique spherical metric $s_x$ on $\CP^1$ by normalizing the Poincar\'e disk model of $\H^3$ so that $x$ is at the center of the 3-disk. 
Given a conformal metric $\mu$ on $C$,  there is a unique mapping $\Ep\col \ti{S} \to \H^3$ such that, the spherical metric $s_{\Ep (x)}$ of $\CP^1$ centered at $\Ep(x)$ coincides with the push-forward metric of $\mu$ at the tangent space $T_{f(x)} \CP^1$.
In fact, this surface is the envelope of the horospheres centered at the points $f(x)$ for $x \in \ti{S}$, and $\Ep$ is also $\rho$-equivariant.

Let $C \cong (X, q)$ be the Schwarzian parametrization of $C$. 
Then the quadratic differential $q$ gives a singular Euclidean metric on $C$, where the singular points are the zeros of the differential.
In this paper, we use the Epstein surface given by this singular Euclidean metric.  
In this case, the Epstein surface is well-defined and smooth in the complement of the singular points (and it continuously extends to the singular points as a mapping to points on $\CP^1$).

\subsection{Traintracks}(\cite{Penner-Harer-92})
A {\sf traintrack graph} on a surface is a (locally finite) $C^1$-smooth graph $G$ such that, at each vertex $p$ of $G$, the edges of $G$ with its endpoint at $p$ are divides into two groups $e_1, \dots, e_m$ and $f_1, \dots, f_n$ such that 
\begin{itemize}
\item the vectors $v$ tangent  to the edges $e_1, \dots, e_m$ at $p$ coincide, 
\item the vectors $u$ tangent to the edges $f_1, \dots, f_n$ at $p$ coincide, and 
\item $v = -u$ in the tangent space at $p$.
\end{itemize}
In this paper, we use traintrack graphs which are trivalent, i.e. $m =1 $ and $n =2$.
  
A {\sf marked rectangle} is a rectangle such that a pair of opposite edges is marked as vertical edges and the other pair is marked as horizontal edges. 
A {\sf fat traintrack} $T$ is an orientable surface with boundary which is obtained by taking a union of marked rectangles $\{R_i\}$ along horizontal edges as follows:  
Divide some horizontal edges into finitely many segments, pair up all horizontal edge segments, and glue each paired horizontal edge by a diffeomorphism.
Each rectangle $R_i$ of $T$ is called a {\sf branch}. 
The points dividing a horizontal edge of a branch appear as non-smooth points on the boundary of $T$.

More generally, a {\sf marked polygon} is a polygon with an even number $2n$ of edges,  such that a set of alternating $n$ edges are marked as {\it vertical edges} and the set of the other alternating $n$ edges are marked as {\it horizontal edges}.
A {\sf polygonal traintrack} is an orientable surface with boundary with singular points obtained, similarly by gluing some marked polygons as follows: 
Divide each horizontal edge into finitely many segments (if necessary), pair up all horizontal segments of the polygons, and glue each pair of horizontal segments by a diffeomorphism.  
Each polygon of the polygonal traintrack is also called a {\sf branch.}
Similar to the case of a fat traintrack, the points dividing a horizontal edge of a branch appear as non-smooth points on the boundary of the polygonal traintrack.
(\cite{Baba_23}.)

Given a lamination $\lambda$ on a hyperbolic surface $\sigma$, a {\sf traintrack neighborhood} $\sigma$ of $\lambda$ is a fat traintrack containing $\lambda$ in its interior, such that, for each branch $R$ of $\sigma$, each component of $\lam \cap R$ is an arc connecting opposite horizontal edges of $R$. 

\begin{definition}[cf. \cite{Minsky92}]
A traintrack neighborhood $\sigma$ on a hyperbolic surface $\sigma$ is $\ep$-nearly straight, if all boundary geodesics have curvature less than $\ep$ at non-singular points and all horocyclic horizontal edges of rectangular branches have curvature less than $\ep$. 
\end{definition}

In fact, for every $\ep > 0$, we can always take an $\ep$-nearly straight neighborhood of $L$ in $\sigma$ by appropriately taking a sufficiently small neighborhood of the lamination. (\cite{Minsky92}].)

\section{Grafting rays and Teichmüller rays}
Recall that the Teichmüler geodesic flow is ergodic in the moduli space $\MM$ of Riemann surface structures on $S$ (\cite{Masur82, Veech82}).  
Let $\PML$ denote the space of projective measured laminations on $S$.
Pick an arbitrary generic Teichmüller geodesic $X(t),\, t \in \R$, such that 
\begin{itemize}
\item the projective vertical foliation $[V] \in \PML$ and the projective horizontal foliation $[H] \in \PML$ are both uniquely ergodic, and they have no saddle connections;
\item its corresponding quadratic differential has only simple zeros; 
\item the projection of the Teichmüller ray $X(t), t \leq 0$ in the negative direction is dense in the unit-tangent space of the moduli space $\MM$.
\end{itemize}

For each $t \in \R$, 
$(X(t), [V])$ conformally equivalent to 
 a unique marked flat surface $E_t$ of unite area with the vertical foliation $[V]$ (\cite{HubbardMasur79}). 
 Let $V_t$ be the representative of the projective foliation $[V]$ on $E_t$, so that $V_t$ has length one on $E_t$.
 By the density assumption, 
pick the decreasing sequence  $0 > t_1 > t_2 > \dots$ diverging to $- \infty$ such that the unmarked tangent vector $[X'(t_i)]$ in the moduli space converges to $[X'_\infty(0)] \in T^1 \MM$ as $i \to \infty$, where $X_\infty(t), $ $t \in \R,$ is an appropriate Teichmüller geodesic in $\TT$. 
Then, for each $i$, there is a mapping class diffeomorphism $\nu_i\col S \to S$ such that $\nu_i X(t_i)$ converges to $X_\infty(0) \eqqcolon X_\infty$ as $i \to \infty$.

Then $\nu_i E(t_i)$ converges to a flat surface $E_\infty$ with unite area, and $\nu_i V_{t_i}$ converges to a vertical measured foliation $V_\infty$ on $E_\infty$.
Then $V_\infty$ has length one on the flat surface $E_\infty$, and it is the vertical foliation of the Teichmüller geodesic $X_\infty(t)$.

By the density assumption, without loss of generality, we can in addition assume that
\begin{itemize}
\item the vertical $V_\infty$ is  uniquely ergodic and has no saddle connections, and
\item every singular point is three-pronged.
\end{itemize}
 
For each $i = 1, 2, \dots$, let $\sigma_i$ be the (marked) hyperbolic structure on $S$ uniformizing the Riemann surface $X(t_i)$.
Let $L_i$ be the measured geodesic lamination on $\sigma_i$ representing the vertical measured foliation $V_{t_i}$. 
For $u \geq 0$, let $\gr_{L_i}^u (\sigma_i)$ denote the conformal grafting of $\sigma_i$ along $u L_i$. 
Then $\gr_{L_i}^u (\sigma_i),  u \geq 0$ is called the {\it conformal grafting ray }starting from $\sigma_i$ along the vertical foliation $L_i$.

In this section, we prove the following uniform (strong) asymptotic property of grafting rays and Teichmüller rays from $X(t_i) \eqqcolon X_i$ as parametrized rays. 
\begin{theorem}\Label{SplitingLimitRay}
For every $\ep > 0$, there are constant $I_\ep > 0$,  $d > 0$ and $s_\ep > 0$ such that, if $i > I_\ep$, then 
$$d_\TT (X(t_i + s), \gr^{\exp(s)}_{L_i /d} (\sigma_i) ) < \ep$$
for all $s > s_\ep$, where $d_\TT$ denotes the Teichmüller distance.

\end{theorem}
The constant $d$ will be explicitly given in \S\ref{sAsymptoticPropertyInLimit}.
It is already known that, for each $i$,  the grafting ray $\gr^{\exp(s)}_{k L_i} (X(t_i))$ is asymptotic to the Teichmüller ray $X(t_i + s)$ as unparametrized rays:
Indeed, Gupta \cite{Gupta14,Gupta15} proved that, for every grafting ray along an arbitrary geodesic lamination $L$, there is a Teichmüller ray in the direction of $L$ typically from a different basepoint which is asymptotic to it up to reparametrization.

Masur proved that, for an arbitrarily fixed recurrent uniquely-ergodic vertical measured foliation $V$,  all Teichmüller rays with vertical foliation $V$ are asymptotic \cite[Theorem 2]{Masur80}.
Therefore, our main contribution of \Cref{SplitingLimitRay} is the asymptotic property as parametrized rays and the uniformness of the asymptotic property.

The strategy of \Cref{SplitingLimitRay} is overall based on the idea of the proof of Gupta. 
However, since we compare the grafting ray with a Teichmüller ray from the same base point, we utilize different and more geometric (than analytic) techniques especially in \S \ref{sHexagonalBranches}.
In particular,   we do not use any Grötzsch-type argument, whereas it is crucial in Gupta's argument (\cite{Gupta14}, Lemma 4.24).   
 \subsection{Fat traintrack structures and nearly-straight traintracks  in the limit}\Label{sAsymptoticPropertyInLimit}

We first show the asymptotic property of the single Teichmüller ray $X_\infty(t)$ in the limit and its corresponding grafting ray---
theuniform asymptotic property in \Cref{SplitingLimitRay} is morally modeled on this asymptotic property in the limit. 

Recall that $V_\infty$ is the vertical measured foliation of the limit flat surface $E_\infty$ of unit length. 
Let $H_\infty$ be the horizontal measured foliation of $E_\infty$ orthogonal to $V_\infty$.

Then, let $L_\infty$ be the measured geodesic lamination on $\sigma_\infty$ obtained by straightening to a measured foliation $V_\infty$. 

    Suppose that a flat surface $E$ has a vertical measured lamination $V$ so that the transversal measure is exactly given by the Euclidean length. 
  Then, the Euclidean length $\length_E V$ of $V$ is exactly the area of this flat surface.
 In particular, as ${\rm Area}\, E_\infty = 1$, $\length_{E_\infty} V_\infty = 1$, 
Then we let $$d = \frac{\length_{E_\infty} V_\infty}{\length_{\sigma_\infty} L_\infty}  = \frac{1}{\length_{\sigma_\infty} L_\infty},$$
where $\length_{\sigma_\infty} L_\infty$ denotes the hyperbolic length of $L_\infty$ (\cite{Thurston86HyperbolicStructuresII}).
We first construct a sequence of fat traintracks by splitting $E_\infty$  along vertical singular leaves.
Let $N = 2 (2 g- 2)$, the number of singular points on $E_\infty$.
Let $r_1, \dots r_N$ be vertical neighborhoods of singular points of $E_\infty$; since $V_\infty$ has no saddle connections, each $\gam_h$ is a tripod contained in a singular leaf.
Then, the complement $E_\infty \minus (r_1 \cup \dots \cup r_N)$ has a fat traintrack structure $T_0$, so that the branches are all Euclidean rectangles with horizontal and vertical edges---namely, we decompose  $E_\infty \minus (r_1 \cup \dots \cup r_N)$ by the horizontal line segments starting from the endpoints of $\gam_1, \dots, \gam_k$ and ending when the segments hit the boundary of $E_\infty \minus (r_1 \cup \dots \cup r_N)$. 
Clearly $T_0$ is foliated by $V_\infty$. 

Let $\sigma_\infty$ be the (marked) hyperbolic structure on $S$ uniformizing $X_\infty$.
Since $L_\infty$ is the geodesic representative of $V_\infty$, there is a traintrack neighborhood $\tau_0$ of $L_\infty$ on $\sigma_\infty$, such that there is a marking preserving diffeomorphism $\sigma_\infty \to E_\infty$ which induces an {\it isomorphism} from $(\tau_\infty, L_\infty)$  to $(T_0, V_\infty)$ as fat-traintracks carrrying .
 
By enlarging  $r_1, \dots r_N$,  we can construct a sequence $T_0, T_1, \dots$ of splitting of the traintrack $T_0$ so that lengths of all branches of $T_j$ diverge to infinity as $j \to \infty$. 
For $j = 1,2, \dots$, we let $r_1(j), \dots r_N(j)$ be this increasing sequence of vertical neighborhoods of singular points of $E_\infty$, such that the support $| T_j | $ of the traintrack $T_j$ is $E_\infty \minus (r_1(j) \cup \dots \cup r_N(j))$.
Clearly $r_h(j)$ is a tripod and the lengths of all three smooth edges go to infinity as $j \to \infty$ for all $h = 1, \dots, N$.

The vertical foliation $V_\infty$ and the horizontal foliation $H_\infty$ induce horizontal and vertical foliations of each $T_j$.
 By collapsing each horizontal leaf of $T_j$ to a point, we obtain a traintrack graph $G_j$.
By modifying the sequence of splittings if necessary,  we may, in addition, assume that $G_j$ is trivalent.
(In other words, there are no horizontal segments in $T_j$ connecting the endpoints of different tripods.)

We next construct the corresponding splitting sequence of traintrack neighborhoods of the measured lamination $L_\infty$ on the limit hyperbolic surface $\sigma_\infty$.
Since $L_\infty$ corresponds to a uniquely ergodic foliation $V_\infty$ without saddle connections, such that the underlying geodesic lamination $|L_\infty|$ of $L_\infty$ is maximal.
   
More generally, let $L$ be a maximal geodesic lamination on a hyperbolic surface $\sigma$.
Then, the complement $\sigma \minus L$ consists of hyperbolic ideal triangles $\Delta$, and each ideal triangle has a canonical horocyclic lamination $\lam_\Delta$ (\Cref{fHorocyclicLaminationAndAlmostHorocyclicFoliation}, left) invariant under the isometrics of $\Delta$: 
namely, leaves are horocyclic arcs centered at the vertices of the ideal triangle, and the complement of the lamination is a single triangle with the horocyclic edges. 
Then, the horocyclic arcs are orthogonal to the edges of the ideal triangle. 
Therefore, those horocyclic laminations on the complementary ideal triangles yield a {\sf horocyclic lamination} $\lambda$ on the hyperbolic surface $\sigma$ orthogonal to $L$. (See \cite{Thurston-98}.)
The horocyclic lamination has a transversal measure given by the hyperbolic length in the direction orthogonal to $\lam$.

 The support $|\lam_\Delta|$  of the horocyclic lamination is the complement of the triangle with horocyclic edges; thus this support is also foliated by geodesic rays orthogonal to $\lam_\Delta$ which limit to a common ideal vertex. 
We can extend this geodesic foliation on the support to a singular foliation $\mu_\Delta$ in $\Delta$ such that 
\begin{itemize}
\item $\mu$ has exactly one singular leaf $t$, and it is a tripod connecting the center $c_{\delta}$ of $\Delta$ to the vertices of $\Delta$ by geodesic rays;
\item each 
component of $\Delta \minus t$ is foliated by one parameter family of smooth lines connecting a pair of adjacent vertices of $\Delta$, and those leaves smoothly converge to the geodesic edge of $\Delta$ connecting the vertices (\Cref{fHorocyclicLaminationAndAlmostHorocyclicFoliation}, middle);
\item this foliation is invariant under the isometries of $\Delta$.
\end{itemize}
We call this singular foliation $\mu_\Delta$ of $\Delta$ a {\sf mostly-straight foliation} of $\Delta$. 

 The complement $\sigma \minus L$ consists of $N$ ideal triangles. 
Thus, the mostly straight foliations $\mu_\Delta$ on the complementary ideal triangles $\Delta$ yield a {\sf mostly-straight (singular) foliation} $\mu$ on $\sigma$ w.r.t. $L$, where the singular points are the centers of the ideal triangles. 

Note that the horocyclic lamination $\lam_\Delta$ of an ideal triangle $\Delta$ has a singular point on each edge where two horocyclic arcs centered at different vertices meet tangentially. 
Pick a singular foliation $\lam_\Delta'$ on $\Delta$ such that
\begin{itemize}
\item $\lam_\Delta$ coincides with $\lam_\Delta'$ in a small neighborhood $U$ of the complementary triangle $\Delta \minus |\lam_\Delta|$ with horocyclic edges;
\item $\lam_\Delta'$ has only one singular leaf $t$ with the singular point $c_\Delta$, and it is a union of three geodesic segments from $c_\Delta$ orthogonal to edges of $\Delta$;
\item each connected component of $\Delta \minus t$ is foliated by smooth (topologically) parallel arcs $\alpha$ such that $\alpha$ connects the endpoints of a horocyclic arc of $\lam_\Delta$ and  $\alpha$ orthogonally intersects the boundary geodesics of $\Delta$;
\item $\lam_\Delta'$ is invariant under the isometrics of $\Delta$. 
 (\Cref{fHorocyclicLaminationAndAlmostHorocyclicFoliation}, right). 
  \end{itemize}
 We call this singular foliation  $\lam_\Delta'$, a {\sf mostly-horocyclic foliation} of the ideal triangle $\Delta$. 
Given a maximal geodesic lamination $L$ on a hyperbolic surface $\sigma$, by taking a union of the mostly-horocyclic foliations on complementary ideal triangles, we obtain a {\sf mostly-horocyclic (singular) foliation} $\lam$ on $\sigma$ w.r.t. $L$. (c.f \cite[\S 4]{Thurston-98}.)
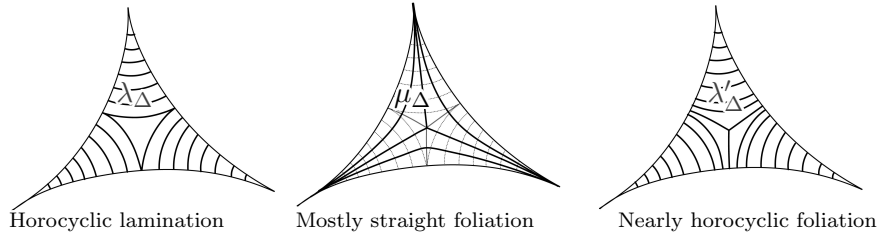
\begin{figure}
\begin{overpic}[scale=.06%, grid,tics=10
] {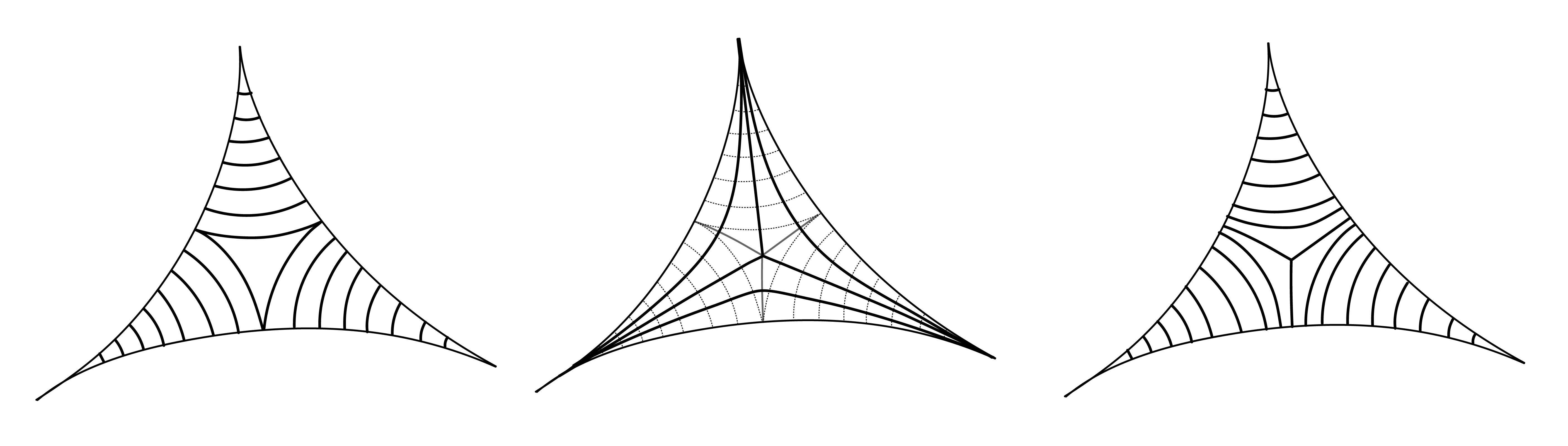} % figure file
 \put(70 ,1 ){\textcolor{black}{\tiny \contour{white}{Nearly horocyclic foliation}}}   
 \put(80 ,15 ){\textcolor{darkgray}{\small \contour{white}{$\lam'_\Delta$}}}   
 \put(2 ,1 ){\textcolor{black}{\tiny \contour{white}{Horocyclic lamination}}}   
 \put(14 ,15 ){\textcolor{darkgray}{\small \contour{white}{$\lam_\Delta$}}}   
  \put(34 ,1 ){\textcolor{black}{\tiny \contour{white}{Mostly straight foliation}}}   
 \put(45 ,15 ){\textcolor{black}{\small \contour{white}{$\mu_\Delta$}}}   
  % \put(, ){\textcolor{}{\tiny \contour{white}{$$}}}   
 %   \put( , ){\textcolor{}{$$}}  
 %   \put( , ){}  
%\put(, ){\color{}\vector(,){}} %{length}
      \end{overpic}
\caption{horocyclic lamination and nearly horocyclic foliation.}\Label{fHorocyclicLaminationAndAlmostHorocyclicFoliation}
\end{figure}

Let $\kap\col (\sigma_\infty, L_\infty) \to (E_\infty, V_\infty)$ be the marking preserving mapping such that 
 $\kap$ collapses, in each complementary triangle $\Delta$ of $(\sigma_\infty, L_\infty)$ to a ``Y-shaped'' graph with half-infinite edges ({\sf tripod}) by collapsing each mostly-horocyclic leaves  of $\lambda_\Del'$.
Then $\kap\col \sigma_\infty \to E_\infty$ takes $L_\infty$ to $V_\infty$ and is injective on each leaf of the maximal lamination $L_\infty$.

Therefore, by each singular tripod leaf of $V_\infty$ corresponds to a complementary ideal triangle of $L_\infty$, we can construct a sequence of train track neighborhoods $\tau_i$ of $L_\infty$ corresponding to the fat traintrack $T_i$, 
such that 
\begin{itemize}
\item $\tau_j$ is $\ep_j$-nearly straight and $\ep_j \searrow 0$ as $j \to \infty;$ 
\item  $(T_j, V_\infty)$ is isomorphic to $(\tau_j, \lambda_\infty)$ by $\kap$; 
\item  horizontal (short) edges of branches of $\tau_j$  are horocyclic (\Cref{fNeaerlyStraightTraintrackNeighborhood}). 
\end{itemize}
\begin{figure}
\begin{overpic}[scale=.1%, grid,tics=10
] {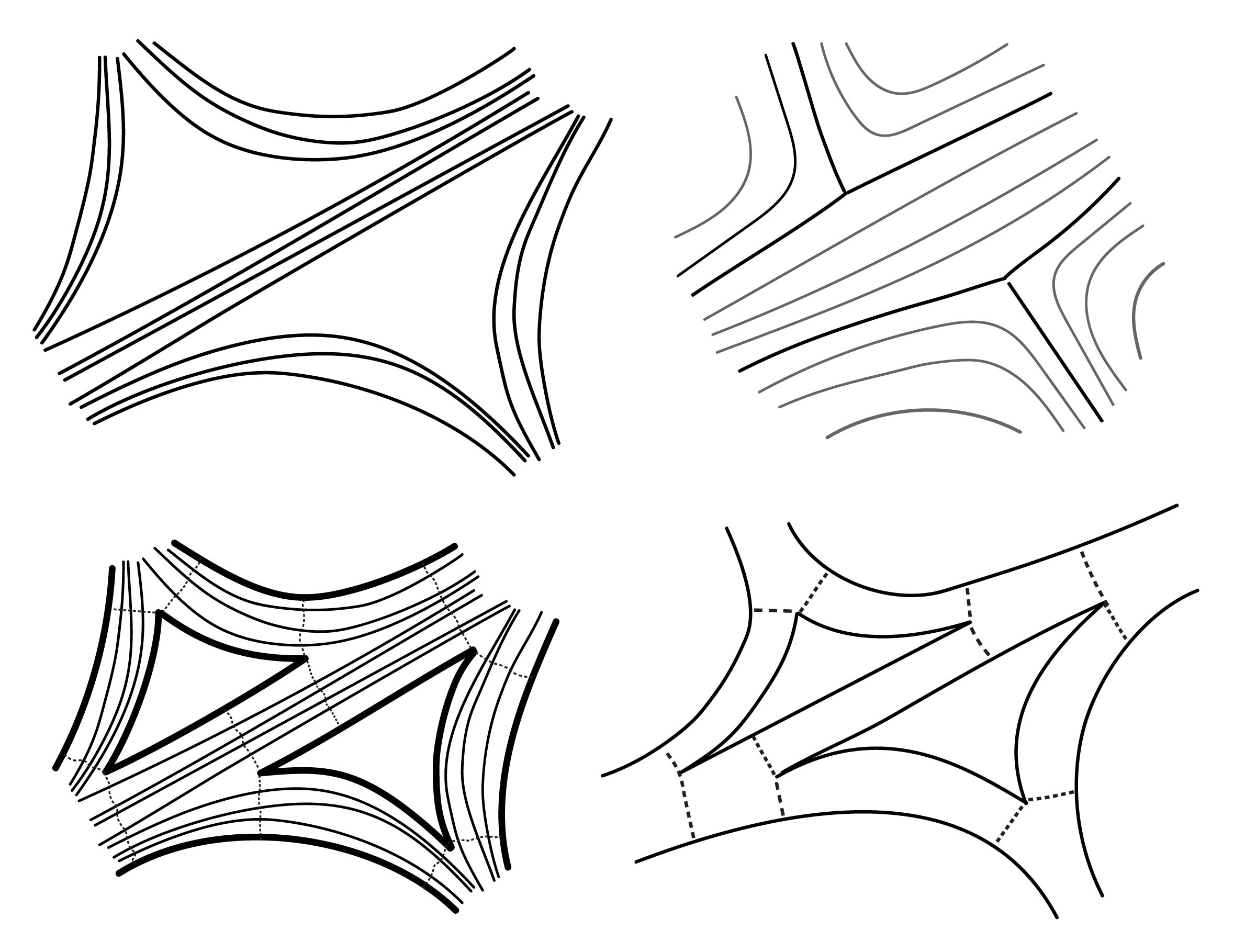} % figure file
 \put(25, 31){\textcolor{Black}{ \contour{white}{$\tau_j$}}}   
  \put(25, 70){\textcolor{Black}{ \contour{white}{$\sigma_\infty$}}}   
    \put(35, 65){\textcolor{Black}{ \contour{white}{$L_\infty$}}}   
    \put(70, 53){\textcolor{Black}{ \contour{white}{$V_\infty$}}}   
  \put(70, 70){\textcolor{Black}{ \contour{white}{$E_\infty$}}}   
      \put(70, 32){\textcolor{Black}{ \contour{white}{$T_j$}}}   
 %   \put( , ){\textcolor{}{$$}}  
 %   \put( , ){}  
      \end{overpic}
\caption{A nearly straight traintrack $\tau_j$ corresponding to a Euclidean fat traintrack $T_j$.}\Label{fNeaerlyStraightTraintrackNeighborhood}
\end{figure}
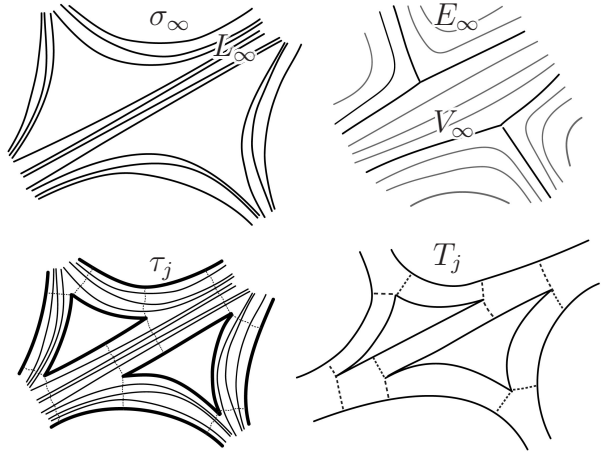
%\begin{proof}

\begin{lemma}\Label{AlmostIsomtricTraintrackDecompositions}
Let $\ep > 0$.
Then, if $j \in \Z_{> 0}$ is sufficiently large,  $$(d - \ep) \length \alpha_j < \length a_j < ( d  + \ep ) \length \alpha_j,$$
for all the vertical edges $a_j$ and $\alpha_j$ of the branches of $T_j$ and $\tau_j$ corresponding by the collapsing map $\kap$. 
\end{lemma}
\begin{proof}
Recall that $V_\infty$ is uniquely ergodic,  and the measured foliation $V_\infty$ gives positive measures to arcs transversal to $V_\infty$. 
Pick the corresponding branches $a_j$ and  $\alpha_j$ of $T_j$ and $\tau_j$ for each $j$.
 Then we can pick sequences of weights $w_j > 0, \omega_j > 0$, such that the sequence of weighted vertrical  segments
$(a_j, w_j)$ converges to $V_\infty$ and similarly
the sequence $(\alpha_j, \omega_j)$ of weighted leaf segments  converges to $L_\infty$ as $j \to \infty$ in weak* topology.

Note that the collapsing map $\kap\col (\sigma_\infty, L_\infty) \to (E_\infty, V_\infty)$ isomorphically takes $V_\infty$ to $L_\infty$ and the vertical edge $a_j$ to the vertical edge  $\alpha_j$ for all $j = 1, 2, \dots$. 
Therefore we can assume that $w_j = \omega_j$ for each $j = 1, 2, \dots$.

Therefore
$$w_j \length (a_j) \to \length_{E_\infty}(V_\infty) = 1$$
as $j \to \infty$.
$$w_j \length (\alpha_j) \to \length_{\sigma_\infty}(L_\infty),$$
as $j \to \infty.$

 Since  $$\frac{ \length_{E_\infty}(V_\infty)}{  \length_{\sigma_\infty}(L_\infty) } =  d,$$
the convergences of the weighted lengths above imply
 $$\frac{\length (a_j) }{\length (\alpha_j)} \to  d$$
 as $j \to \infty$. 
 Thus the lemma follows immediately. 
\end{proof}

\subsection{Stretching a traintrack along a Teichmüller ray and a grafting ray}
 Recall that the Teichmüller ray $X_\infty\col [0, \infty) \to \TT$ from $X_\infty$ has the vertical foliation $V_\infty$ and the horizontal foliation $H_\infty$, and the flat surface $E_\infty$ conformally realizes $X_\infty$ and geometrically realizes the vertical foliation $V_\infty$ and the horizontal foliation $H_\infty$.
 
 Let $E_\infty(s)$ be the marked flat structure on $S$ corresponding to $X_\infty(s)$ obtained by stretching $E_\infty$ only in the horizontal direction by $\exp(s)$; then $E_\infty(s)$ realizes the vertical measured foliation $\exp(s) V_\infty$, preserving the horizontal measured foliation $H_\infty$. 
Let  $f_{\infty, s}\col E_\infty = E_\infty(0) \to E_\infty(s)$ denote this linear stretching map in the horizontal direction. 
By  $f_{\infty, s}$ the traintrack structure $T_j$ of $E_\infty \minus (\gam_1(j) \cup \dots \cup \gam_N(j))$ descends to a traintrack structure $T_j(s)$ of $E_\infty(s) \minus  f_{\infty, s}(\gam_1(j) \cup \dots \cup \gam_N(j))$. 
 
Next we consider a corresponding grafting ray starting from the hyperbolic surface $\sigma_\infty$ representing $X_\infty$ along the geodesic representative $L_\infty$ of $V_\infty$. 
For $s > 0$, let $\Gr_{L_\infty}^s \sigma_\infty$ denote the projective structure on $S$ obtained by grafting the hyperbolic surface $\sigma_\infty$ along the (scaled) measured lamination $s L_\infty$.
Since $L_\infty$ has no periodic leaves, we let  $g_s\col \sigma_\infty \to \Gr_{L_\infty}^{s} \sigma_\infty$ be the canonical grafting $C^1$-diffeomorphism.
Then $s L_\infty$ is geometrically realized as a circular lamination on the projective surface $\Gr_{L_\infty}^s \sigma_\infty$. 
Namely, the grafting map $g_s$ takes the geodesic lamination $s L_\infty$ to the geometric realization. 
(See \cite{Kamishima-Tan-92, Baba20ThurstonParameter} for Thurston's parametrization of $\CP^1$-structures.)
 Then, by $g_t\col \sigma_\infty \to \Gr_{L_\infty}^s \sigma_\infty$,  the nearly-straight traintrack structure $\tau_j$ on $\sigma_\infty$ descends to a traintrack neighborhood $\tau_j(s)$ of the circular lamination $s L_\infty.$
    
Given a fat traintrack $T$, we call, by the {\sf one-skeleton},  the union of the horizontal and vertical edges of the rectangular branches, and denote it by $T^1.$
It has a structure of a finite graph: The vertices are the vertices of the rectangular branches of the traintrack, and the edges are segments connecting adjacent vertices. 
Note that an edge of a branch is not necessarily an edge of the one-skeleton, and it may be divided into multiple edges of the one-skeleton by some vertices.

       \begin{corollary}\Label{BilipschitzOneSkeleton}
   For every $\ep > 0$, there are $J_\ep > 0$ and $s_\ep > 0$  such that, if $j > J_\ep$ and $s > s_\ep$, then
   there is a  $(d - \ep, d + \ep)$-bilipschitz ``linear'' isomorphism between the one-skeletons 
    $$\phi_j^s\col \tau_j^1 (\exp(s)/ d) \to T^1_j (s)$$  for sufficiently large $s > 0$,  such that
    \begin{itemize}
    \item   $\phi_j^s$ is linear on each edge of $ \tau_j^1 (\exp(s)/ d)$ with respect to arc length  with respect to the Thurston metric on $\Gr_{L_\infty}^{\exp(s)/d}$ and the Euclidean metric on $E_\infty(s)$,
      and
    \item $\phi_j^s$ extends to a marking preserving homeomorphism $$\Gr_{L_\infty}^{\exp(s)/d} \sigma_\infty \to E_\infty(s).$$
\end{itemize}

   \end{corollary} 
 \begin{proof}
Let  $\phi_j^s\col \tau_j^1 (\exp(s)/ d) \to T^1_j (s)$ be the edge-wise linear mapping which continuously extends to a marking-preserving homeomorphism $\Gr_{L_\infty}^{\exp(s)/d} \sigma_\infty \to E_\infty(s)$. 

 We first consider the bilipschitz property in the vertical direction. 
  By  \Cref{AlmostIsomtricTraintrackDecompositions}, if $j$ is sufficiently large,  
corresponding vertical edges of $T_j$ and $\sigma_j$ are  $(d - \ep, d + \ep)$-bilipschiz.
Then the grafting map $\sigma_\infty \to \Gr_{L_\infty}^{\exp(s)/d} \sigma_\infty$  by $\exp(s)L$ preserves  the vertical length of branches of $\tau_j$, and 
the horizontal stretch map $f_{\infty, s}\col E_\infty \to E_\infty(s)$ preserves the vertical length of branches of $T_j$.
 Therefore, the edge-wise linear mapping $\phi_j^s\col \tau_j^1 (\exp(s)/ d) \to T^1_j (s)$ that is $(d+ \ep, d- \ep)$-bilipschitz in the vertical direction.

We next consider the bilipschitz property in the horizontal direction.
Let $e_j$ and $e_j'$ be corresponding horizontal edges of $T_j(s)$ and $\tau_j( \exp (s)/d)$, respectively. 
Since $\tau_j$ is a  $\ep_j$-nearly straight traintrack with $\ep_j \searrow 0$ as $j \to \infty$,
for every $\ep > 0$, there is $J_\ep > 0$ such that, if $j > J_\ep$,  every (horocyclic) horizontal edges of branches of $\tau_j$ has length less than $\ep$.

Let $e_j(s), e_j'(s)$ be the corresponding horizontal edges of $T_j(s)$ and $\sigma_j(\exp (s)/d)$, respectively.
Then $$\length_{E_\infty(s)} e_j(s) = \exp(s) V_\infty(e_j),$$ where $V_\infty(e_j)$ is the $V_\infty$-transversal measure of $e_j$ and $$\length_{\sigma_\infty^s} e_j'(s) = \exp(s) L_i(e_j')/d + \length_{\sigma_\infty} e_j',$$ where $\length_{\sigma_\infty^s}$ denote the length with respect to Thurston's metric on the projective surface $\Gr_{L_\infty}^{\exp s/d} \sigma_\infty$. 

Therefore, for every $\ep > 0$, there is $J_\ep > 0$, such that, if $j > J_\ep$, then since $V_\infty (e_j) = L_\infty(e_j')$ and $ \length_{\sigma_\infty} e_j' < c,$
we have
$$\Big\vert \frac{\length_{E_\infty^s} e_j(s)}{\length_{\sigma_\infty^s} e_j'(s)} - d \Big\vert < \ep,$$
for sufficiently large $s > 0$.
Therefore, we can make $\phi_j^s$ a $( d - \ep, d + \ep)$-bilipschitz mapping also on the horizontal edges. 
\end{proof}

\section{Construction of almost conformal mappings}\Label{sAlmostConformalMapping}      
      
  Recall that $X_\infty(s)$ is the Teichmüller geodesic ray from $X_\infty$ with the vertical measured foliation $V_\infty$ parametrized by $s \geq 0$. 
Let $\gr^s_L(\sigma_\infty)$ be the conformal grafting ray from the hyperbolic surface $\sigma_\infty$ along the measured geodesic lamination $L_\infty$, where $\sigma_\infty$ uniformizes the marked Riemann surface $X_\infty$ and the measured lamination $L_\infty$ corresponds to the measured foliation $V_\infty$. 

In this section, we prove the asymptotic property of those rays as parametrized rays without modifying their base points. 
\begin{theorem}\Label{Asymptotic}
For every $\ep > 0$, there is $s_\ep > 0$ such that 
$$d(X_\infty(s), \gr_ V^{\exp (s)/d} \sigma_\infty  ) < \ep$$
for all $s > s_\ep$. 
\end{theorem}
\begin{figure}
\begin{overpic}[scale=.15%, grid,tics=10
] {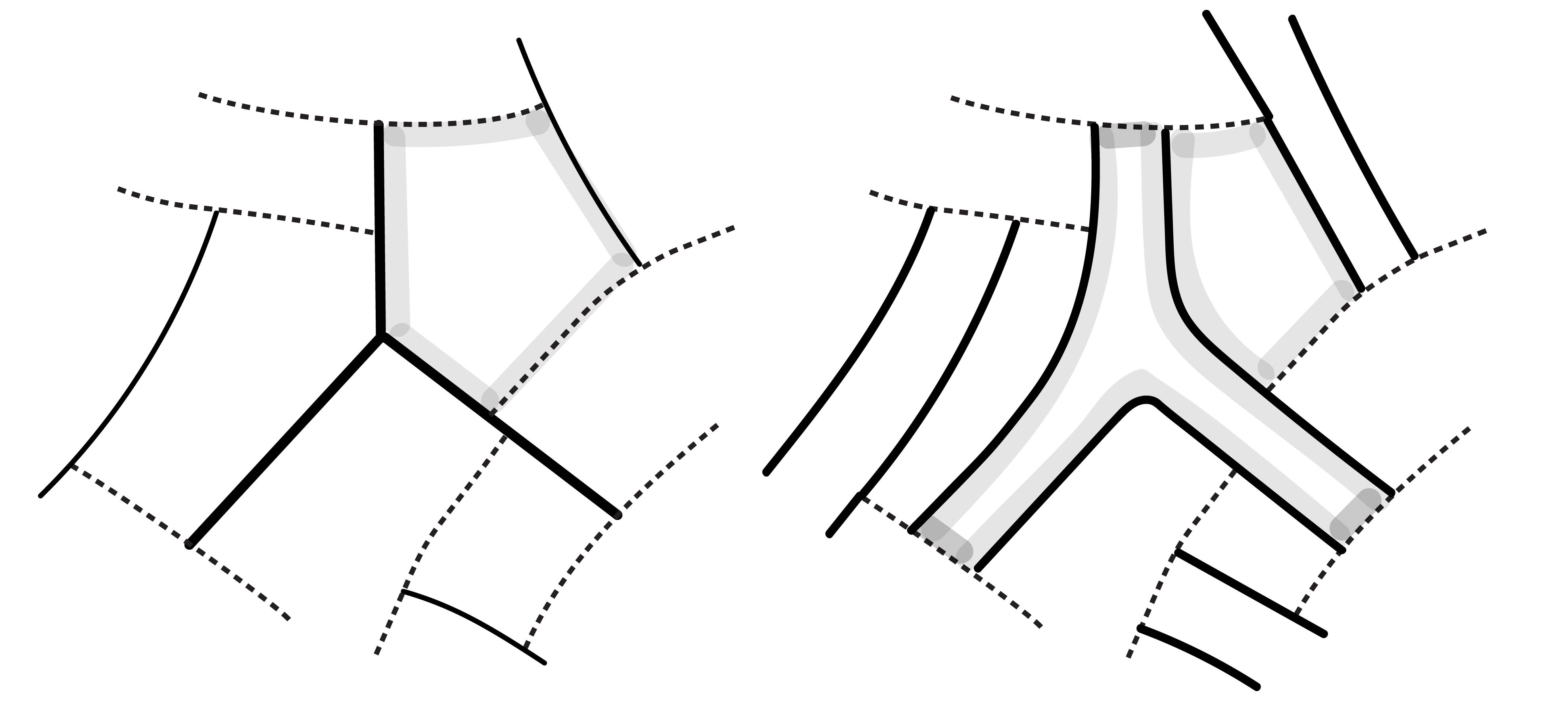} % figure file
        \put(5 , 36 ){\textcolor{Black}{
    \contour{white}{$T_j$}
    }}  

        \put(18 ,21.5 ){\textcolor{Black}{
    \contour{white}{$\gamma_h$}
    }}  
        \put(29 ,26 ){\textcolor{Black}{
    \contour{white}{$R$}
    }}  
   
        \put(70 , 21.5 ){\textcolor{Black}{\small
    \contour{white}{$Q_{j, h}$}
       }}  
       \put(77 , 27 ){\small \textcolor{Black}{
    \contour{white}{$R_{j, k}$}
       }}  
           \put(55 , 40 ){\small \textcolor{Black}{
    \contour{white}{$E_{\infty, j}$}
       }}  
 %   \put( , ){}  
      \end{overpic}
\caption{
The traintrack structure $T_j$ induces a  polygonal traintrack decomposition 
$ E_{\infty, j} = (\cup_{h = 1}^N Q_{j, h})  \cup (\cup_{k = 1}^{N'} R_{j, k})$ (right). 
Solid segments are vertical edges and dotted segments are horizontal edges of the decompositions.}\Label{sEuclideanRectangleHexsagonDecomposition}
\end{figure}
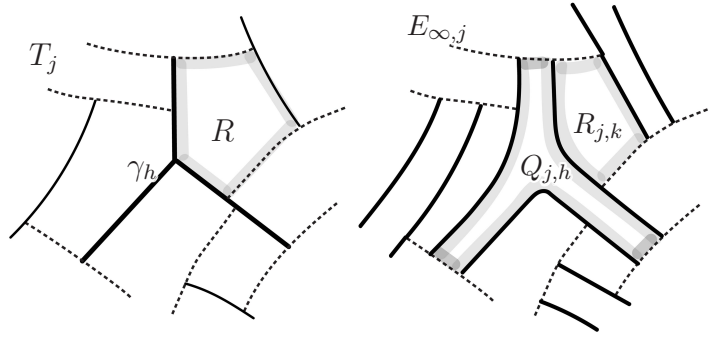

We first construct a decomposition of $E_\infty$ into rectangles and hexagons from the traintrack structure $T_j$ of $E_\infty \minus (\gam_1(j) \cup \dots \cup \gam_N(j))$. 

Given $\ep > 0$ and a subset $A$ of a flat surface $E$ with a horizontal foliation $H$, the {\sf $\ep$-horizontal neighborhood} of $A$ is the subset of $E$ consisting of points $p$ can be connected to $A$ by a segment of a leaf of $H$ with length at most $\ep$. 

Recall that the branches of $T_j$ are Euclidean rectangles; 
let $m_j$ be the shortest horizontal length of the branches of $T_j$.
For each $h =1, \dots, n$,
 let $Q_{j, h}$ be the horizontal $m_j/3$-neighborhood of the $Y$-shaped graph $\gam_h$.
 Then $Q_{j, h}$ is a hexagon with horizontal and vertical edges and one singular point in the middle, and its horizontal edges have length $2m_j/3$; see \Cref{sEuclideanRectangleHexsagonDecomposition}. 
Therefore, by the definition of $m_j$,  the hexagons $Q_{j, 1}, Q_{j, 2}, \dots, Q_{j, N}$ are pairwise disjoint.
In addition, for each branch $R$ of $T_j$,  the Euclidean rectangle $R$ minus the $m_j/3$-horizontal neighborhood of the vertical edges is still a rectangle with horizontal and vertical edges. 
Therefore, the traintrack structure $T_j$ of  $E_\infty \minus (\gam_1(j) \cup \dots \cup \gam_N(j))$ gives a rectangle decomposition of $E_\infty \minus (Q_{j, 1} \cup \dots \cup Q_{j, N}) = \cup_{k = 1}^{N'} R_{j, k}$, and   so that each rectangle piece $R_{j, k}$ is a branch of $T_j$ minus the $m_j/3$-horizontal-neighborhood of its vertical edges (\Cref{sEuclideanRectangleHexsagonDecomposition}). 
Thus we have a decomposition of the flat surface $E_\infty$ into hexagons and rectangles with vertical and horizontal edges,  
$$ E_{\infty, j} = (\cup_{h = 1}^N Q_{j, h})  \cup (\cup_{k = 1}^{N'} R_{j, k}).$$
Then this polygonal traintrack $E_{\infty, j}$ {\it carries} the vertical foliation $V_\infty$, i.e. the intersection of $V_\infty$ with each hexagon and rectangle is a union of disjoint arcs connecting different horizontal edges.

By the horizontal stretch map $f_{\infty, s}\col E_\infty \to E_\infty(s)$, this polygonal decomposition $E_{\infty, j}$ induces a corresponding polygonal decomposition of $ E_\infty(s)$
$$E_{\infty, j}(s) = (\cup_{k = 1}^N Q_{j, h}^s)  \cup (\cup_{h = 1}^{N'} R_{j, k}^s),$$
where $f_{\infty, s}(Q_{j, h}) = Q_{j, h}^s$ and  $f_{\infty, s}(R_{j, k}) = R_{j, k}^s$.

Recall that we constructed an $\ep_j$-nearly straight traintrack neighborhood $\tau_j$ of $L_\infty$ on $\sigma_\infty$, where $\ep_j \searrow 0$ as $j \to \infty$.
This decomposition $\tau_j$ is induced by the traintrack decomposition $T_j$ of $E_\infty$ so that $\tau_j$ descends to $T_j$ by the (picked) collapsing map $\kap\col (\sigma_\infty, L_\infty) \to (E_\infty, V_\infty)$.

Similarly, for each $j = 1, 2, \dots,$ the polygonal decomposition $E_{\infty, j} = (\cup_{h = 1}^N Q_{j, h})  \cup (\cup_{k = 1}^{N'} R_{j, k})$ induces   a polygonal decomposition 
$\sigma_{\infty, j} = (\cup_{h = 1}^N \QQ_{j, h})  \cup (\cup_{k = 1}^{N'} \RR_{j, k})$ such that
\begin{itemize}
\item  $\sigma_{\infty, j} = (\cup_{k = 1}^N \QQ_{j, h})  \cup (\cup_{h = 1}^{N'} \RR_{j, k})$ is isomorphic to $E_{\infty, j} = (\cup_{h = 1}^N Q_{j, h})  \cup (\cup_{k = 1}^{N'} R_{j, k})$ as polygonal traintrack carrying $L_\infty \cong V_\infty$; 
\item this isomorphism is realized by the collapsing map $\kap$;
\item the vertical edges of the $\QQ_{j, h}$ and $\RR_{j, k}$ are segments of leaves of $L_\infty$, and their horizontal edges are segments of the horocyclic foliation $\lambda_\infty$ of $(\sigma_\infty, L_\infty)$.
\end{itemize}

Recall that the traintrack neighborhood $\tau_j$ of $L_\infty$ on $\sigma_\infty$ is transformed into a traintrack neighborhood $\tau_j(s)$ of the Thurston lamination $s L_\infty$ on $\Gr_{L_\infty}
^s \sigma_\infty$.
Similarly, the polygonal decomposition $\sigma_{\infty, j}$ induces a decomposition $\sigma_{\infty, j}(s)$ after grafting:
  $$\Gr_{L_\infty}^{\exp(s)/d} \sigma_\infty  =  ( \cup_{h = 1}^N \QQ_{j, h}^s) \cup (\cup_{k = 1}^{N'} \RR_{j, k}^s ),$$
  
where $\QQ_{j, h}^s$ is obtained by grafting $\QQ_{j, h}$ along the restriction of $\frac{\exp{s}}{d} L_\infty$ to $\QQ_{j, h}$ and 
$\RR_{j, k}^s$ is obtained by grafting $\RR_{j, k}$ along the restriction of  $\frac{\exp{s}}{d}  L_\infty$ to $\RR_{j, k}$. 

\begin{proposition}
For every $\ep > 0$, there are $J_\ep > 0$ and $s_\ep > 0$ such that, if $j > J_\ep$ and $s > s_\ep$, 
there is a $(d - \ep, d + \ep)$-bilipschitz map between the one-skeletons of the polygonal decompositions
$$\phi_j^s\col ( \cup_h \partial \QQ_{j, h}^s) \cup (\cup_k \partial \RR_{j, k}^s ) \to  (\cup_h \partial Q_{j, h}^s) \cup (\cup_k \partial R_{j, k}^s) $$
for sufficiently large $s > 0$, such that $\phi_j^s$ is linear on each edge. 
\end{proposition}
\begin{proof}
The proof is similar to that of \Cref{BilipschitzOneSkeleton}.
\end{proof}
  
 Our main goal of this section is to prove the following. 
      \begin{proposition}\Label{qfExtension}
    For every $\ep > 0$, there are $J_s > 0$ and $s_\ep > 0$ such that, if $j > J_s$ and $s > s_\ep$, then we can extend the above bilipschitz mapping between the one-skeletons $$\phi_j^s \col ( \cup_{h = 1}^N \partial \QQ_{j, h}^s) \cup (\cup_{k = 1}^{N'} \partial \RR_{j, k}^s ) \to (\cup_{h = 1}^N \partial Q_{j, h}^s) \cup (\cup_{k = 1}^{N'} \partial R_{j, k}^s)$$ to a $(1 + \ep)$-quasiconformal mapping $$\Phi_j^s\col \Gr_{L_\infty}^{\exp(s)/d} \sigma_\infty \to E_\infty(s)$$    
     taking the polygonal decomposition $\tau_j(s)$ to  the polygonal decomposition $E_{\infty, j}(s)$. 
\end{proposition}

In order to prove \Cref{qfExtension}, we construct a desired extension on both types of polygonal pieces in the following subsections.

  \subsubsection{Rectangles}\Label{sAlmostConformalRectangles} 
In this subsection, we extend $\phi_j^s$ to a quasiconformal mapping  $\RR_{j, k}^s \to R_{j, k}^s$ with small distortion for each rectangular branch $\RR_{j, k}^s$.   
  
\begin{lemma}\Label{HorocyclicLengthDerivative}
Consider the region $P$ in $\H^2$ bounded by two disjoint geodesics sharing an endpoint in the ideal boundary $\bdr \H^2$.
Then $P$ is foliated by a one-parameter family of horocyclic arcs $\{a_u\}$ centered at the common endpoint. 
We can parametrize the family of the horocyclic arc $a_u (u \in \R)$ by their distances, so that the difference of their parameters corresponds to the length between the arcs (\Cref{fVerticalStrip}). 
Then 
$$\frac{d}{du} \length_{\H^2}(a_u) = - \length_{\H^2} a_u$$
\end{lemma}  
\begin{proof}
         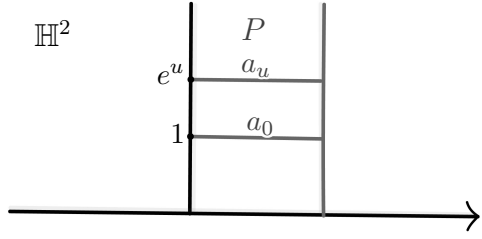
\begin{figure}
\begin{overpic}[scale=.15%, grid,tics=10
] {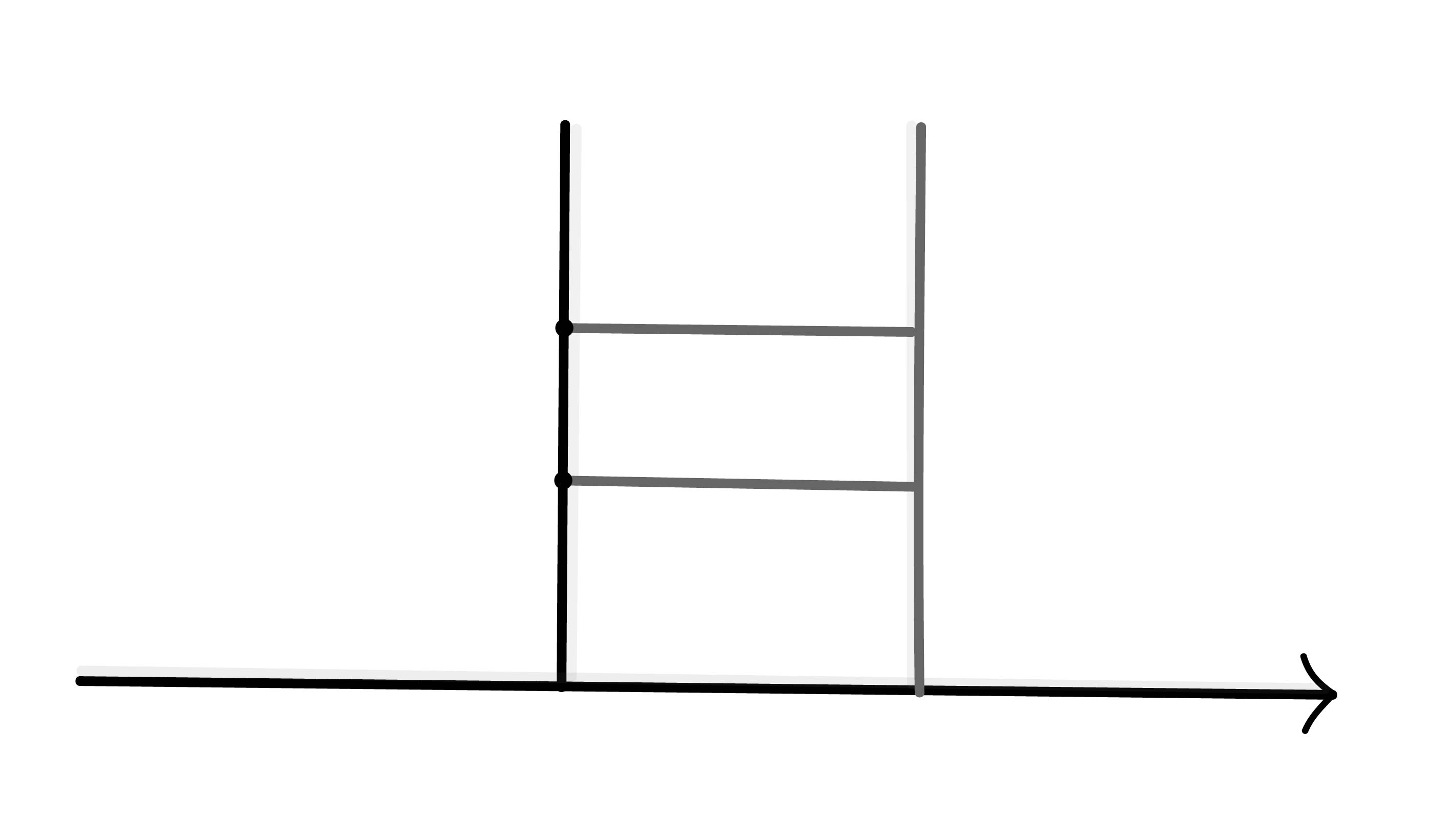} % figure file
 \put(35 , 22){\textcolor{black}{\small $1$}}  
 \put(32.5 , 33){\textcolor{black}{\small \contour{white}{$e^u$} }}  
  \put(49 , 24.4){\textcolor{darkgray}{\small \contour{white}{$a_0$} }}  
     \put(48 , 35.1){\textcolor{darkgray}{\small \contour{white}{$a_u$} }}  

 \put(10 , 40){\textcolor{black}{ \contour{white}{$\H^2$} }}  
  \put(48 , 41){\textcolor{darkgray}{ \contour{white}{$P$} }}  
 %   \put( , ){\textcolor{}{$$}}  
 %   \put( , ){}  
      \end{overpic}
\caption{horocyclic arcs in the region $P$.}\Label{fVerticalStrip}
\end{figure}
We first normalize the region $P$ in the upper half-plane model of $\H$ so that the common endpoint is at $\infty$.
It suffices to show the derivative formula at $u = 0$, and we can further normalize the region $P$ so that $a_0$ is the horizontal arc at height one; see Figure \ref{fVerticalStrip}. 
Let $\ell$ be the length of $a_0$; then, since $a_u$ is parametrized by the vertical (hyperbolic) distance,  the length of $a_u$ is $\frac{\ell}{e^u}$.
Therefore, we have 
 $$\frac{d}{du} \Big(\frac{\ell}{e^u}\Big)  = - \ell e^{-u}.$$
 Thus   
 $$\frac{d}{du} \Big(\frac{\ell}{e^u}\Big) \Big\vert_{u = 0}=  -\ell.$$
 \end{proof}

Pick a rectangular piece $\RR_{j, k}^s$  of the polygonal decomosition $$\Gr_{L_\infty}^{\exp s/d}\sigma_\infty= ( \cup_{h = 1}^N \partial \QQ_{j,h}^s) \cup (\cup_{k = 1}^{N'} \partial \RR_{j, k}^s ).$$ 
  Then $\RR_{j, k}^s$ is foliated by the leaf segments of horocyclic foliation $\lambda_\infty$ of $(\sigma_\infty, \lambda_\infty)$.
  Therefore, the vertical edges of $\RR_{j, k}^s$ are geodesic segments of the same length;
   let $\ell ( = \ell_{j, k}^s)$ denote this vertical length of $\RR_{j, k}^s$.

Consider the branch $\RR_{j, k}$ of the polygonal decomposition of $\sigma_\infty$ which, after grafting, corresponds to $\RR_{j, k}^s$. 
Then $\RR_{j, k}$ is foliated by the horocyclic segments of the (horizontal) horocyclic lamination $\lam_\infty$, since the non-foliated parts are contained in hexagonal branches.
Let $\lam_{j, k}^s$ denote this horocyclic foliation of $\RR_{j, k}$

The measured geodesic lamination $L_\infty$ is orthogonal to the horocyclic lamination $\lam_\infty$. 
Then, the restriction of $L_\infty$ to $\RR_{j, k}$ extends to a (vertical) geodesic foliation $\mu = \mu_{j, k}$ in $\RR_{j, k}$ orthogonal to the horocyclic foliation.
Note that the lengths of the leaves of the foliation $\mu_{j, k}$ are the same, since there are isometries between the leaves given by the translation along the homocyclic foliation $\mu_{j, k}$. 

As $\RR_{j, k}^s$ is obtained by grafting $\RR_{j, k}$ along $s L_\infty$, the horocyclic foliation $\lam_{j, k}$ induces a horocyclic foliation $\lam_{j, k}^s$ on $\RR_{j, k}^s$, so that the collapsing map $\kap_s \col \Gr_{L_\infty}^s \sigma_\infty \to \sigma_\infty$ takes leaves of  $\lam_{j, k}^s$ to leaves of $\lam_{j, k}$.
Similarly, the vertical geodesic foliation $\mu_{j, k}$ induces the vertical geodesic foliation $\mu_{j, k}^s$ on $\RR_{j, k}^s$, so taht $\kap_s$ takes $\mu_{j, k}^s$ to $\mu_{j, k}$. 
   
\begin{lemma}\Label{UniformizaingGraftedRectangle}
For every $\ep > 0$, there is $J_\ep > 0$, such that, if $j > J_\ep$, then,  for every sufficiently large $s > 0$, every rectangular branch $\RR_{j, k}^s$ of the polygonal decomosition $\Gr_{L_\infty}^{\exp s/d}\sigma_\infty= ( \cup_{h = 1}^N \partial \QQ_{j,h}^s) \cup (\cup_{k = 1}^{N'} \partial \RR_{j, k}^s )$ is a $C^1$-smooth $(1 - \ep, 1 + \ep)$-bilipschitz mapping from $\RR_{j, k}^s$ to an Euclidean rectangle of the same length $\ell = \ell_{j, k}^s$ and the width $\exp(s) L(\RR_{j, k}^s)$, where  $L(\RR_{j, k}^s)$ denote the transversal measure of the horizontal edge of $\RR_{j, k}^s$ given by $L$. 
\end{lemma}
\begin{proof}

 Let $F = F_{j, k}^s$ be the Euclidean rectangle of length $\ell_{j, k}^s$ and width  $\exp(s) L(\RR_{j, k}^s)$.
We construct an almost isometric mapping $\zeta_{j, k}^s\col \RR_{j, k}^s \to F_{j, k}^s$ preserving horizontal leaves.

 Pick a horizontal (horocyclic) edge $e_h$ of $\RR_{j, k}^s$, and a vertical (geodesic) edge $e_v$ of $\RR_{j, k}^s$.
 Let $z$ be a point in $\RR_{j, k}^s$.
 Then $z$ is contained in a leaf $u_z$ of the horizontal horocyclic foliation $\lam_{j, k}^s$, and a leaf $w_z$ of the vertical geodesic foliation $\mu_{j, k}^s$.
 Let $y$ be the length of the geodesic segment of $w_z$ from $z$ to $e_h$ (along $w$).
 Let $x$ be the length of the segment of $u_z$ connecting $z$ to a point in $e_v$. 
 Then, we define a mapping $\zeta_{j, k}^s\col \RR_{j, k}^s \to F_{j, k}^s$ by $$z =(x, y) \mapsto (L(\RR_{j, k}^s) \frac{y}{\length~ u_z}, y),$$ 
 so that it is linear along a horocyclic leaf $u_z$ with respect to arc length (\Cref{sEulideanizingMap}).

  \begin{figure}
\begin{overpic}[scale=.035%, grid,tics=10
] {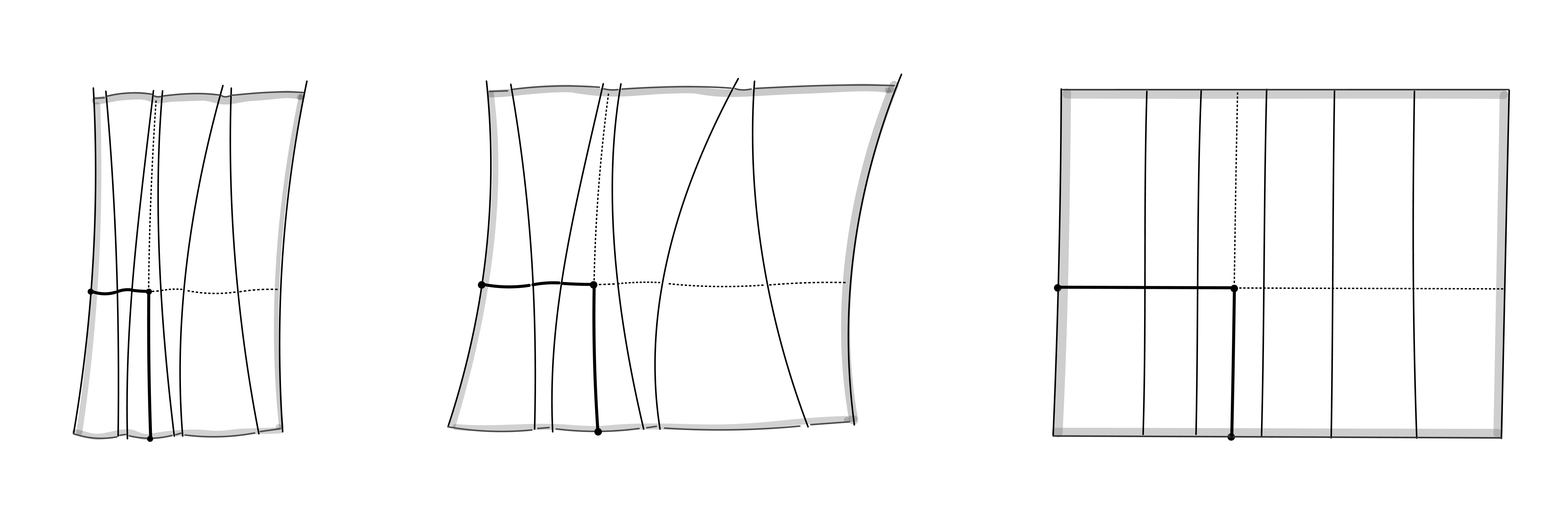} % figure file
  \put(3 ,20 ){\textcolor{darkgray}{\small \contour{white}{$e_v$}}}  
  \put(13 ,3 ){\textcolor{darkgray}{\small \contour{white}{$e_h$}}}  
  
  \put(12 ,23 ){\textcolor{darkgray}{\small \contour{white}{$R_{j, k}^0$}}}

   \put(19, 15){\color{black}\vector(1,0){10}}
  \put(20 , 16.7 ){\textcolor{black}{\small \contour{white}{$\Gr_{s L_\infty}$}}}   
    
  \put(32 , 24 ){\textcolor{darkgray}{\small \contour{white}{$R_{j, k}^s$}}}  
   \put(34 ,15.5 ){\textcolor{black}{\small \contour{white}{$x$}}}
         \put(38.5 , 10 ){\textcolor{black}{\small \contour{white}{$y$}}}

   \put(57, 15){\color{black}\vector(1,0){8}}
  \put(59 , 17 ){\textcolor{black}{\small \contour{white}{$\zeta_{j, k}^s$}}}   

   \put(69 , 24 ){\textcolor{darkgray}{\small \contour{white}{$F$}}}      
  \put(79.5 , 10 ){\textcolor{black}{\small \contour{white}{$y$}}}   
  \put(73 , 15.5 ){\textcolor{black}{\small \contour{white}{$x$}}}
      \put(85 , 15.5 ){\textcolor{darkgray}{\small \contour{white}{$u_z$}}}

 %   \put( , ){\textcolor{}{$$}}  
 %   \put( , ){}  
      \end{overpic}
\caption{Mapping hyperbolic rectangle into a Euclidean rectangle.}\label{sEulideanizingMap}
\end{figure}
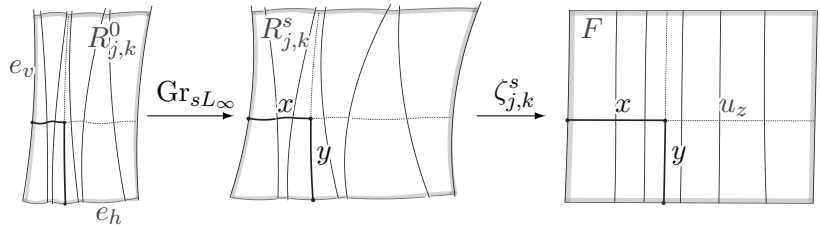
 Next we show that $ \zeta_{j, k}^s$ is an almost conformal mapping for sufficiently large $j, s > 0$.
     .
Each horocyclic leaf $u$ in $R_{j, k}= R_{j, k}^0$ intersects the measured lamination $L_\infty$ in a (Lebesgue) measure zero set. 
As $L_\infty$ is a maximal lamination, 
we set $$u \setminus L_\infty = \cup_{r = 1}^\infty u_r,$$ where $u_r$ are the connected segments of $u \setminus L_\infty$.
  Since $L \cap u$ has measure zero in $u$,  $$\length_{\sigma_\infty} u = \Sigma_{r = 1}^\infty \length_{\sigma_\infty} u_r .$$ 
  
 We parametrize the horocyclic leaves $h_x$ of $\lam_{j, k}^0$ with $x \in [0, \ell_{j, k}^s]$  by the length from the horizontal ledge $e_h$ along vertical geodesic leaves of $\mu_{j, k}^s$.
 For every $\ep > 0$, there is $J_\ep > 0$, such that if $j > J_\ep$, then  $\length~ h < \ep$ for all horocyclic leaves $u$ of $\lam_{j, k}^s$.
  Then, by \Cref{HorocyclicLengthDerivative},
  \begin{eqnarray}
  \Label{iDerivativeBound}
  \bigg\vert \frac{d (\length\, u_x)}{dy }\bigg\vert  &=& \bigg\vert   \frac{d}{d y} (\sum_{r = 1}^\infty \length_{\H^2} u_r)  \bigg\vert \\
  \notag
  &\leq& \sum_{r = 1}^\infty  \bigg\vert \frac{d}{dy}( \length_{\H^2} u_r)  \bigg\vert < \ep.
\end{eqnarray}

The grafting of $\sigma_\infty$ along the measured lamination $L_\infty$ inserts Euclidean structure along $L_\infty$, and the length of all horocyclic leaves of $R_{j, k}$ equally increases by the constant $\exp(s) L(\RR_{j, k}^s)$ w.r.t. Thurston's metric.
Clearly $\exp(s) L(\RR_{j, k}) \to \infty$ as $s \to \infty$.
   Therefore,  for every $\ep > 0$, there are $J_\ep > 0$ and $s_\ep > 0$, such that if $j > J_\ep$ and $s > s_\ep$, then the horizontal derivative  $\frac{d \zeta_{j, k}^s}{d x} (z)$ is  the vector $(t,  0)$ with $t \in (1 - \ep, 1 + \ep)$ for all $z \in R_{j, k}^s$.    
   
Since $\zeta_{j, k}^s$ preserves the height by its definition, a similar argument shows that there are $J_\ep > 0$ and $s_\ep > 0$, such that if $j > J_\ep$ and $s > s_\ep$, then 
$\frac{d \zeta_{j, k}^s}{d x} (z)$ is  $(t, 1)$ with $t \in (-\ep, \ep)$ for all $z \in \RR_{j, k}^s$ by (\ref{iDerivativeBound}).

We have shown that $d \zeta_{j, k}^s$ almost preserves the orthonormal frames.
Therefore,  for every $\ep > 0$, there are $J_\ep > 0$ and $s_\ep > 0$ such that if $j > J_\ep$ and $s > s_\ep$ such that $\zeta_{j, k}^s$ is $(1 - \ep, 1+ \ep)$-bilipschitz and $C^1$-smooth. 
\end{proof}

\begin{corollary}
 For every $\ep > 0$, there are $J_\ep > 0$ and $s_\ep > 0$, such that if $j > J_\ep$  and $s > s_\ep$, then
the edge-wise linear map $\phi_j^s$ on the one-skeleton extends continuously to a $(1 + \ep)$-quasiconformal mapping from $R_{j, s}^s$ to $\RR_{j, k}^s$.
\end{corollary}

\begin{proof}
Let $\xi_{j, k}^s\col F_{j, k}^s \to  R_{j, k}^s$ be the linear mapping between Euclidean rectangles which preserves horizontal and vertical edges.
Then, by \Cref{BilipschitzOneSkeleton} and the definition of $F_{j, k}^s$,  for every $\ep > 0$ there is $J_\ep > 0$ such that, if $j > J_\ep$, 
implies that the linear mapping $\xi_{j, k}^s\col F_{j, k}^s \to  R_{j, k}^s$ is a $(d - \ep, d + \ep)$-bilipschitz for sufficiently large $s > 0$. 
Therefore, we can in addition assume that the composition $\xi_{j, k}^s \circ \zeta_{j, k}^s \col \RR_{j, k}^s \to R_{j, k}^s$ is a $(1 + \ep)$-quasiconformal mapping. 

The restriction of $\phi_j^s$ to $\partial \RR_{j, k}^s$ is a piecewise-linear mapping which is linear on the vertical edges but not necessarily linear on the horizontal edges of $ \RR_{j, k}^s$. 
Since the fat traintracks correspond to trivalent graphs, a horizontal edge of $\RR_{j, k}^s$ may be decomposed into three linear pieces for $\phi_j^s$. 
For every $r > 0$, there are $J_r > 0$ and $s_r > 0$, such that,  then if $j > J_r$ and $s > s_r$, then 
 the vertical edge of $R_{j, k}^s$ has length at least $r$, and every linear segment of each horizontal edge also has length at least $r$. 
Therefore, we can easily adjust $\xi_{j, k}$ near the boundary of $F_{j, k}^s$ by a quasi-conformal mapping with small dilatation,  so that the composition $\xi_{j, k}^s \circ \zeta_{j, k}^s$ remains  $(1 + \ep)$-quasiconformal and its restriction to $\bdr \RR_{j, k}^s$ agrees with $\phi_j^s$. 
\end{proof}

\subsection{Extension to hexagonal branches}\Label{sHexagonalBranches}
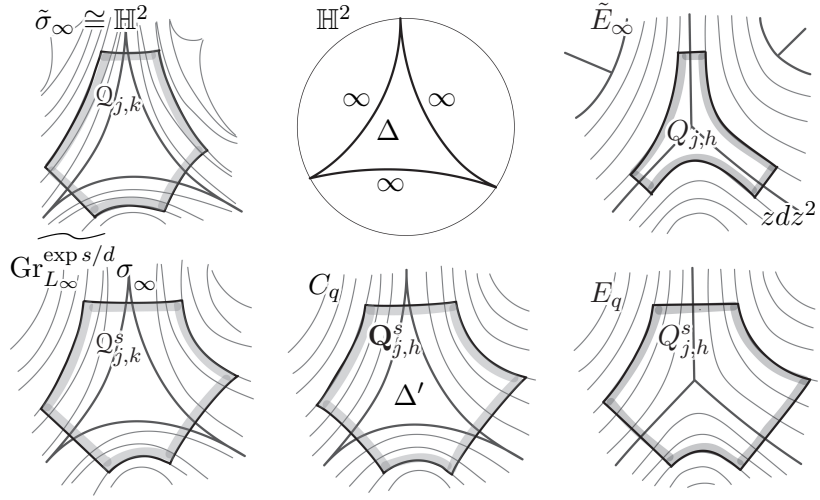
\begin{figure}
\begin{overpic}[scale=.04%, grid,tics=10
] {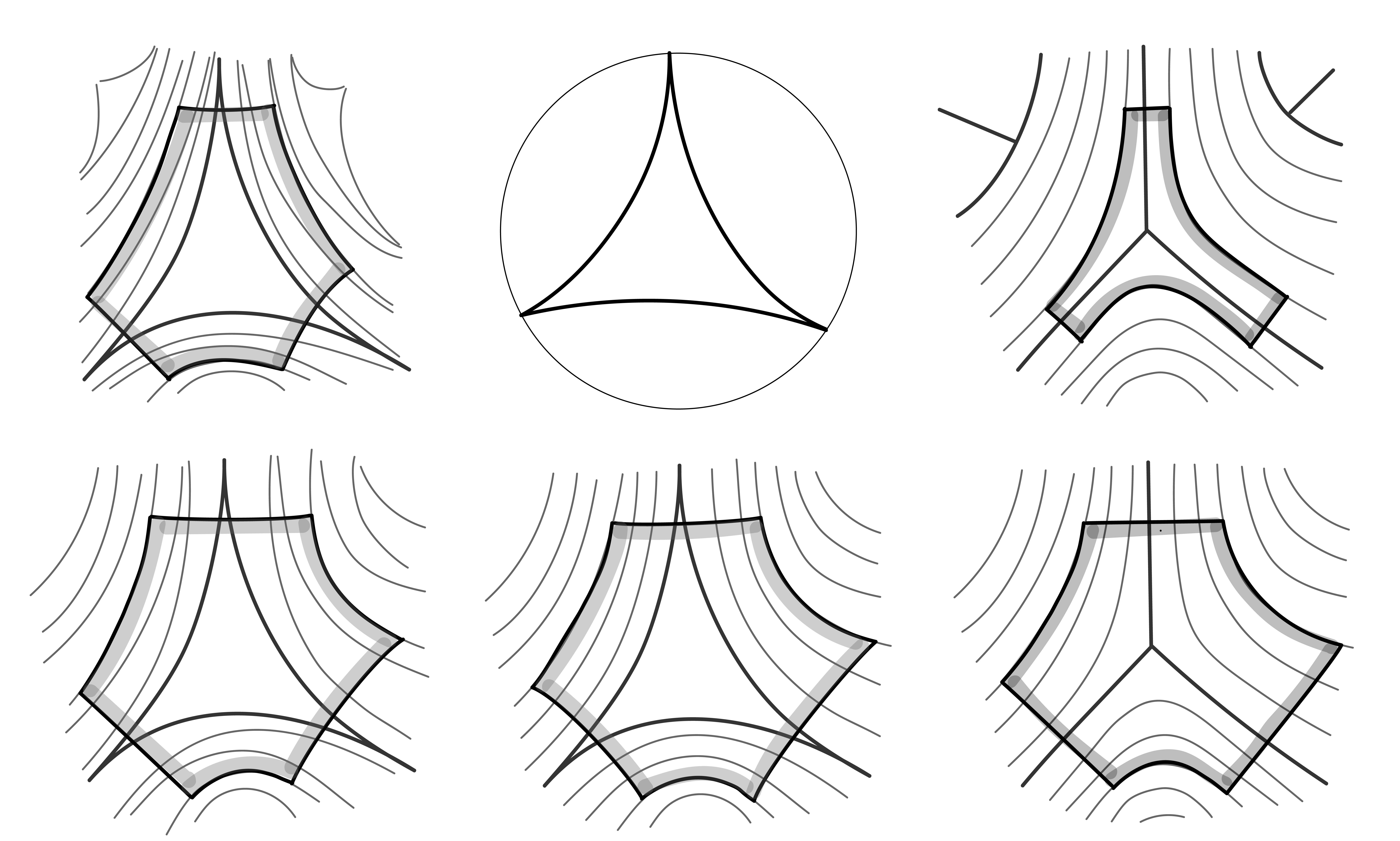} % figure file
         \put(5, 55){\textcolor{black}{\small \contour{white}{$\ti{\sigma}_\infty \cong \H^2$} }}  
    \put(2, 26){\textcolor{black}{\small \contour{white}{$\widetilde{\Gr_{L_\infty}^{\exp s/ d}\sigma}_\infty$} }}  
 \put(90, 32){\textcolor{black}{\small \contour{white}{$z dz^2$}}}   
  \put(45, 42){\textcolor{black}{\small \contour{white}{$\Delta$} }}  
  \put(38, 55){\textcolor{black}{\small \contour{white}{$\H^2$} }}  
    \put(41, 46){\textcolor{black}{\small \contour{white}{$\infty$} }}  
  \put(51, 46){\textcolor{black}{\small \contour{white}{$\infty$} }}  
  \put(45, 36){\textcolor{black}{\small \contour{white}{$\infty$} }}  
      \put(37, 24){\textcolor{black}{\small \contour{white}{$C_q$} }}  
   \put(79, 42){\textcolor{Black}{\small \contour{white}{$Q_{j, h}$} }} 
    \put(12, 46){\textcolor{Black}{\small \contour{white}{$\QQ_{j, k}$} }}  
    \put(12, 17){\textcolor{Black}{\small \contour{white}{$\QQ_{j, k}^s$} }}  
    \put(70, 55){\textcolor{Black}{\small \contour{white}{$\ti{E}_\infty$} }}  
        \put(70, 23){\textcolor{Black}{\small \contour{white}{$E_q$} }}  
        \put(78, 18){\textcolor{Black}{\small \contour{white}{$Q_{j, h}^s$} }}  
        \put(44, 18){\textcolor{Black}{\small \contour{white}{$\QQQ_{j, h}^s$} }}  
  \put(47, 11){\textcolor{black}{\small \contour{white}{$\Delta'$} }}
       %   \put( , ){\textcolor{}{$$}}  
 %   \put( , ){}  
      \end{overpic}
\caption{Interporeating hyperbolic and singular Euclidean hexagons}\Label{fInterpolatingHexagons}
\end{figure}

In this subsection, we construct a quasi-conformal extension of $\phi_j^s$ with small distortion to each hexagonal branch $\QQ_{j, h}^s$ .
We first construct a model projective structure on a hexagon that interpolates between a hyperbolic hexagonal branch $\QQ_{j, h}^s$ and its corresponding flat hexagonal branch $Q_{j, h}^s$ which contains one cone point of angle $3\pi$ in its interior.

Let $q$ be the quadratic differential $z dz^2$ on $\C$.    
Consider the singular Euclidean metric $E_q$ on $\C$ given by $q$.
Let $V_q$ denote the vertical measured foliation on $\C$ given by $q$.
Then $\C$ is the union of three Euclidean half-planes with a common boundary point at $0$. 
The vertical singular foliation $V_q$ has a $Y$-shaped singular leaf. 

Let $C_q$ be the $\CP^1$-structure on $\C$ given by the quadratic differential $q$.
Then Thurston's parameters of $C_q$ are the ideal triangle $\Delta$ in $\H^3$ and the measured lamination $L_q$ consisting of the boundary geodesics of $\Delta$ with infinite weight. 
Let $\LL_q$ be the corresponding Thurston's lamination on $\C$ (\cite{Kulkani-Pinkall-94}, see also \cite{Baba20ThurstonParameter});  
then, with respect to Thurston's metric,  the complement of $\LL_q$ is an ideal triangle $\Delta'$, while the foliated region $| \LL_q|$ consists of three Euclidean half-planes. (See \Cref{fInterpolatingHexagons}, Middle.)

Let $\lam_{\Del'}$ be the horocyclic measured lamination of the ideal triangle $\Del'$.
Then there is a collapsing map of  $\Delta'$ to a Y-shaped metric graph with infinite ends, which collapses each leaf of $\lam_{\Del'}$ to a point and the complementary triangle to a point. 
The collapsing map collapses each horocyclic leaf to a point and the complementary triangle to the vertex of the Y-shaped graph.

Let $C_q \to (\C, \frac{z}{(\sqrt{2}d)^2} dz^2)$ be the mapping which, by the above collapsing map, takes the ideal triangle $\Delta'$ to the Y-shaped singular vertical leaf, such that $C_q$ is isometric on each half plane of $C_q \minus \Del'$. 
Let $\kap\col C_q \to E_q = (\C, z dz^2)$ be the composition of this collapsing map with the scaling map $\C \to \C, z \mapsto (\sqrt{2}d) z$ by $\sqrt{2} d$.

Recall that $Q_{j, h}^s$ is a hexagonal branch of the polygonal traintrack decomposition $E_{\infty, j}(s)$ of $E_\infty(s)$ associated with the fat traintrack structure $T_{\infty, j}(s)$.
Then $Q_{j, h}^s$ is isometrically embedded in the singular Euclidean surface of $(\C, q)$, so that the horizontal foliation of $Q_{j, h}^s$ maps to the vertical foliation of $(\C, q)$ and the vertical foliation of $Q_{j, h}^s$ maps to the horizontal foliation of $(\C, q)$.
Regarding $Q_{j, h}^s$ as a subset of $\C$ by this embedding, we let $\QQQ_{j, h}^s$ be $\kap^{-1} (Q_{j, h}^s)$ as in \Cref{sModelHexagon}.

We will construct a desired almost conformal mapping from the Euclidean hexagon $Q_{j, h}^s$ to the hyperbolic hexagon $\QQ_{j, h}^s$ through this model Euclidean hexagon $\QQQ_{j, h}^s$.
(\Cref{fCommuativeDiagram}.)

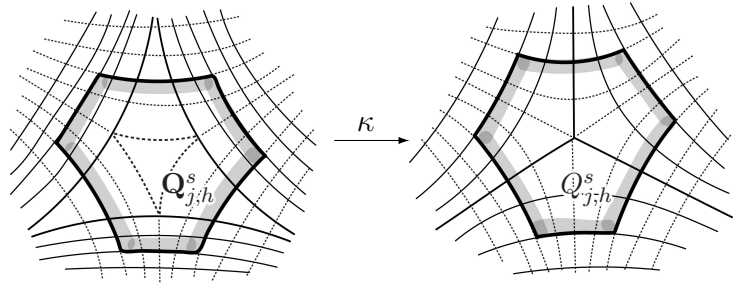
\begin{figure}
\begin{overpic}[scale=.15%, grid,tics=10
] {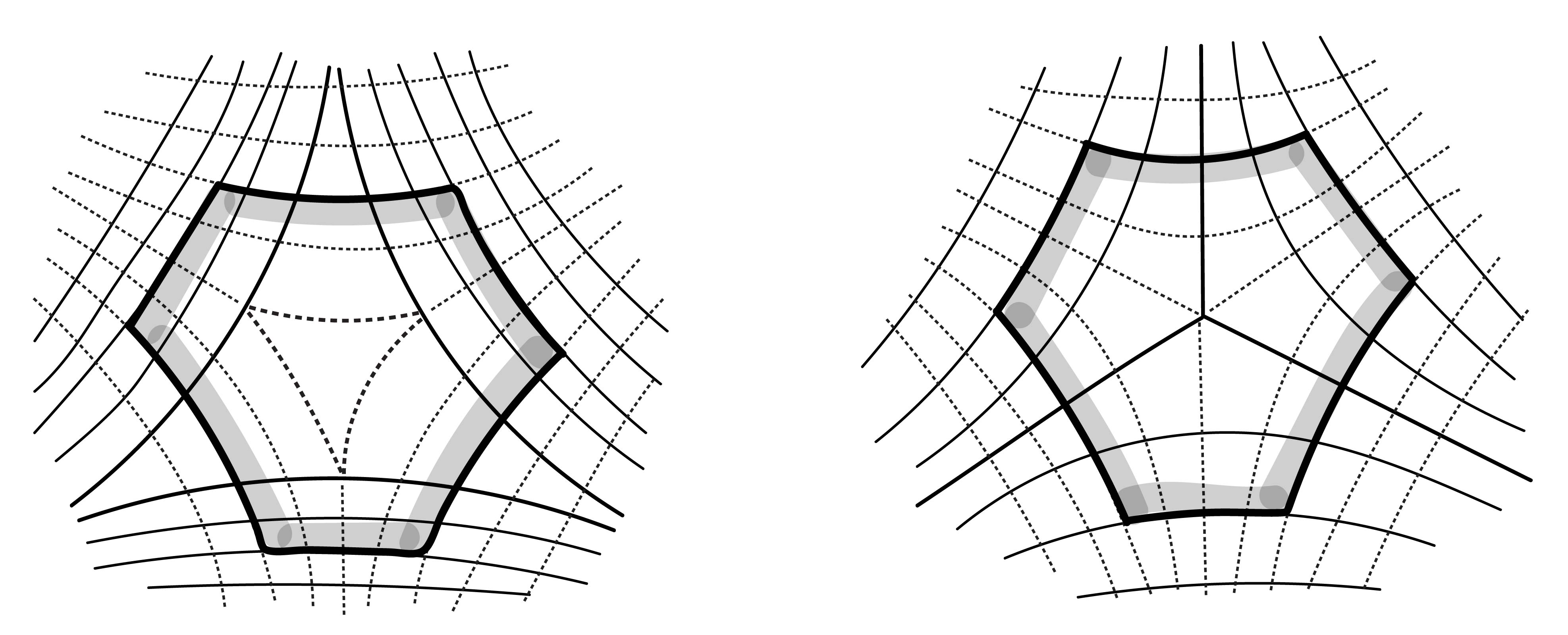} % figure file
\put(75, 13){\textcolor{Black}{\small \contour{white}{$Q_{j, h}^s$}}}   
\put(22, 13){\textcolor{Black}{\small \contour{white}{$\QQQ_{j, h}^s$}}}   
 \put(45,20){\color{black}\vector(1,0){10}}
 \put( 48,21.2 ){\textcolor{black}{$\kappa$}}  
% \put(, ){\textcolor{white}{\tiny \contour{}{$$}}}   
 %   \put( , ){\textcolor{}{$$}}  
 %   \put( , ){}  
      \end{overpic}
\caption{Constructing a model hexagon.}\Label{sModelHexagon}
\end{figure}

\subsubsection{Almost conformal identification of the Euclidean hexagon $Q_{j, h}^s$ and the model projective hexagon $\QQQ_{j, h}^s$}

Let $f_q\col \C \to \CP^1$ denote the developing map of the $\CP^1$-structure given by $(\C, \frac{z}{(\sqrt{2}d)^2} dz^2).$
The half-Euclidean-plane in $E_q$ bounded by a vertical line passing the singular point $0$ is called an {\it anti-Stokes sector}.
Then the flat surface of $E_q$ is divided into three anti-Stokes sectors. 
\begin{theorem}[Corollary 4.1 in  \cite{GuptaMj21}]\Label{iAssymptotic}
In every anti-Stokes sector,  for every $m \geq 0$, 
\begin{eqnarray} 
( f_q(z) - \exp[\sqrt{2} z^{\frac{3}{2}}] ) z^m \to 0 \end{eqnarray}
as $|z| \to \infty$. 
(\Cref{fModelMapping}.)
\end{theorem}
\begin{figure}
\begin{overpic}[scale=.15%, grid,tics=10
] {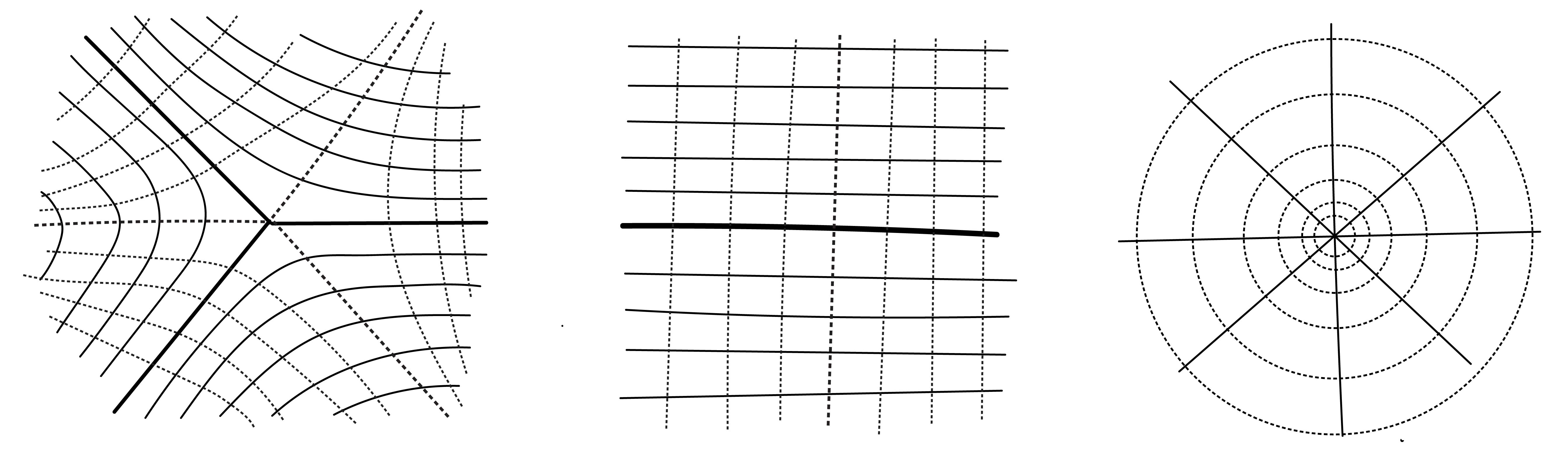} % figure file
\put(2, 27){\textcolor{black}{\small \contour{white}{$(\C, \frac{z}{(\sqrt{2}d)^2} dz^2)$} }}   
\put(40, 27){\textcolor{black}{\small \contour{white}{$\C$} }}   
\put(70, 27){\textcolor{black}{\small \contour{white}{$\C \minus \{0\}$} }}   
 \put(32, 14){\color{black}\vector(1,0){5}}
   \put(33 ,16 ){\textcolor{black}{\tiny $\xi = z^{\frac{3}{2}}$}}  
 \put(65, 14){\color{black}\vector(1,0){5}}
   \put(65 ,16 ){\textcolor{black}{\tiny $e^{-2\xi}$}}  
% \put(, ){\textcolor{}{\tiny \contour{}{$$}}}   
 %   \put( , ){\textcolor{}{$$}}  
 %   \put( , ){}  
      \end{overpic}
\caption{The model mapping $\exp[\sqrt{2} z^{\frac{3}{2}}].$}\label{fModelMapping}
\end{figure}
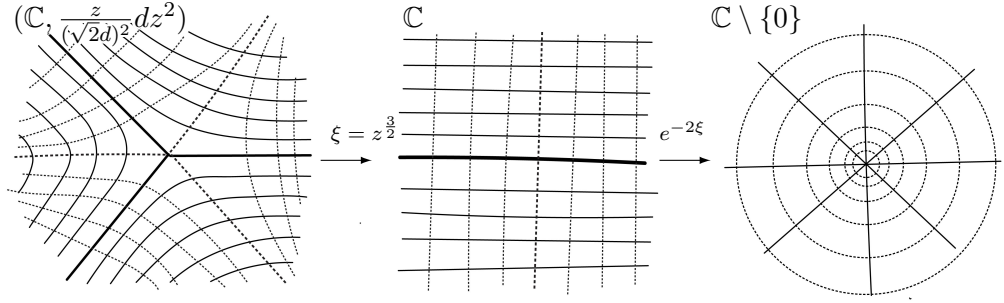

Let $Z_{j, k}^s$ be the set of the boundary points of $Q_{j, h}^s$  which are vertices of rectangles or hexagons of the polygonal decomposition $E_\infty(s) = (\cup_{h = 1}^N \partial Q_{j, h}^s) \cup (\cup_{k =1}^{N'} \partial R_{j, k}^s)$.
By the construction of the polygonal decomposition, $Z_{j, k}^s$ is the set of the vertices of $Q_{j, h}^s$, possibly, with some finitely many points on the vertical edges.
Note that $\kap | \bdr \QQQ_{j, h}^s$ is not a homeomorphism onto $\bdr Q_{j, k}^s$ as $\kap$ collapses some horizontal segments $\bdr Q_{j, h}^s$, but homotopic to a homeomorphism, preserving the vertices. 
The restriction is a linear diffeomorphism on each vertical edge, and $\kap$ takes the restriction of the circular lamination $\LL_q$ to ${\bf Q}_{j, k}^s$ to the vertical measured foliation $V_\infty | Q_{j, h}^s$.

Let $\eta_{j, h}^s \col \bdr Q_{j, h}^s \to \bdr\QQQ_{j, h}^s$ be the edge-wise linear homeomorphism, such that $\eta_{j, h}^s$ coincides with $\kap^{-1}$ at the six vertices of $Q_{j, h}^s$.   
\begin{proposition}\Label{AlmostConformalToModelHexagon}
For every $\ep > 0$, there are $J_\ep > 0$ and $s_\ep > 0$, such that, if   $j > J_\ep$ and $s > s_\ep$, then there is a $(1 + \ep)$-quasiconformal mapping
$Q_{j, h}^s \to \QQQ_{j, h}^s$ which coincides with the edge-wise linear mapping $\eta_{j, h}^s$ on the boundary.
\end{proposition}

The remainder of this subsection is the proof of \Cref{AlmostConformalToModelHexagon}.
Let $\iota_{j, h}^s\col Q_{j, h}^s \to (\C, q =  \frac{z}{(d \sqrt{2})^2} dz^2) \cong C_q$ be the isometric embedding exchainging horizontal and vertical directions, with respect to the flat structure on $C$ given by the differential.
On the other hand,  $\QQQ_{j, h}^s$ is already a subset of $C_q$. 
First we show that $\iota_{j, h}^s  (Q_{j, h}^s)$ is in a bounded distance away from $\QQQ_{j, h}^s$, almost preserving the tangent directions.  
\Cref{iAssymptotic} immediately implies  the Hausdorff distance between $\QQQ_{j, h}^s$ and $\iota_{j, h}^s Q_{j, h}^s$ are uniformly bounded. 
\begin{lemma}\Label{BoundedHausdorffDistance}
There are constants $b > 0, s_\ep > 0, J_\ep > 0$, such that, if $j > J_\ep$ and $s > s_\ep$, then,
for each $x \in \bdr Q_{j, h}^s$, $$d_{Th} (\iota_{j, h}^s(x), \eta_{j, k}^s(x) ) < b,$$ 
in the Thurston metric $d_{Th}$ on $C_q$ (\S \ref{sThurstonMetric}).
\end{lemma}

We now show the closeness of the tangent directions on the hexagon boundary.
\begin{proposition}\Label{AlomstExponentialMap}
For every $\ep > 0$ and $r > 0$, there are $J_\ep > 0$ and $s_\ep > 0$, such that, if $j > J_\ep$ and $s > s_\ep$, then,
\begin{enumerate}
\item for each vertical edge $e$ of $Q_{j, h}^s$, $\iota_{j, h}^s | e$ is $\ep$-almost parallel to Thurston's lamination $\LL_q$ (i.e. the angles between the tangent directions of the curve and the lamination are less then $\ep$);  \Label{iVirticallyParallelToThurstonLamination} 
\item for each horizontal edge $e$ of $Q_{j, h}^s$, $\iota_{j, h}^s | e$ is $\ep$-almost parallel to the horocyclic lamination $\HH_q$ orthogonal to Thurston's lamination $\LL_q$;  \Label{iHorizontallyParallelToHorocyclicLamination}
\item the restriction of $\iota_{j, h}^s$ to the $r$-neighborhood of the boundary $\partial Q_{j, h}^s$ is $(\frac{1}{d} - \ep, \frac{1}{d} + \ep)$-bilipschitz embedding onto its image in $C_q$ w.r.t the Thurston metric.  
\Label{iAlmostIsomtricEmbedding}
\end{enumerate}
\end{proposition}
\begin{proof}
Throughout this proof, we identify $Q_{j, h}^s$ and its image in $E_q$ by $\iota_{j, h}^s\col Q_{j, h}^s \to C_q$.
(\ref{iVirticallyParallelToThurstonLamination})
Recall that
the Thurston parameters of $C_q$ are the ideal hyperbolic triangle $\Delta$ and the geodesic lamination $L_\infty$ consisting of the boundary geodesics of $\Delta$ with weight infinity, and $\LL_q$ is Thurston's circular lamination on $C_q$.
Recall also that the complement of $\LL_q$ in $C_q$ is an ideal triangle $\Delta'$ corresponding to $\Delta$, and $C_q \minus \Delta'$ consists of three Euclidean half-planes foliated by leaves of $\LL_q$. 
Let $\kap_	q\col C_q \to \Delta$ be the collapsing map associated with Thurston's parametrization: it collapses each complementary half-plane to its corresponding boundary geodesic of $\Delta$,  taking leaves of $\LL_q$ diffeomorphically to the boundary geodesic. 

 Let $\ell$ be a leaf of the vertical measured foliation $V_q$ of $(\C, q)$ such that $\ell$ contains a vertical edge $e$ of $Q_{j, h}^s$. 
Let $m$ be the boundary geodesic of the ideal triangle $\Delta$ corresponding to $\ell$. 
By Thurston's parametrization $(\Delta, L_q)$ of $C_q$, the induced bending map is simply an isometric embedding of the ideal triangle $\Delta$ into a totally geodesic plane in $\H^3$. 
By this embedding, $m$ is isometrically identified with a geodesic in $\H^3$.
Then, the ideal boundary $\CP^1$ of $\H^3$ minus the endpoints of $m$ is foliated by round circles bounding disjoint hyperbolic planes orthogonal to $m$; let $\mathcal{C}$ denote this foliation of $\CP^1$ minus two points by those round circles.
 
The developing map of a $\CP^1$-structure corresponds to the bending map of its Thurston parameters by certain nearest point projections (see \ref{sThurstonParameters}).  
Thus the leaves of Thurston lamination $\LL_q$ corresponding to $m$ map to circular arcs connecting the endpoints of $m$ (\Cref{fAlmostVertical}); those circular arcs are orthogonal to the round foliation $\mathcal{C}$.  
Therefore it suffices to show that each tangent vector $v$ along a vertical edge $e$ (in $\ell$) maps to a tangent vector $\ep$-almost orthogonal to the circle foliation $\CC$.

Let $\Ep_q\col C_q \to \H^3$ be the Epstein surface of $C_q$ (\S \ref{sEpsteinSurfaces}).
For every $\ep > 0$, there is $R > 0$ such that, if the disctance of $\ell$ from the singular point, the zero, is least $R$, then $\Ep_q \ell$ is a $(1 - \ep, 1 + \ep)$-bilipschitz embedding and $\ep$-close to the geodesic $m$ (\cite[Lemma 3.4]{Dumas18HolonomyLimit}.) 
Recall that $f_q\col \C \to \CP^1$ denotes the developing map of $C_q$. 
By the property of the Epstein surface,  $d f_q v$ corresponds to $d \Ep_q v$ by the orthogonal projection to the Epstein surface $\Ep_q$. 
Therefore,  if $j > 0$ and $s > 0$ are sufficiently large, then $\Ep_q \ell$ is tangentially very close to the geodesic $m$,  and thus $d f_q v$ is $\ep$-almost orthogonal to a leaf of $\CC$ (\Cref{fAlmostVertical}). 
\begin{figure}
\begin{overpic}[scale=.15%, grid,tics=10
] {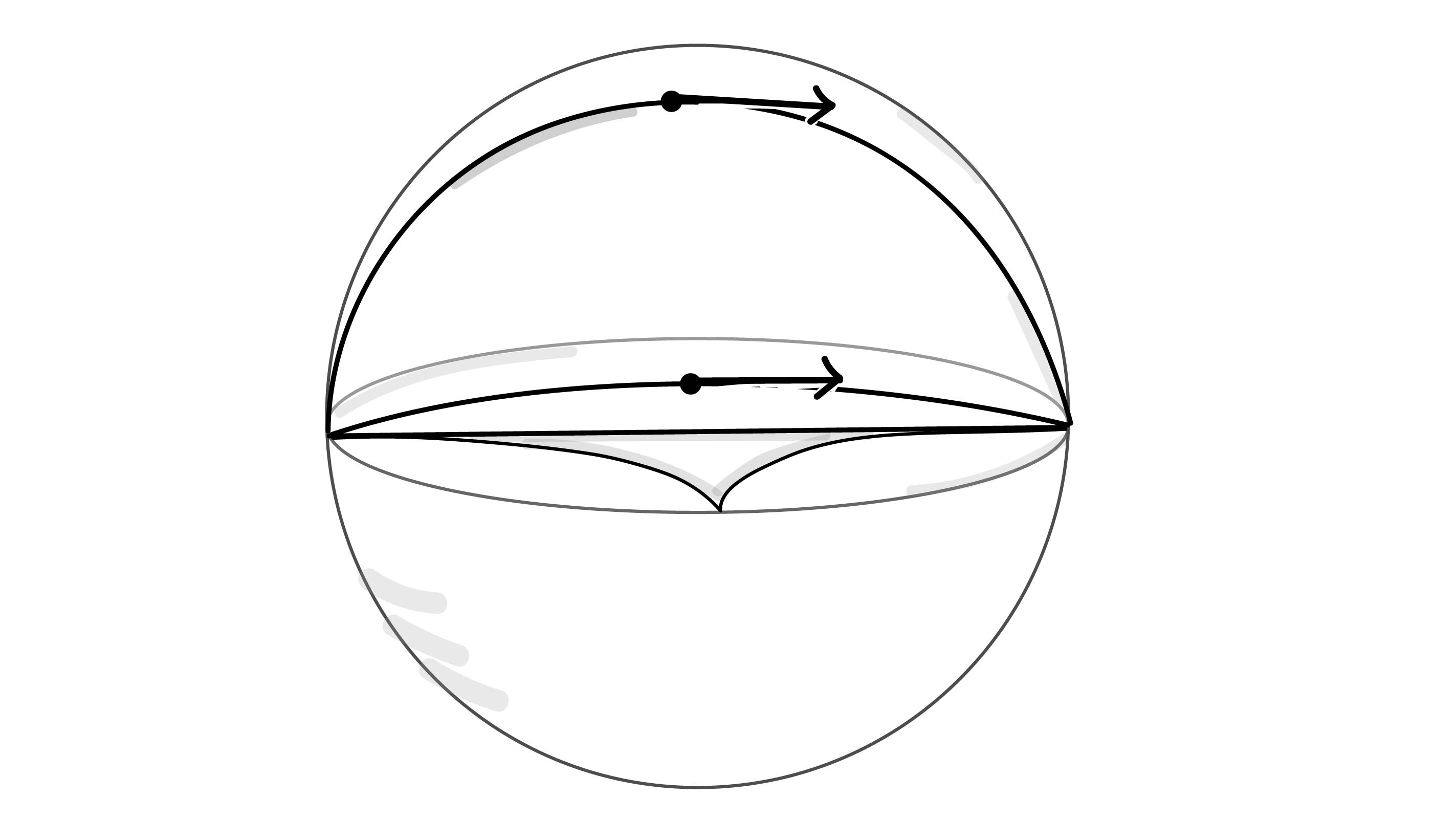} % figure file
\put(40, 23 ){\textcolor{black}{\small \contour{white}{$m$}}}   
\put(35, 31){\textcolor{Black}{\small \contour{white}{$\Ep_q \ell$}}}   
\put(50, 33.5){\textcolor{Black}{\small \contour{white}{$d \Ep_q (v)$}}}   
\put(47, 51.7){\textcolor{Black}{\small \contour{white}{$d f (v)$}}}   
\put(50, 12){\textcolor{Black}{\small \contour{white}{$\H^3$}}}   
\put(25, 48){\textcolor{Black}{\small \contour{white}{$\CP^1$}}} 
% \put(, ){\textcolor{}{\tiny \contour{white}{$$}}}   
 %   \put( , ){\textcolor{}{$$}}  
 %   \put( , ){}  
      \end{overpic}
\caption{}\Label{fAlmostVertical}
\end{figure}

(\ref{iHorizontallyParallelToHorocyclicLamination})
Let $e$ be a horizontal edge of $Q_{j, h}^s$.
If $j > 0$ and $s > 0$ are sufficiently large, we can pick a rectangle $R_e$ in $E_q$ with horizontal and vertical edges, such that $R_e$ is sufficiently far from the zero of $E_q$ and the vertical edges of $R_e$ are long. 
Let  $h_w$ be the horizontal foliation of $R_e$ parametrized linearly by $w \in [0,1]$. 
Let $v_u$ be the vertical foliation of $R_e$ parametrized linearly by $u \in [0,1]$; then $v_0$ and $v_1$ are its vertical edges.
Let $\ell_0$ and $\ell_1$ be the vertical leaves of $V_q$ containing the vertical edges $v_0$ and $v_1$, respectively.
Similar to (\ref{iVirticallyParallelToThurstonLamination}),  let $m_0, m_1$ be the boundary geodesics of the ideal triangle $\Delta$ corresponding to $\ell_0$ and $\ell_1$.
 (\Cref{fHorizontalEdgesAndHorocyclicLamination}.)

For every  $\ep > 0$, if $j > 0$ and $s > 0$ are sufficiently large, then
Since $R_e$ is far from the zero of $q$,  $\Ep_q v_u$ are $(\sqrt{2} - \ep, \sqrt{2} + \ep)$-bilipschitz  embedding into $\H^3$ for all $u \in [0,1]$ and $\Ep_q h_w$ has length less than $\ep$ for all $w \in [0,1]$ (\cite{Epstein84}, Lemma 2.6, Lemma 3.4 in \cite{Dumas18HolonomyLimit}). 
Therefore, we may, in addition, assume that the long almost-geodesic curves $\Ep_q v_u ( u \in [0,1])$ are $\ep$-close to each other (\Cref{fHorizontalEdgesAndHorocyclicLamination}). 
Similar to (\ref{iVirticallyParallelToThurstonLamination}), let $\mathcal{C}$ be the foliation of $\CP^1$ minus endpoints of $m_0$ by round circles which bound hyperbolic planes orthogonal to the geodesic $m_0$. 
Then, for every $\ep > 0$, if $j, s > 0$ are sufficiently large, then $f_q v_u$ are $\ep$-almost orthogonal to $\CC$ for all $u \in [0,1]$, since the almost geodesic curves $\Ep_q v_u$ are very close to a common segment $\alpha$ of  the geodesic  $m_0$. 
Therefore $f_q h_w$ are $\ep$-almost parallel to $\CC$ for all $w \in [0,1]$.

The horocyclic foliation $\HH_q$ is orthogonal to the Thurston lamination $\LL_q$ on $C_q$. 
Since $R_e$ is sufficiently far away from the zero of $q$,  the boundary $m_0, m_1$ of the ideal triangle $\Delta$ are close to each other near $\alpha$.
Therefore, $h_w$ are $\ep$-almost orthogonal to $\LL_q$, and therefore $\ep$-almost parallel to $\HH_q$. if $j$ and $s$ are sufficiently large.
 
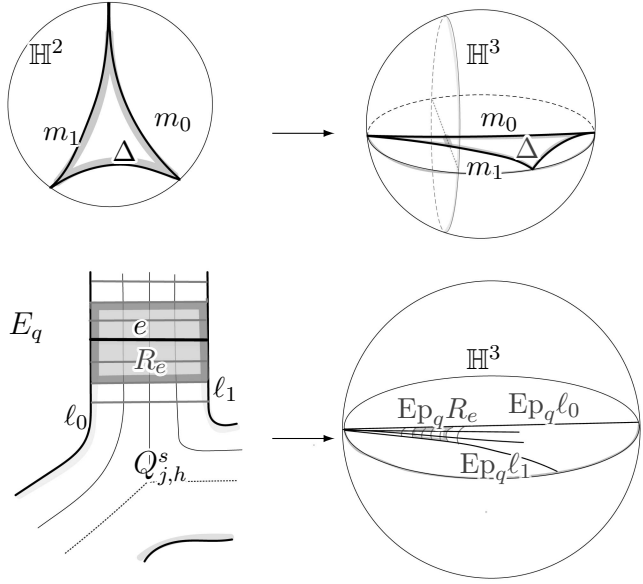
\begin{figure}
\begin{overpic}[scale=.25%, grid,tics=10
] {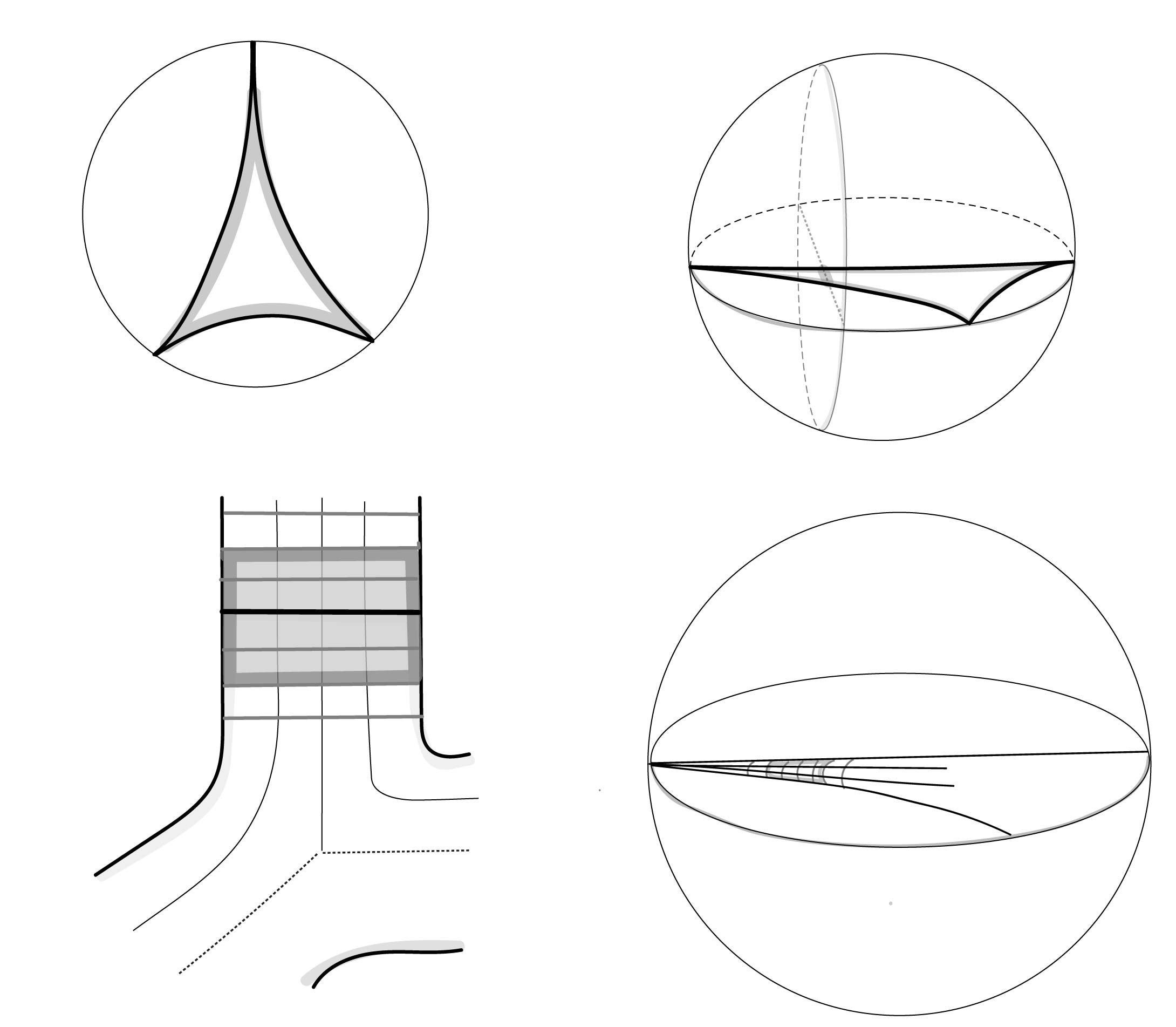} % figure file
% upperleft 
\put(10, 75){\textcolor{black}{\small \contour{white}{$\H^2$}}}
\put(28, 66){\textcolor{black}{\small \contour{white}{$m_0$}}}  
\put(12, 64){\textcolor{black}{\small \contour{white}{$m_1$}}}   
\put(73, 75){\textcolor{black}{\small \contour{white}{$\H^3$}}}
\put(22, 61){\textcolor{black}{\small \contour{white}{$\Delta$}}}   
\put(80., 61.5){\textcolor{darkgray}{\small \contour{white}{$\Delta$}}}  
\put(75, 66){\textcolor{black}{\small \contour{white}{$m_0$}}}   
\put(15, 23){\textcolor{black}{\small \contour{white}{$\ell_0$}}}   
\put(79, 24.7){\textcolor{darkgray}{\small \contour{white}{$\Ep_q\ell_0$}}}   
\put(73, 59){\textcolor{black}{\small \contour{white}{$m_1$}}}   
\put(36.5, 27){\textcolor{black}{\small \contour{white}{$\ell_1$}}}   
\put(72, 16){\textcolor{darkgray}{\small \contour{white}{$\Ep_q\ell_1$}}}   
\put(25, 16){\textcolor{black}{ \contour{white}{$Q_{j, h}^s$}}}   
\put(25, 36){\textcolor{black}{\small \contour{white}{$e$}}}
\put(7, 36){\textcolor{black}{\contour{white}{$E_q$}}}
\put(25, 31){\textcolor{darkgray}{\small \contour{white}{$R_e$}}}
\put(63, 24){\textcolor{darkgray}{\small \contour{white}{$\Ep_q R_e$}}}
\put(73, 31){\textcolor{black}{\small \contour{white}{$\H^3$}}}

 \put(45,21){\color{black}\vector(1,0){9}}
 \put(45,65){\color{black}\vector(1,0){9}}

% \put(, ){\textcolor{}{\tiny \contour{white}{$$}}}   
 %   \put( , ){\textcolor{}{$$}}  
 %   \put( , ){}  
%\put(, ){\color{}\vector(,){}} %{length}
      \end{overpic}
\caption{The $\iota_{j, h}^s$-images of horizontal edges of $Q_{j, s}^s$ are almost parallel to the horocyclic lamination of $C_q$.}\Label{fHorizontalEdgesAndHorocyclicLamination}
\end{figure}

(\ref{iAlmostIsomtricEmbedding})
For every $R > 0$,  if  $j > 0$ and $s > 0$ are sufficiently large, then the $r$-neighborhood of $\bdr Q_{j, h}^s$ has a $\iota_{j, h}^s$-image in $(\C, q =  \frac{z}{(d \sqrt{2})^2} dz^2)$ whose distance from $0$ is at least $R$.
Therefore, by 
\cite[Proposition 4.9]{Baba-25}, the developing map $f_q$ on the $r$-neighborhood is well approximated by the exponential map. 
Hence, (\ref{iVirticallyParallelToThurstonLamination})
 and  (\ref{iHorizontallyParallelToHorocyclicLamination}) imply the desired bilipschitz property. 
\end{proof}

Let $\iota_\phi\col C_q \to E_q$ be the identification map given by the Schwarzian parametrization $C_q \cong (\C, q)$.
We have seen that $\iota_\phi$ embeds $Q_{j, h}^s$ into $\C$ so that its image is bounded Hausdorff distance from $\QQQ_{j, h}^s$ (\Cref{BoundedHausdorffDistance}), and moreover they are close in a $C^1$-manner (\Cref{AlomstExponentialMap}). 
\begin{lemma}
For every $\ep > 0$, there are $J_\ep > 0$ and $s_\ep > 0$ such that, if $j > J_\ep$ and $s > s_\ep$, then 
we can modify $\iota_\phi | Q_{j, h}^s$, with respect to the (singular) Euclidean metric $E_q$ of $(\C, q)$,  by a $(1 - \ep, 1 + \ep)$-bilipschitz  mapping so that
\begin{enumerate}
\item $Q_{j, h}^s$ is identified with $\QQQ_{j, h}^s$ by a $(1 + \ep)$-quasiconformal mapping, and \Label{iImageAdjustment}
\item  the boundary of $Q_{j, h}^s$ is identified with the boundary of $\QQQ_{j, h}^s$ by the edgewise linear mapping $\eta_{j, h}^s$, with respect to arc length. \Label{iBoundaryMatching}
\end{enumerate}
\end{lemma}
\begin{proof}
We identify $Q_{j, h}^s$ with the image of $Q_{j, h}^s$ in $E_q$ by $\iota_{j, h}^s$.
Let $H$ denote the (hexagonal) boundary of the hexagonal branch $Q_{j, h}^s$ in $E_q$. 
For $R > 0$, let $N_R$ be  the $R$-neighborhood of the hexagonal boundary $H$ in $Q_{j, h}^s$ in the Euclidean metric $E_q$ (\Cref{fProductStructuresOnHexagonalCylinders}, left).
Then $N_R$ is topologically a cylinder. 
The outer boundary $H$ is identified with the inner boundary of $N_R$ by the edge-wise linear homeomorphism which identifies a pair of parallel edges. 
Thus $N_R$ has a natural product structure $H \times [0,1]$ by linearly extending this identification of the hexagonal boundary components; for each $h \in H$, the segment $h \times [0, 1]$ is a line segment in $N_R$ that connects a pair of identified points on the corresponding inner and outer edges.  
Let $b > 0$ be the Hausdorff distance bound in  \Cref{BoundedHausdorffDistance}. 
Then, if $R > 2b$, then $\iota_\phi(\QQQ_{j, h}^s)$ contains the inner boundary of $N_R$.

Let $\NNN_R$ be the region in $E_q$ bounded by the boundary hexagon of $\iota_\phi (\QQQ_{j, h}^s)$ and  the inner boundary $H \times \{0\}$ of $N_R$ (\Cref{fProductStructuresOnHexagonalCylinders}).
We shall define a natural product structure $H \times [0,1]$ on $\NNN_R$, such that, via this product structure $\NNN_R = H \times [0,1]$, 
\begin{itemize}
\item  the identification $\NNN_R = H \times [0,1] = N_R$ agrees with the identity on the inner hexagonal boundary and $\eta_{j, h}^s\col \bdr Q_{j, h}^s \to \bdr \QQQ_{j, h}^s$ on the outer hexagonal boundary, and 
\item $h \times [0,1]$ are straight line segments connecting the outer and inner hexagonal boundary for all $h \in [0,1]$. 
\end{itemize}

   If $j > 0$ and $s > 0$ are sufficiently large, 
the Hausdorff distance between the outer boundary hexagon of $Q_{j, h}^s$ and the boundary hexagon of $\iota_\phi \QQQ_{j, h}^s$ is less than a fixed constant $b$ by \Cref{BoundedHausdorffDistance}, and then the corresponding edges are $\ep$-almost parallel by \Cref{AlomstExponentialMap} w.r.t the singular Euclidean metric $E_q$.

Therefore we let $f\col N_R \to \NNN_R$ be the mapping such that 
 \begin{itemize}
 \item the restriction of $f$ to the outer boundary of $N_R$ is equal to $\eta_{j, h}^s$;
 \item the restriction of $f$ to the inner boundary of $N_R$ is the identity map; 
 \item $f$ is linear of each $\{h\} \times [0,1]$ for all $h \in H$. 
\end{itemize}
The edges of the hexagons are long if $j > 0$ and $s > 0$ are large. 
Therefore
\Cref{AlomstExponentialMap} implies that following. 
 \begin{claim}
 For every $\ep > 0$, there are $J_\ep > 0$ and $s_\ep > 0$, such that if $ j > J_\ep$ and $s > s_\ep$, then 
$f$ is a $(1 + \ep)$-quasiconformal mapping. 
\end{claim}
  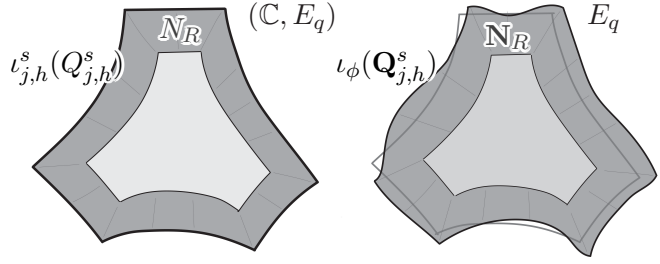
\begin{figure}
\begin{overpic}[scale=.03%, grid,tics=10
] {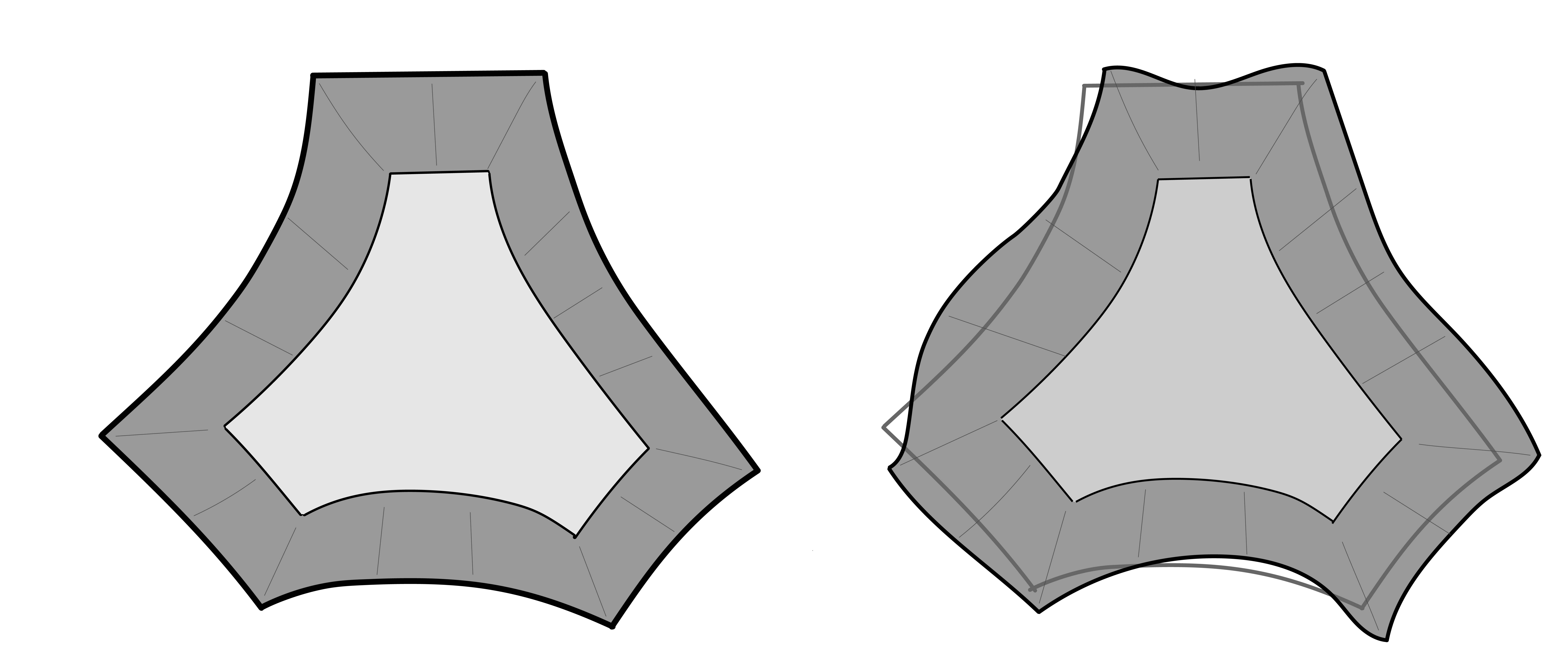} % figure file
 \put(25, 33){\textcolor{darkgray}{\small \contour{white}{$N_R$}}}   
\put(73, 32){\textcolor{darkgray}{\small \contour{white}{$\NNN_R$}}}   
\put(3, 28){\textcolor{black}{\small \contour{white}{$\iota_{j, h}^s(Q_{j, h}^s)$}}}   
\put(51, 28){\textcolor{black}{\small \contour{white}{$\iota_\phi(\QQQ_{j, h}^s)$}}}   
\put(88, 35){\textcolor{Black}{\small \contour{white}{$E_q$}}}   
\put(38, 35){\textcolor{Black}{\small \contour{white}{$(\C, E_q)$}}}   
% \put(, ){\textcolor{}{\tiny \contour{white}{$$}}}   
 %   \put( , ){\textcolor{}{$$}}  
 %   \put( , ){}  
      \end{overpic}
\caption{The product structures on hexagonal cylinders.}\label{fProductStructuresOnHexagonalCylinders}
\end{figure}
\end{proof}

We completed the proof of \Cref{AlmostConformalToModelHexagon}.
\subsubsection{Hyperbolic Hexagons are almost conformal to model projective Hexagons}
Let $ \mathbf{Z}_{j, k}^s =  (\phi_{j, h}^s)^{-1} Z_{j, k}^s$, the set of points on the boundary of $\QQQ_{j, h}^s$ corresponding to the vertices of the polygonal decomposition of $E_\infty(s)$.
For every $\ep > 0$, there are $J_\ep > 0$ and $s _\ep > 0$, such that, for every hexagonal branch, we 
 constructed a $(d - \ep, d + \ep)$-bilipschitz mapping 
$\phi_{j, h}^s \col \bdr \QQ_{j, h}^s \to \bdr Q_{j, h}^s$ which is linear on each segment of $\bdr \QQ_{j, h} \minus \mathbf{Z}_{j, k}^s$ and  a $(\frac{1}{d\sqrt{2}} - \ep, \frac{1}{d\sqrt{2}} + \ep)$-bilipschitz mapping   $\eta_{j, h}^s\col \bdr Q_{j, h}^s \to \bdr \QQQ_{j, h}^s$ which is linear on each segment between adjacent points of $\mathbf{Z}_{j, k}^s$.

\begin{figure}

\begin{tikzpicture}[
 % node distance=2cm and 2.5cm,
  every node/.style={inner sep=5pt},
  arrow/.style={->},
    hookarrow/.style={{Hooks[right]}->},
]

% Nodes (objects)
\node (QQ) at (0, 0) {$\QQ_{j, h}^s $};
\node (Q) at (3.5, 0) {$Q_{j, h}^s$};
\node (QQQ) at (0, -1.5) {$\QQQ_{j, h}^s$};
\node (C) at (0.9, -1.5) {$\subset C_q$};
\node (CC) at (3.5, -1.5) {$(\C, z dz^2)$};

% Arrows (morphisms)
\draw[arrow] (QQ) -- (Q) 
	node[midway, above] {\tiny $(d - \ep, d + \ep)$-bilip.}
         node[midway, below] {\small $\xi_{j, h}^s$};
\draw[arrow ] (QQ) -- (QQQ) 
	node[midway, left] {\tiny $ (1 + \ep)-q.c.$}
	node[midway, right] {\tiny  $\phi_{j, k}^s$};
\draw[arrow] (Q) -- (CC)
	 node[midway, right] {\tiny $(\frac{1}{d \sqrt{2}} - \ep, \frac{1}{d \sqrt{2}} + \ep)$-bilip.}
	 node[midway, left] { $\iota_{j, h}^s$};
\draw[arrow] (C) -- (CC) node[midway, above] {\tiny $\times \sqrt{2}$};

\end{tikzpicture}

\caption{The commutative diagram for the construction of an almost conformal mapping $\QQ_{j, h}^s \to Q_{j, h}^s $.}\label{fCommuativeDiagram}
\end{figure}

We shall construct an almost conformal mapping from $\QQ_{j, h}^s$ to $\QQQ_{j, h}^s$ which continuoutly extends the piecewise linear homeomorphism $\eta_{j, h}^s \circ \phi_{j, h}^s \col \bdr \QQ_{j, h}^s \to \bdr \QQQ_{j, h}^s$.  
Those singular points are only on the vertical edges of $\bdr \QQ_{j, h}^s$  and the number is uniformly bounded from above by $2 (2 g- 2)$, the number of the singular points on $E_\infty(s)$, where $g$ is the genus of the surface.

\begin{lemma}
For every $\ep > 0$, there are $J_\ep > 0$ and $s_\ep > 0$, such that $j > J_\ep$ and $s > s_\ep$,  then 
there is a $(1 + \ep)$-quasiconformal mapping
$\xi_{j, h}^s\col Q_{j, h}^s \to \QQQ_{j, h}^s$ continuously extending the edgewise linear mapping  $\eta_{j, h}^s\col  \bdr Q_{j, h}^s \to \bdr \QQQ_{j, h}^s$.
\end{lemma}

\begin{proof}
In the hexagon $\QQQ_{j, h}^s$, the complement of the Thurston lamination $\LL_q$,   the hyperbolic hexagon $\mathbf{H}_{j, h}^s$ obtained by cutting by the ideal triangle along horocyclic arcs centered at the vertices. 
Then the complement of  $\mathbf{H}_{j, h}^s$ in  $\QQQ_{j, h}^s$ consists of three Euclidean rectangles in Thurston's metric. 
(\Cref{fDecomposirtionsOfHexagons}.)

Similarly, the hexagon $Q_{j, h}^s$ contains the hyperbolic hexagon $H_{j, h}^s$ obtained by cutting the ideal triangle $\Delta$ along three horocyclic arcs. 
For every $\ep > 0$, there are $J_\ep > 0$ and $s_\ep > 0$, such that if $j > J_\ep$ and $s > s_\ep$, then 
then three components of $Q_{j, h}^s \minus H_{j, h}^s$ are, with respect to Thurston's metric,  $(1 - \ep, 1 + \ep)$-bilipschitz to the corresponding three complementary Euclidean rectangles of $\QQQ_{j, h}^s \minus \HHH_{j, h}^s$, as described in \S\ref{sAlmostConformalRectangles}, using the product structure given by the horocyclic foliation and the orthogonal geodesic foliation.
Recall that this map linearly preserves horizontal foliation, and it is linear with respect to the vertical hyperbolic distance. 

For every $\ep > 0$, 
if $J_\ep > 0$ is sufficiently large, then the corresponding  vertical edges $\HHH_{j, h}^s$ and $H_{j, h}^s$ are $(1 - \ep, 1 + \ep)$-bilipschitz, for sufficiently large $s_\ep$.
Moreover, we can in addition assume that the difference of vertical edge length is less than $\ep$, since the estimate is quadratic in the distance from the zero set (\cite{Epstein84}, \cite[Lemmas 2.6, 3.4]{Dumas18HolonomyLimit}; c.f. \cite[Lemma 3.1]{Baba_23}).  
Therefore there is a $(1 + \ep)$-quasiconformal mapping $H_{j, h}^s \to\HHH_{j, h}^s$, preserving vertical and horizontal edges such that it is linear on each edge of $H_{j, h}^s$ with respect to the hyperbolic length. 

\begin{figure}
\begin{overpic}[scale=.15%, grid,tics=10
] {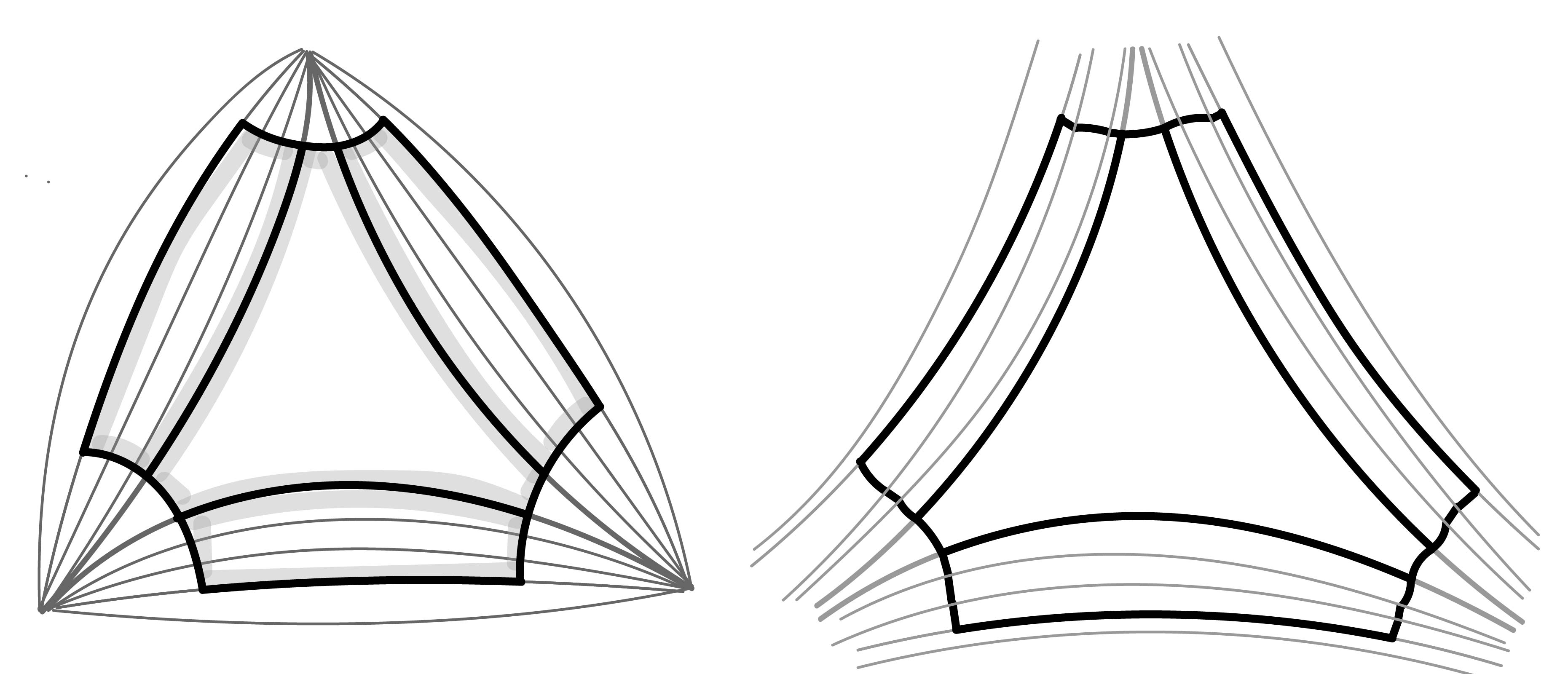} % figure file
 \put(12, 28){\textcolor{black}{\small \contour{white}{$\QQQ_{j, h}^s$}}}  
 \put(67, 30){\textcolor{black}{\small \contour{white}{$Q_{j, h}^s$}}}  
 \put(19, 17){\textcolor{Black}{\small \contour{white}{$\HHH_{j, h}^s$}}}  
 \put(70, 17){\textcolor{Black}{\small \contour{white}{$H_{j, h}^s$}}}  
   % \put(, ){\textcolor{}{\tiny \contour{}{$$}}}   
 %   \put( , ){\textcolor{}{$$}}  
 %   \put( , ){}  
      \end{overpic}
\caption{Mapping blue hexagon to blue hexagons and rectangles to rectangles.}\Label{fDecomposirtionsOfHexagons}
\end{figure}
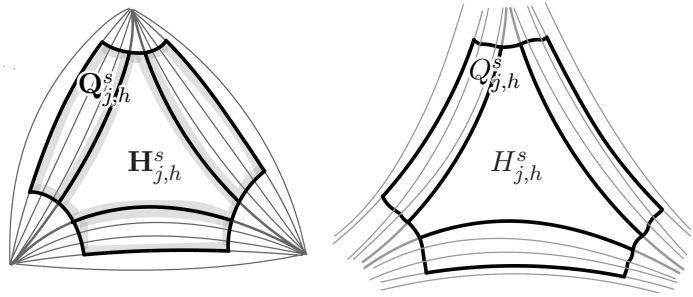

We have constructed  $(1 + \ep)$-quasiconformal mappings from rectangle and hexagon pieces of  $Q_{j, h}^s$ to corresponding rectangle and hexagon pieces of  $\QQQ_{j, h}^s$ so that they agree along their common vertical edges. 
Thus, by gluing those quasi-conformal mappings along vertical edges, we obtain a desired $(1 + \ep)$-quasiconformal mapping. 
\end{proof}

 \begin{corollary}
  We can in addition assume that the $(1 + \ep)$-quasiconformal mapping $\xi_{j, h}^s\col Q_{j, h}^s \to \QQQ_{j, h}^s$  coincides with $\eta_{j, h}^s \circ \phi_{j, h}^s$ on the boundary $\bdr Q_{j, h}^s$. 
\end{corollary}
\begin{proof}   
We first modify $\xi_{j, h}^s$ so that it coincides with  $\eta_{j, h}^s \circ \phi_{j, h}^s$ on vertical edges of $Q_{j, h}^s$. 
Let $v$ be a vertical edge of $\QQQ_{j, h}^s$. 
Let $\mathbf{R}$ be the corresponding rectangular component of $\QQQ_{j, h}^s \minus \HHH_{j, h}^s$ such that $v$ is also a vertical edge of $\RRR$ (\Cref{fSupportForModification}). 

    There is a unique piecewise linear map from $\zeta\col \RRR \to \RRR$ such that 
\begin{itemize}
\item the restriction of $\xi_{j, h}^s$ on $v$ coincides with the composition of $\eta_{j, h}^s \circ \phi_{j, h}^s$ with $\zeta$ on $v$;
\item $\zeta$ is linear on each horizontal leaf of $\RRR$;
\item $\zeta$ is the identity map on the vertical edge of $\RRR$ opposite to $v$. 
\end{itemize}
Then, for every $\ep > 0$,  there are $J_\ep > 0$ and $s_\ep > 0$, such that, if  $j > J_\ep$ and  $s > s_\ep$, then $\zeta$ is a $(1 - \ep, 1 + \ep)$-bilipschitz mapping, as $\RRR$ has sufficiently long vertical and horizontal edges. 
Then, by modifying $\xi_{j, h}^s \col  Q_{j, h}^s \to \QQQ_{j, h}^s$ by  post-composing $\xi_{j, h}^s$ with $\zeta\col \RRR \to \RRR$,   $\xi_{j, h}^s$ coincides with $\eta_{j, h}^s \circ \phi_{j, h}^s$ on the vertical edge $v$. 

By applying this to all vertical edges of $\QQQ_{j, h}^s$, we can modify $\xi_{j, h}^s \col  Q_{j, h}^s \to \QQQ_{j, h}^s$ so that  $\xi_{j, h}^s$ coincides with $\eta_{j, h}^s \circ \phi_{j, h}^s$ on all three vertical edges of $Q_{j, h}^s $.
Then,  there are $J_\ep > 0$ and $s_\ep > 0$, such that, if  $j > J_\ep$ and  $s > s_\ep$, then this modified mapping $\xi_{j, h}^s$ is still $(1 + \ep)$-quasiconformal after this modification. 

Similarly, we can modify $\xi_{j, h}^s$ along appropriately large rectangular regular neighborhoods of horizontal edges of  $\QQQ_{j, h}^s$, so that $\xi_{j, h}^s$ also coincides with $\eta_{j, h}^s \circ \phi_{j, h}^s$ along horizontal edges. 
\begin{figure}
\begin{overpic}[scale=.15%, grid,tics=10
] {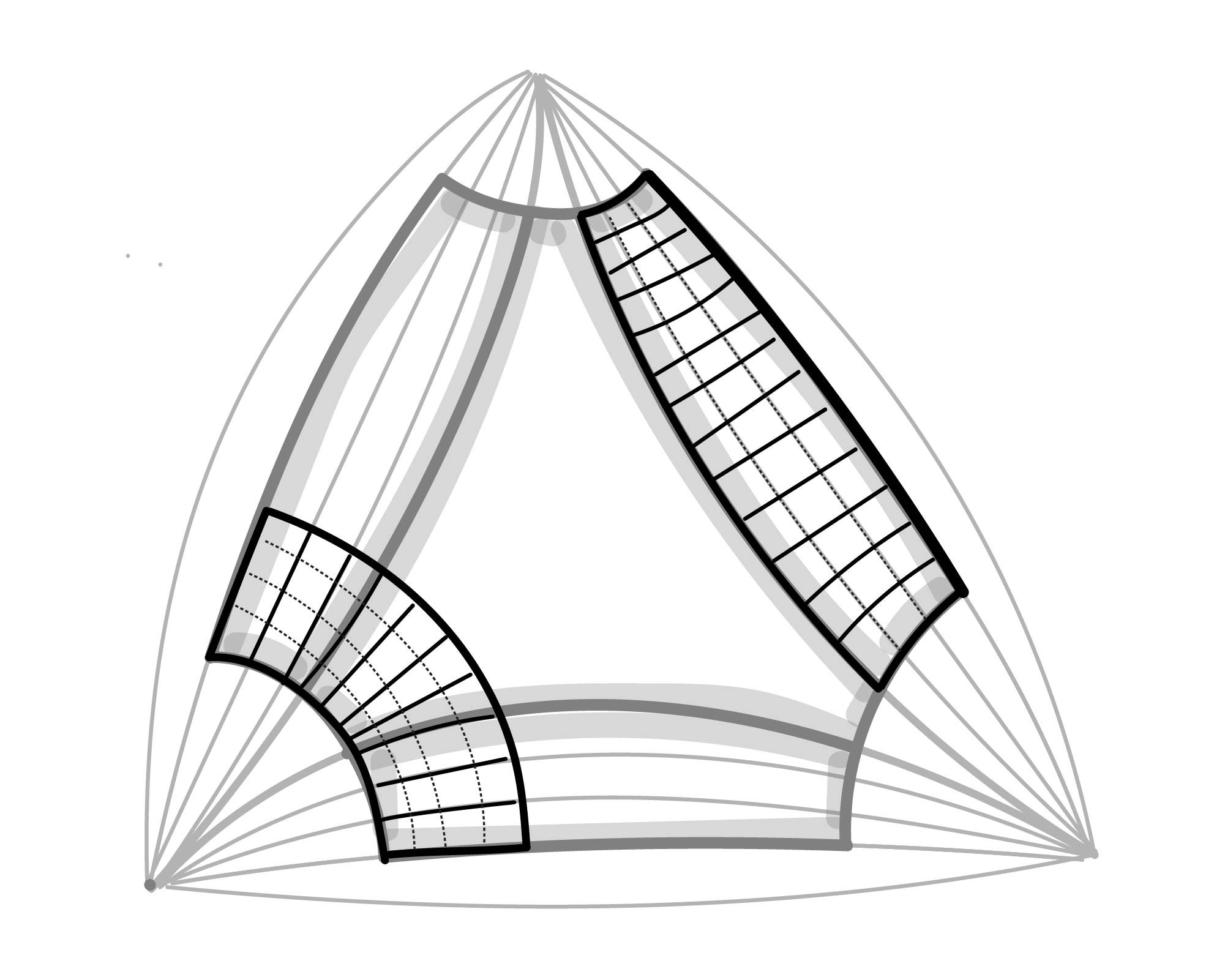}  
\put(60, 44){\textcolor{Black}{\small \contour{white}{$\RRR$}}}   
\put(65, 55){\textcolor{Black}{\small \contour{white}{$v$}}}   
% \put(, ){\textcolor{white}{\tiny \contour{}{$$}}}   
 %   \put( , ){\textcolor{}{$$}}  
 %   \put( , ){}  
      \end{overpic}
\caption{Reagions in $\QQQ_{j, h}^s$ on which $\xi_{j, h}^s$ is modified. }\label{fSupportForModification}
\end{figure}
\end{proof}

%  
%  \Qed{Asymptotic}

 We completed the proof of \Cref{Asymptotic}.   
  
  \section{Uniform asymptoticity}
 We have proved, in \Cref{Asymptotic}, that the limit of the Teichmuller ray $X_\infty\col \R \to \TT$ is asymptotic to the corresponding conformal grafting ray from the same base point $X_\infty(0)$, directly by constructing a quasi-conformal mapping between corresponding points on those rays. 
 
Utilizing this asymptotic property, we, in this section,  show a uniform asymptotic property for the pairs of a Teichmuller ray $X_i\col \R \to \TT$ and a corresponding grafting ray with the same base point $X_i (0)$, such that $X_i\col \R \to \TT$ limits to the Teichmüller ray $X_\infty\col \R \to \TT$ as $i \to \infty$.

Recall that the vertical measured foliation $V_i$ on $X(t_i)$ is normalized by scaling, so that $V_i$ has length one on the corresponding flat surface $E_i $.
 For each $i = 1, 2 \dots, $
 we let $$d_i = \frac{\length_{E_i} V_i}{\length_{\sigma_i} L_i} = \frac{1}{\length_{\sigma_i} L_i} \in \R_{>0}.$$
 Then $d_i \to d$ as $i \to \infty$, since $[E_i]$ converges to $[E_\infty]$ and accordingly $[V_i]$ converges to $[V_\infty]$ as $i \to \infty$. 
 We have the uniform asymptotic pairs of the Teichmüller ray and the grafting ray. 
  \begin{theorem}\Label{UniformAsymptoticity}
For every $\ep > 0$, there are $I_\ep > 0$ $s_\ep > 0$ such that, if $i > I_\ep$ and $s > s_\ep$, then
$$d_\TT(X_i(s), \gr_ {L_i}^{d_i \exp (s)} \sigma_i  ) < \ep.$$
\end{theorem}
The rest of this section is a proof of this theorem. 

 \subsection{Converecne of Euclidean polygonal structure}\Label{sEuclideanPolygonalStructuresConverge}
 
We first analyze the convergence of the Euclidean surfaces.
By the convergence of $\nu_i (E_i, V_i) \to (E_\infty, V_\infty)$ implies the following proposition. 
\begin{lemma}\Label{BilipschitzEuclideanTraintracks}
For every $\ep > 0$, there are $I_\ep > 0, J_\ep > 0$ and $J_i > 0$ with $J_i \to \infty$ as $i \to \infty$, such that,  
if $i > I_\ep$ and $J_\ep < j < J_i$, then,  
there is a $(1 -\ep, 1 + \ep)$-bilipschitz map 
$$\nu_{i, j}\col E_i \to  E_\infty$$ homotopic to $\nu_i$ such that 
 \begin{itemize}
 \item $\nu_{i, j}$ preserves singular points;
\item the inverse map $\nu_{i, j}^{-1}$ takes the tripods $\gam_1(j), \dots,  \gam_N(j)$ into tripods in singular leaves of the vertical foliation $V_i$;
\item  $\nu_{i, j}$ is  $(1- \ep, 1 + \ep)$-bilipschitz mapping in the vertical and horizontal distances. 
\end{itemize}
\end{lemma}
 By the second property of $\nu_{i, j}$, $E_i$ minus the $\nu_{i, j}^{-1}$-images of  $\gam_1(j), \dots,  \gam_N(j)$ can be decomposed into Euclidean rectangles along horizontal segments from the endpoints of the pullback tripods; then 
we obtain a traintrack structure $T_{i, j}$ on the complement which is isomorphic to $T_{\infty, j}$ as fat-traintracks, as in the construction as the traintrack structure $T_{\infty, j}$ on $E_\infty \minus \gam_1(j) \cup \dots \cup \gam_N(j)$.
 
 In  \S \ref{sAlmostConformalMapping}, 
 we constructed, from the traintrack structure $T_{\infty, j}$,  a decomposition $E_{\infty, j} = (\cup_{k = 1}^{N'} R_{j, k}) \cup (\cup_{h = 1}^{N} Q_{j, h})$ into rectangles $ R_{j, k}$ and hexagons $Q_{j, h}$.
 Similarly, we let $m_{i, j} > 0$ be the shortest width of the (rectangular) branches of $T_{i, j}$. 
By \Cref{BilipschitzEuclideanTraintracks}, for each $h = 1, \dots, N$, the inverse-image  $\nu_{i, j}^{-1} (\gam_h (j))$ is also a tripod embedded in a singular leaf of $V_i$. 
Let $Q_{i, j, h}$ be the hexagon which is the $(m_{i, j}/3)$-neighborhood of the tripod $\nu_{i, j}^{-1} (\gam_h (j))$   in the horizontal direction of $E_i$. 
Then, by removing the hexagonal part $Q_{i, j, 1} \cup \dots \cup Q_{i, j, N}$ from the rectangular branches of $T_{i, j}$, we obtain $2 m_{i, j}/ 3$- thiner rectangles  $R_{i, j, 1}, \dots, R_{i, j, N'}$. 
 We thus obtain a decomposition $E_{i, j}$  of $E_j$
$$(\cup_{k = 1}^{N'} R_{i, j, k}) \cup (\cup_{h = 1}^{N} Q_{i, j, h})$$ 
into the hexagonal  $Q_{i, j, 1}, \dots, Q_{i, j, N}$ and the rectangles $R_{i, j, 1}, \dots, R_{i, j, N'}$, which have disjoint interiors.

We can, by isotopy, modify the  $(1 - \ep, 1 + \ep)$-bilipschitz map $\nu_{i,  j}\col E_i \to E_\infty$ so that the decomposition $E_{i, j}$ to the decomposition $E_{\infty, j}$ by $\nu_{i, j}$, retaning the properties in \Cref{BilipschitzEuclideanTraintracks}. 
\begin{proposition}\Label{BilipschitzGraftedEuclideanTraintracks}
For every $\ep > 0$, there are $I_\ep > 0, J_\ep > 0$ and $J_i > 0$ with $J_i \to \infty$ as $i \to \infty$, such that  
if $i > I_\ep$ and $J_\ep < j < J_i$, then  for all $s > 0$. 
There is a $(1 -\ep, 1 + \ep)$-bilipschitz map 
$$\nu_{i, j}'\col E_i \to  E_\infty$$
homotopic to $\nu_i$ such that 
\begin{itemize}
\item $\nu_{i, j}'$ preserves the singular points; 
\item $\nu_{i, j}'$ induces an isomorphism between polygonal decompositions
$$E_{i, j} = (\cup_{k = 1}^{N'} R_{i, j, k}) \cup (\cup_{h = 1}^N Q_{i, j, h}) \to E_{\infty, j} = (\cup_{k = 1}^{N'} R_{j, k}) \cup (\cup_{h = 1}^N Q_{j, h});$$
and 
\item $\nu_{i, j}'$ is  $(1 + \ep)$-bilipschitz both in the vertical and horizontal directions. 
\end{itemize}
\end{proposition}

  Similar to $E_\infty(s)$ in \S\ref{sAsymptoticPropertyInLimit}, for each $s \geq  0$,
we let $E_i(s)$ be the marked flat structure on $S$ obtained by stretching $E_i$ by $\exp(s)$ only in the horizontal direction, so that $E_i(s)$ is conformally equivalent to $X_i(s)$.
Similar to $f_{\infty, s}\col E_\infty = E_\infty(0) \to E_\infty(s)$, 
we let $f_{i, s}\col E_i(0) \to E_i (s)$ denote this stretch map by $\exp(s)$ so that $f_{i, s}$ is the Teichmüller maping realizing the smallest quasi-conformal distortion from $E_i(0)$ to $E_i (s)$.

Then, by $f_{i, s}$, the polygonal decomposition $E_{i, j} = (\cup_k R_{i, j, k}) \cup (\cup_h Q_{i, j, h})$ descends to a polygonal decomposition of $E_i(s)$; we set 
$$E_{i, j} (s) = (\cup_{k = 1}^{N'} R_{i, j, k}^s) \cup (\cup_h^N Q_{i, j, h}^s).$$
Then, the $(1 -\ep, 1 + \ep)$-bilipschitz map  $$\nu_{i, j}'\col E_i \to  E_\infty$$  induces
$$\nu_{i, j}^s\col E_i(s) \to  E_\infty(s)$$ 
so that $f_{\infty, s} \circ\nu_{i, j} = \nu_{i, j}^s \circ f_{i, s}$.
Since the mappings $f_{i, s}$ and $f_{\infty, s}$ both stretch $E_i$ and $E_\infty$ by $\exp(s)$ in the horizontal direction,  $\nu_{i, j}^s$ retains the properties of $\nu_{i, j}'$, and we obtain the following corollary. 

\begin{corollary}\Label{BilipschitzGraftedEuclideanTraintracks}
Under the same assumption, for all $s > 0$, 
the mapping 
$$\nu_{i, j}^s\col E_{i, j}(s) \to  E_{\infty, j}(s)$$ is  a $(1 -\ep, 1 + \ep)$-bilipschitz map
such that 
\begin{itemize}
\item $\nu_{i, j}^s$ gives a polygonal isomorphism as smooth Euclidean traintracks,
$$E_{i, j} (s) = (\cup_k R_{i, j, k}^s) \cup (\cup Q_{i, j, h}^s) \to E_{\infty, j} (s)= (\cup_k R_{j, k}^s) \cup (\cup Q_{j, h}^s)$$
\item this induced isomorphism is $(1 + \ep)$-bilipschitz both in the vertical and horizontal directions. 
\end{itemize}
\end{corollary}

\subsection{Convergence of decompositions of hyperbolic surfaces}
In \S \ref{sEuclideanPolygonalStructuresConverge}, we constructed a Euclidean polygonal decomposition $E_{i, j}$ of $E_i$ which converges to the Euclidean polygonal decomposition $E_{\infty, j}$ of $E_\infty$ as $i \to \infty$.
In this subsection, we construct a corresponding polygonal decomposition of $\tau_{i, j}$ converging to the polygonal decomposition $\tau_{\infty, j}$.

By the convergence $\nu_i (\sigma_i, L_i) \to (\sigma_\infty, L_\infty)$, up to isotopy, implies the following Lemma. 
\begin{lemma}\Label{BilipschitzHyperbolicTraintracks}
For every $\ep > 0$, there are constants $I_\ep > 0, J_\ep > 0$ and a sequence $J_i > 0$ with $J_i \to \infty$ as $i \to \infty$, such that,  
if $i > I_\ep$ and $J_\ep < j < J_i$, then,  
there are
\begin{itemize}
\item a  $(1 -\ep, 1 + \ep)$-bilipschitz map 
$$\upsilon_{i, j}\col \sigma_{i, j} \to  \sigma_{\infty, j}$$ 
homotopic to the diffeomorphism $\nu_i$, and
\item an $\ep$-nearly straight traintrack $\tau_{i, j}$ on $\sigma_i$ combinatorially isomorphic to $\tau_{\infty, j}$,
\end{itemize}
such that 
\begin{itemize}
\item $\upsilon_{i, j}$  induces a  $(1 -\ep, 1 + \ep)$-bilipschitz isomorphism of traintrack neighborhoods $$\tau_{i, j} \to  \tau_{\infty, j},$$ and
\item  the $L_i$-weights of $\tau_{i, j}$   are $(1 -\ep, 1 +\ep)$-bilipschitz close to the $L_\infty$-weights of $\tau_{\infty, j}$ (on the corresponding branches). 
\end{itemize}
\end{lemma}
Recall, from \S \ref{sAlmostConformalMapping}, that  the polygonal decomposition  $\sigma_{\infty, j} = (\cup_k \RR_{j, k}) \cup (\cup \QQ_{j, h})$ carrying $L_\infty$ is constructed from the $\ep$-nearly straight traintrack $\tau_{\infty, j}$ so that it corresponds to $ E_{\infty, j} = (\cup_{k = 1}^{N'} R_{j, k}) \cup (\cup_{h = 1}^N Q_{j, h})$ carrying $V_\infty$. 

Recall that the Euclidean polygonal decomposition $E_{i, j} = (\cup_{k = 1}^{N'} R_{i, j, k}) \cup (\cup_h Q_{i, j, h})$ of $E_i$ carries the vertical foliation $V_i$. 
Then, we can similarly construct a corresponding polygonal decomposition $$\sigma_{i, j} = (\cup_{k = 1}^{N'} \RR_{i, j, k}) \cup (\cup_{h =1}^N \QQ_{i, j, h})$$ carrying $L_i$, where
 $\RR_{i, j, k}$ and $\QQ_{i, j, h}$ are rectangles and hexagons with horocyclic horizontal edges and with vertical edges in $L_i$, 
 such that 
\begin{itemize}
\item  $\sigma_{i, j} = (\cup_{k = 1}^{N'} \RR_{i, j, k}) \cup (\cup_{h =1}^N \QQ_{i, j, h})$ is combinatorially  isomorphic to $ E_{i, j} = (\cup_{k = 1}^{N'} R_{i, j, k}) \cup (\cup_{h = 1}^N Q_{i, j, h})$ by a marking-preserving homeomorphism from $\sigma_i$ to $E_i$;
\item moreover $\sigma_{i, j} = (\cup_{k = 1}^{N'} \RR_{i, j, k}) \cup (\cup_{h = 1}^N \QQ_{i, j, h})$ carries $L_\infty$ combinatorially in the same manner as  $ E_{i, j} = (\cup_{k = 1}^{N'} R_{i, j, k}) \cup (\cup_{h = 1}^N Q_{i, j, h})$ carries $V_\infty$, respecting the identification of  the geodesic measured lamination $L_i$ and the measured foliation $V_i$;
\item the union of the horizontal edges of $\RR_{i, j, k}$ and $\QQ_{i, j, h}$ is the union of (horocyclic) horizontal edges of $\tau_{i, j}$;
\end{itemize}

As the polygonal decomposition $\sigma_{i, j}$ is geometrically determined by the nearly-straight traintrack neighborhood $\tau_{i, j}$, \Cref{BilipschitzHyperbolicTraintracks} implies the following. 
\begin{proposition}\Label{BilipschitzHyperbolicPolygonalDecomposition}
For every $\ep > 0$, there are $I_\ep > 0, J_\ep > 0$ and $J_i > 0$ with $J_i \to \infty$ as $i \to \infty$, such that,  
if $i > I_\ep$ and $J_\ep < j < J_i$, then,  
there is a $(1 -\ep, 1 + \ep)$-bilipschitz map 
$$\upsilon_{i, j}'\col \sigma_{i, j} \to  \sigma_{\infty, j}$$ homotopic to $\upsilon_i$, such that 
\begin{itemize}
\item $\upsilon_{i, j}'$  induces a  $(1 -\ep, 1 + \ep)$-bilipschitz isomorphism between polygonal decompositions
$$\sigma_i = (\cup_{k = 1}^{N'} \RR_{i, j, k}) \cup (\cup_{h = 1}^N \QQ_{i, j, h}) \to \sigma_\infty = (\cup_{k = 1}^{N'} \RR_{j, k}) \cup (\cup_{h = 1}^N \QQ_{j, h}),$$
and
\item the $L_i$-weights of $(\cup_{k = 1}^{N'} \RR_{i, j, k}) \cup (\cup_{h = 1}^N \QQ_{i, j, h})$ are $(1 -\ep, 1 +\ep)$-bilipschitz close to the $L_\infty$-weights of $ (\cup_{k = 1}^{N'} \RR_{j, k}) \cup (\cup_{h = 1}^N \QQ_{j, h})$  on the corresponding horizontal edges of the polygonal decompositions.
\end{itemize}
 \end{proposition}
 
In \S \ref{sAlmostConformalMapping}, we showed that the grafting of $\sigma_\infty$ along $s L_\infty$ transforms the polygonal decomposition $\sigma_{\infty, j} = (\cup_{k = 1}^{N'} \RR_{j, k}) \cup (\cup \QQ_{j, h})$ to a polygonal decomposition of $\Gr_{s L_\infty} \sigma_\infty$ $$\Gr_{s L} \sigma_\infty = (\cup_{k = 1}^{N'} \RR_{j, k}^s) \cup (\cup_{h = 1}^N \QQ_{j, h}^s)$$ for $s \geq 0.$ 
Similarly, 
by the grafting of $\sigma_i$ along $s L_i \,(s \geq 0)$,  the polygonal decomposition $$\sigma_{i, j} = (\cup_{k = 1}^{N'} \RR_{i, j, k}) \cup (\cup_{h = 1}^N \QQ_{i, j, h})$$ induces a polygonal decomposition 
$$\Gr_{sL_i} \sigma_{i, j} = (\cup_{k = 1}^{N'} \RR_{i, j, k}^s) \cup (\cup_{h = 1}^N \QQ_{i, j, h}^s),$$ 
where $\RR_{i, j, k}^s$ is a rectangle obtained by grafting $\RR_{i, j, k}$ along the restriction of $L_i$ to $\RR_{i, j, k}$ and $\QQ_{i, j, h}^s$ is a hexagon obtained by grafting $\QQ_{i, j, h}$ along the restriction of $L_i$ to $\QQ_{i, j, h}$.
 
Then, since the way $\sigma_{i, j}$ carries $L_i$ geometrically converges to the way $\sigma_{\infty, j}$ carries $ L_\infty$, \Cref{BilipschitzHyperbolicPolygonalDecomposition} implies that the convergence of grafted decompositions. 

\begin{corollary}\Label{BilipschitzGraftedHyperbolicTraintracks}
Under the same assumption, for all $s > 0$, 
there is a $(1 -\ep, 1 + \ep)$-bilipschitz map 
$$\upsilon_{i, j}^s \col \Gr_{s L_i} \sigma_{i, j} \to \Gr_{s L_\infty} \sigma_{\infty, j}$$ homotopic to $\nu_i$, such that  $\upsilon_{i, j}^s$  induces a branch-wise $C^1$-smooth  $(1 -\ep, 1 + \ep)$-bilipschitz isomorphism
$$\sigma_i = (\cup_k \RR_{i, j, k}^s) \cup (\cup \QQ_{i, j, h}^s) \to \sigma_\infty = (\cup_k \RR_{j, k}^s) \cup (\cup \QQ_{j, h}^s).$$
 \end{corollary}

\subsection{Uniform quasi-conformal mappings}
By compositing the $C^1$-smooth bilipschitz mappings,  we obtain a desired quasi-conformal mapping with small distortion.
\proof[Proof of \Cref{UniformAsymptoticity}]
    For every $\ep > 0$,   there are $I_\ep > 0, J_\ep > 0$, $J_i > 0$ with $J_i \to \infty$ as $i \to \infty$ and $s_\ep > 0$, such that,
    if $i > I_\ep$ and $J_\ep < j < J_i$, and $s > s_\ep$, 
    combining the quasi-conformal mappings with small distortion  $$\nu_{i, j}^s\col E_{i, j}(s) \to  E_{\infty, j}(s)$$  in  \Cref{BilipschitzGraftedEuclideanTraintracks}, $$\upsilon_{i, j}^s \col \Gr_{s L_i} \sigma_{i, j} \to \Gr_{s L_\infty} \sigma_{\infty, j}$$ in  \Cref{BilipschitzGraftedHyperbolicTraintracks}, 
    $$\Phi_j^s\col \Gr_{L_\infty}^{ \exp(s)/d} \to E_\infty(s)$$ 
 in \Cref{qfExtension}, 
   we obtain a desired $(1 + \ep)$-quasiconformal mapping, 
  
   $$ E_i(s)  \xrightarrow{\nu_{i, j}^s}  E_\infty(s)  \xrightarrow{\Phi_j^s}   \Gr_{L_\infty}^{\exp s/d} \sigma_\infty(s)   \xrightarrow{(\upsilon_{i, j}^s)^{-1}}  \Gr_{L_i}^{ \exp s/ d} \sigma_i (s) $$

 (see \Cref{fCloseToTheLimit}).%which inducde the mapping the 
\begin{figure}[H]
\begin{overpic}[scale=.2%, grid,tics=10
] {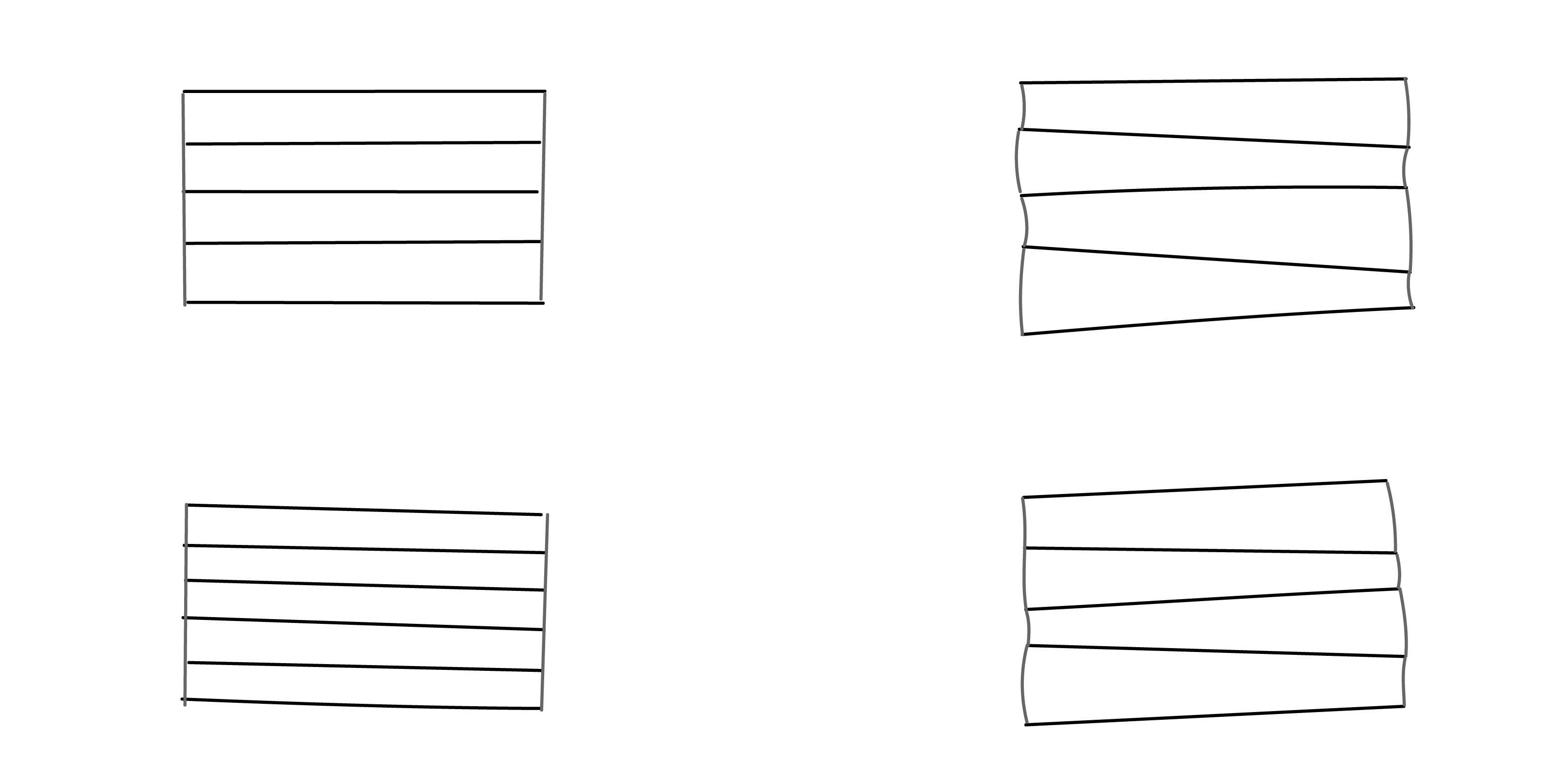} % fig ure file
 \put(75, 10 ){\textcolor{black}{\small \contour{white}{$\RR_{i, j, k}^s$}}}   
  \put(20, 10 ){\textcolor{black}{\small \contour{white}{$R_{i, j, k}^s$}}}   
 \put(75, 36 ){\textcolor{black}{\small \contour{white}{$\RR_{j, k}^s$}}}   
 \put(20, 36 ){\textcolor{black}{\small \contour{white}{$R_{j, k}^s$}}}   
 \put(48, 39 ){\textcolor{black}{\small \contour{white}{$\Phi_j^s$}}}   
 \put(16, 23 ){\textcolor{black}{\small \contour{white}{$\nu_{i, j}^s$}}}   
 \put(81, 23 ){\textcolor{black}{\small \contour{white}{$(\upsilon_{i, j}^s)^{-1}$}}}   
   \put(42, 37){\color{black}\vector(1,0){15}}
  \put(22, 21){\color{black}\vector(0,1){6}}
  \put(80, 27){\color{black}\vector(0,-1){6}}
        % \put(, ){\textcolor{}{\tiny \contour{white}{$$}}}   
 %   \put( , ){\textcolor{}{$$}}  
 %   \put( , ){}  
      \end{overpic}
\caption{The composition $(\upsilon_{i, j}^s)^{-1}\circ \Phi_j^s \circ \nu_{i, j}^s$ from $R_{i, j, k}^s$ to $\RR_{j, k}^s$.}\Label{fCloseToTheLimit}
\end{figure}
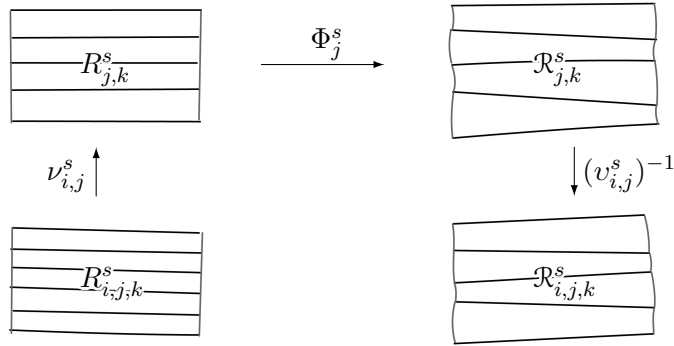
\Qed{UniformAsymptoticity}

\section{Uniform approximation of grafting rays by integral grafting}
Recall that $\sigma_i$ is a sequence of marked hyperbolic structures on $S$ and $\nu_i\col S \to S$ is a diffeomorphism  such that  $\nu_i (\sigma_i)$ converges to $\sigma_\infty \in \TT$ as $i \to \infty$.
Moreover,  $L_i$ is a maximal measured lamination on $\sigma_i$ such that $\nu_i (L_i)$ converges to the maximal measured lamination $L_\infty$ on $\sigma_\infty$ as $i \to \infty$. 

Gupta showed that every grafting ray is conformally well-approximated by a sequence of integral grafting toward infinity; see \cite[Lemma 6.19]{Gupta14}. 
In this section, using Gupta's idea in our setting, we show that the family of grafting rays $\gr_{L_i}^s \sigma_i \, (s \geq 0)$ is well-approximated by the integral graftings of $\sigma_i$, indeed,  in a uniform manner.

\begin{theorem}\Label{IntegralGrafting}
For every $\ep > 0$, there are $I_\ep > 0$ and $s_\ep > 0$ such that, if $i > I_\ep$ and $s > s_\ep$, then there is a multiloop $M = M_{i, s}$ with weights multiples of $2\pi$, such that 
$$d_\TT(\gr_ {L_i}^s (\sigma_i), \gr_{M_{i, s}}(\sigma_i)  ) < \ep,$$
where $d_\TT$ denotes the Teichmüller distance on $\TT$.
\end{theorem}

The rest of this section is the proof of  \Cref{IntegralGrafting},

\subsection{Uniform approximation of grafting lamination rays} 
By the convergence $\nu_i (\tau_i, L_i) \to (\tau_\infty, L_\infty)$ in $\TT \times \ML$, we can take a nearly-straight traintrack neighborhood of $L_i$ converging to a nearly-straight traintrack neighborhood of $L_\infty$: 

\begin{lemma}\Label{UniformlyAlmostIsomtricTraintracks}
For all $\ep > 0$, there are constants $I_\ep > 0$, $J_\ep > 0$, and  $J_i > 0 ~(i = 1, 2,\dots)$ with $J_i \to \infty$ as $i \to \infty$, such that,  if $i > I_\ep$ and $J_\ep < j <  J_i$, then we can take an $\ep$-nearly-straight traintrack neighborhood $\tau_{i, j} = \cup_{k = 1}^{N'} R_{i, j, k}$ of $L_i$ on $\sigma_i$ and an $\ep$-nearly straight traintrack neighborhood $\tau_{\infty, j} = \cup_{k = 1}^{N'} R_{j, k}$,
 such that  there is a $(1 - \ep_i, 1 + \ep_i)$-bilipschitz map $\nu_{i, j}\col \sigma_i \to \sigma_\infty$ homotopic to $\nu_i$ which takes $\tau_{i, j}$ to $\tau_{\infty, j}$ for some $0 < \ep _i < \ep$ with $\ep_i \searrow 0$ as $i \to \infty$. 
\end{lemma}
For $\ep > 0$, by applying \Cref{UniformlyAlmostIsomtricTraintracks},  we fix  $I_\ep > 0, J_\ep > 0, J_i > 0$, and a traintrack neighorhoods  $\tau_{i, j}$ and $\tau_{i, j}$. 
Since $L_i$ is maximal, let $\cup_{h = 1}^N \Delta_{i, j, h}$ denote the complement $\sigma_i \minus \tau_{i, j}$ where $\Delta_{i, j, h}$ are (triangular) connected components. 
Then, for  $i > I_\ep$ and $J_\ep < j < J_i$, we have the decomposition $$\sigma_{i, j} = (\cup_{k = 1}^{N'} R_{i, j, k}) \cup (\cup_{h =1}^N \Delta_{i, j, h})$$
of $\sigma_i$ into reclantular branches $R_{i, j, k}$ of $\tau_{i, j}$ with horocyclic horizontal edges and the triangular complements $\Delta_{i, j, h}$, so that $R_{i, j, k}$ and $\Delta_{i, j, h}$ have disjoint interiors.

\begin{lemma}[Lemma 6.14 in \cite{Gupta14}]\Label{ApproximationByMultiloop}
For all $i = 1, 2, \dots$ and $j = 1,2, \dots, \infty$, there is $K_{i, j} > 0$ such that for every measured lamination $L$ carried by $\tau_{i, j}$, there is a  multiloop  $M$ with weight $\Z_{> 0}$ carried by $\tau_{i, j}$ such that, for each branch $R$ of $\tau$, the difference of  the weights of $L$ and $M$ on $R$ is less than $K_{i, j}$. 
\end{lemma}

Recall that the traintrack structure of $\tau_{i, j}$ is identified with $\tau_{\infty, j}$ by the diffeomorphism $\nu_i \col \sigma_i \to \sigma_\infty$, and $K_{i, j}$ is independent on $i = 1, 2, \dots$. 
Moreover, there are only finitely many combinatorial types of $\tau_{i, j}$, we can take $K_{i, j} > 0$ independently of $i, j$: 
\begin{corollary}\Label{UniformApproximationByMultiloop}
 There is $K > 0$ such that, for every $i = 1, 2, \dots, \infty$ and  $j = 1,2,\dots$ and every measured lamination $L$ carried by $\tau_{i ,j}$, there is a multiloop  $M$  with weight $\Z_{> 0}$ carried by $\tau_{i, j}$ such that, for each branch $R$ of $\tau$, the difference of  the weights of $L$ and $M$ on $R$ is less than $K$. 
\end{corollary}

\begin{proposition}\Label{CarryingGeodesicMeausredLamintion}
Pick  arbitrary $\ep > 0$ and arbitrary $J > J_\ep$, so that, there is $I_{\ep, J} > 0$ such that,  if $i > 0$ is sufficiently large, then  $J_ \ep <  J < J_i$ and the $\ep$-nearly-straight traintrack neighborhood $\tau_{i, J}$ of $L_i$ exists by \Cref{UniformlyAlmostIsomtricTraintracks}.
For $s > 0$,  consider the $\Z_{> 0}$-weighted multiloop by applying \Cref{UniformApproximationByMultiloop} to $s L_i$, and letting $M_{i, j}^s$ be its geodesic representative on $\tau_i$. 
  
 Then,  there exist $s_{\ep, J} > 0$ and $I_{\ep, J} > 0$ such that, 
if  $i > I_{\ep, J}$ and $J_\ep  < j < J_i$,  and $s > s_{\ep, J}$,  then the geodesic multiloop $M_{i, J}^s$ on $\sigma_i$ is carried by the $\ep$-nearly straight traintrack $\tau_{i, J}$ on $\sigma_i$ (without isotopy). 
\end{proposition}
\begin{proof}
Fix $\ep > 0$, and let $\tau_{\infty, J}^\ep$ is an $\ep$-nearly straight traintrack on $\sigma_\infty$ carring $L_\infty$ obtained by \Cref{UniformlyAlmostIsomtricTraintracks}.
If  a neighborhood $U_\infty$ of $[L_\infty]$ in ${\rm PML}$ is suffiicently small, then  $\tau_{\infty, J}^\ep$ carries all geodesic measured laminations on $\sigma_\infty$ whose proejctive class contained in $U_\ep$ (without isotopy). 
For sufficiently large $i$, let  $\tau_{i, J}^\ep$  be an $\ep$-nearly-striaght traintrack on $\sigma_i$ carrying $L_i$ from \Cref{UniformlyAlmostIsomtricTraintracks}, so that $\nu_i (\sigma_i, \tau_{i, J})$  converges to $(\sigma_\infty, \tau_\infty)$ as $i \to \infty$.

Then, as the bilispchiz constants of $\nu_{i, j}$ converge to one, 
if the neighborhood $U_\ep$ of $[L_\infty]$ in $\PML$ is sufficientely small, then there is $I_{\ep, J} > 0$ such that the $\ep$-nearly straight traintrack $\tau_{i, J}^\ep$ contains all geodesic laminations whose projective classes are in $U_\ep$. 
Let $M_{i, J}^s$ be the geodesic multiloop on $\sigma_i$ so that the difference of the weights of $s L_i$ and $M_{i, J}^s$ is less than $K$ on each branch of $\tau_{i, J}$.

Therefore we can pick sufficiently large $s_{\ep, J} > 0$ so that  
 if $s > s_{\ep, J}$ and $i > I_\ep$, then the projective class of $M_{i, J}^s$ is contained in $U_\ep$. 
 Thus the geodesic multiloop $M_{i, J}^s$ is carried by $\tau_{i, J}^\ep$ (without isotopy). 
\end{proof}

\subsection{$2\pi$-grafting of nearly straight traintracks}

 Recall that, for $J_\ep <  J$, $\tau_{i, J}$ is the $\ep$-nearly-straight traintrack neighborhood of $L_i$ on $\sigma_i$, and 
  $$\sigma_{i, J} = (\cup_{k = 1}^{N'} R_{i, J, k}) \cup (\cup_{h = 1}^N \Delta_{i, J, h})$$ 
  is the traintrack decomposition of $\tau_{i, J}$ such that horizontal edges of $R_{i, J, k}$ are contained in leaves of  the horocyclice lamination $\lam_i$ of $(\sigma_i, L_i)$.

Consider the projective grafting of $\sigma_i$ along $s L_i~ (s \geq 0)$. 
Since the geodesic measured lamination $L_i$ is carried by $\tau_{i, J}$ , the above traintrack decomposition of $\sigma_i$ induces a traintrack decomposition of $\Gr_{s L_i} \sigma_i$ for $s \geq 0$, and we set 
$$\Gr_{s L_i} \sigma_i = (\cup_{k =1}^{N'} R^s_{i, J,k}) \cup (\cup_{h =1}^N \Delta_{i, J, h}),$$
where $R^s_{i, J, k}$ are grafting of $R_{i, J, k}$ along the restriction of $s L_i$ to the branch $R^s_{i, J, k}$. 
Recall that $R_{i, J, k}$ is foliated by nearly horocyclic foliation and mostly-straight foliation orthogonal to it, and the restriction of $sL_i$ to  $R_{i, J, k}$ is contained in the mostly-straight foliation. 
Therefore, as $L_i$ is irrational, this pair of orthogonal foliation induces a horocyclic and a mostly-straight foliation on $R^s_{i, J, k}$ so that the collapsing map $R^s_{i, J, k} \to R_{i, J, k}$ preserves those foliations.

By \Cref{CarryingGeodesicMeausredLamintion}, there are $I_{\ep, J} > 0$ and $s_{\ep, J}$, such that the traintrack  $\tau_{i, J}$ carries the geodesic representative of $M_{i, J}^s$ on $\tau_i$ for $i > I_{\ep, J}$ and $s >  s_{\ep, J}$. 
Then, similarly, the traintrack decomposition $\sigma_{i, J}$ induces a traintrack decomposition of the grafting of $\sigma_i$ along $M_{i, J}^s$ as follows.
Along each loop $m$ of $M_{i, J}^s$, the grafting $\Gr_{M_{i, J}^s}$ inserts an Euclidean cylinder of width $2\pi$ times the weight along $M_{i, J}^s$ (in $\Z_{\geq 0}$) in Thurston metric.
Then, for each branch $R_{i, J, k}$,   the restriction of $M_{i, J}^s$ to $R_{i, J, k}$ is a geodesic multi-arc connecting horizontal horocyclic edges. 
Then, let $R^{M_s}_{i,J,k}$ denote the grafting of $R_{i, J, k}$ along the multi-arc. 
In Thurston metric, along each geodesic arc of the multiarc,  the grafting inserts an Euclidean rectangle of length equal to the length of the arc and width equal to  $2\pi$ times the weight of the arc. 
Then the induced traintrack decomposition is 
$$\Gr_{M_{i, J}^s}\sigma_i = (\cup_{k = 1}^{N'} R^{M_s}_{i, J, k}) \cup (\cup_{h = 1}^N \Delta_{i, J, h}).$$

\subsection{Model Euclidean Traintracks}
Let $F_{i} (s L_i)$ be the {\it Euclidean traintrack} which represents the sum of the structure inserted to the hyperbolic traintrack $\tau_{i, j}$ by $\Gr_{s L_i}$.
Namely, 
\begin{itemize}
\item  $F_{i} (s L_i)$ is diffeomorphic to $\tau_{i, j} = \cup_{k = 1}^{N'} R_{i, j, k}$ as fat traintracks. 
\item the branch of  $F_{i} (s L_i)$ corresponding to $ R_{i, j, k}$ is a Euclidean rectangle of length equal to the length of $ R_{i, j, k}$ and width equal to the weight of $s L_i$ on $ R_{i, j, k}$.
\end{itemize}

Similarly, 
 let $F_{i} (M_{i, j}^s)$ be the Euclidean traintrack representing the sum of the structure inserted to $\tau_{i, j}$ along $M_{i, j}^s$.
 Namely,  
\begin{itemize}
\item  $F_{i} (M_{i, j}^s)$ is diffeomorphic to $\tau_{i, j} = \cup_{k = 1}^{N'} R_{i, j, k}$ as fat traintracks, and
\item if the branch of  $F_{i} (M_{i, j}^s)$ corresponding to $ R_{i, j, k}$, then it is a Euclidean rectangle of length equal to the length of $ R_{i, j, k}$ and width equal to the weight of $M_{i, j}^s$ on $ R_{i, j, k}$.
\end{itemize} 
 
Each branch of the grafted train track $\Gr_{s L_i}\tau_{i, j}$ is foliated by mostly-horocyclic foliation and mostly-straight foliation orthogonal to it. 
Let 
$$\xi_{s L_i}\col \Gr_{s L_i}\tau_{i, j} \to F_{i, j} (s L_i)$$ be the straightening mapping defined similar to the proof of $\Cref{UniformizaingGraftedRectangle}$ using the mostly-horocyclic foliation and mostly-straight foliation orthogonal to it.
Namely, 
\begin{itemize}
\item $\xi_{s L_i}$ takes horizontal foliation of $\Gr_{s L_i}\tau_{i j}$ to the horizontal foliation of $F_{i, j} (M_{i, j}^s)$, and
\item  $\xi_{s L_i}$ is linear on each vertical edge of $\Gr_{s L_i}\tau_{i, j}$ with respect to vertical distance. 
\end{itemize}

Then, by the construction of $ F_{i, j} (s L_i) $ we have the following.  
\begin{proposition}\Label{AlmostIsometricEuclideanizationForLamination}
For every $\ep > 0$, there are $I_\ep > 0, J_\ep > 0, s_\ep> 0$ such that, if $i > I_\ep, s > s_\ep$, then
$$\xi_{s L_i}\col \Gr_{s L_i}\tau_i \to F_{i, J_\ep} (s L_i) $$ is a $(1 - \ep, 1 + \ep)$-bilipschitz homeomorphism. 
\end{proposition}

\begin{proof}
Similar to the proof of  \Cref{UniformizaingGraftedRectangle}, for every $\ep > 0$, given sufficiently large $J_\ep > 0$, one can prove the derivatives of $\xi_{M_{i, j}^s}$ in both horizontal and vertical directions are $\ep$-close to one. This implies the assertion. 
\end{proof}

 As $\Gr_{M_{i, j}^s}$ inserts to each branch $R_{i, j, k}$ of $\tau_{i, j}$, Euclidean rectangles rectangles along the geodesic arcs of  $M_{i, j}^s | R_{i, j, k}$.
 The grafted branches $R_{i, j, k}^{M^s}$ have horizontal foliations obtained by gluing the obvious horizontal foliation of the rectangles and the mostly-horocyclic foliation of $R_{i, j, k}$.
 Then,  similar to $\xi_{s L_i}$,
let 
$$\xi_{M_{i, j}^s}\col \Gr_{M_{i, j}^s}\tau_i \to F_{i, j} (M_{i, j}^s)$$ be the straightening mapping defined similar to $\zeta_{j, k}^s$ in the proof of $\Cref{UniformizaingGraftedRectangle}$.

\begin{proposition}\Label{AlmostIsometricEuclideanizationForMultiloop}
 For every $\ep > 0$, there are $I_\ep > 0, J_\ep > 0, s_\ep> 0$ such that, if $i > I_\ep, s > s_\ep$, then
$$\xi_{M_{i, J_\ep}^s}\col \Gr_{M_{i, j}^s}\tau_i \to F_{i, J_\ep} (M_{i, J_\ep}^s) $$ is a $(1 - \ep, 1 + \ep)$-bilipschitz mapping. 
\end{proposition}
Recall that the boundary of the traintrack $\tau_{i, j}$ on $\sigma_i$ is identified both with the boundary of $F_{i, j}(s L_i)$ by $\xi_{s L_i}$and the boundary of $F_{i, j}(M_{i, j}^s)$ by $\xi_{M_{i, j}^s}$. 
Thus we have a canonical ``identity" mapping $\bdr \zeta_{i, j}^s$ from  $ \bdr F_{i, j} (s V)$ to $\bdr F_{i, j} (M_{i, j}^s)$ by composing those idnetifications. 
With respect to this identification,  endpoints of horizontal leaves of  $F_{i, j}(s L_i)$ coincide with endpoints of horizontal leaves of $F_{i, j}(M_{i, j}^s)$, since the constructions of $\Gr_{M_{i, j}^s}\tau_i $ and $\Gr_{s L_i}\tau_i$ preserve horizontal leaves. 
  
  Therefore,  we can finally define a $C^1$-diffeomorphism $\zeta_{i, j}^s\col  F_{i, j} (s L_i) \to  F_{i, j} (M_{i, j}^s)$ so that
\begin{itemize}
\item $\zeta_{i, j}^s$ coincides with $\bdr \zeta_{i, j}^s$ on the boundary of $F_{i, j} (s V)$, and 
\item  $\zeta_{i, j}^s$ linear on each horizontal leaf of $F_{i, j}(s L_i)$ with respect to arc length (\Cref{fPiecewieLinearExtension}).
\end{itemize}
\begin{proposition}\Label{AlomstConformalPLMapping} 
 For every $\ep > 0$, if $J > 0$ is sufficient large, then  there are $I_\ep > 0, s_\ep > 0$ such that,  if $i > I_\ep$ and $s > s_\ep$, then the piecewise $C^1$-diffeomorphism
 $\zeta_{i, j}^s\col F_{i, j}(s L_i) \to F_{i, j}(M_{i, j}^s)$ is a $(1 - \ep, 1 + \ep)$-bilipschitz map.   (\Cref{fPiecewieLinearExtension}.)
\end{proposition}
\begin{proof}
For an arbitaray branch $R_{i, j, k}$ of $\tau_{i, j}$, let  $R_L$ and $R_M$ be its corresponding branches of  $F_{i, j}(s L_i)$ and $F_{i, j}(M_{i, j}^s)$, respectively. 
Then the width of $R_L$ is the weight of $s L_i$ on $R_{i, j, k}$, and the width of $R_M$ is the weight of $M_{i, j}^s$ on $R_{i, j, k}$.
As $s_\ep > 0$ is sufficiently large, the ratio of the width of $R_L$ and $R_M$ is $\ep$-close to one by \Cref{UniformApproximationByMultiloop}. 
Therefore, under the assumption of the assertion,  $\zeta_{i, j}^s$ is $(1 - \ep, 1 + \ep)$-bilipschitz in the horizontal direction. 

By the definition of $F_{i, j}(s L_i)$ and $F_{i, j}(M_{i, j}^s)$, the lengths of the corresponding branches are the same. 
Since $\tau_{i, j}$ are sufficiently straight,   $\zeta_{i, j}^s$ is $(1 - \ep, 1 + \ep)$-bilipschitz in the vertical direction as well. 

Since $\zeta_{i, j}^s\col F_{i, j}(s L_i) \to F_{i, j}(M_{i, j}^s)$ is a $(1 - \ep, 1 + \ep)$-bilipschitz in both vertical and horizontal direction, the proposition follows.
\end{proof}

\begin{figure}
\begin{overpic}[scale=.15%, grid,tics=10
] {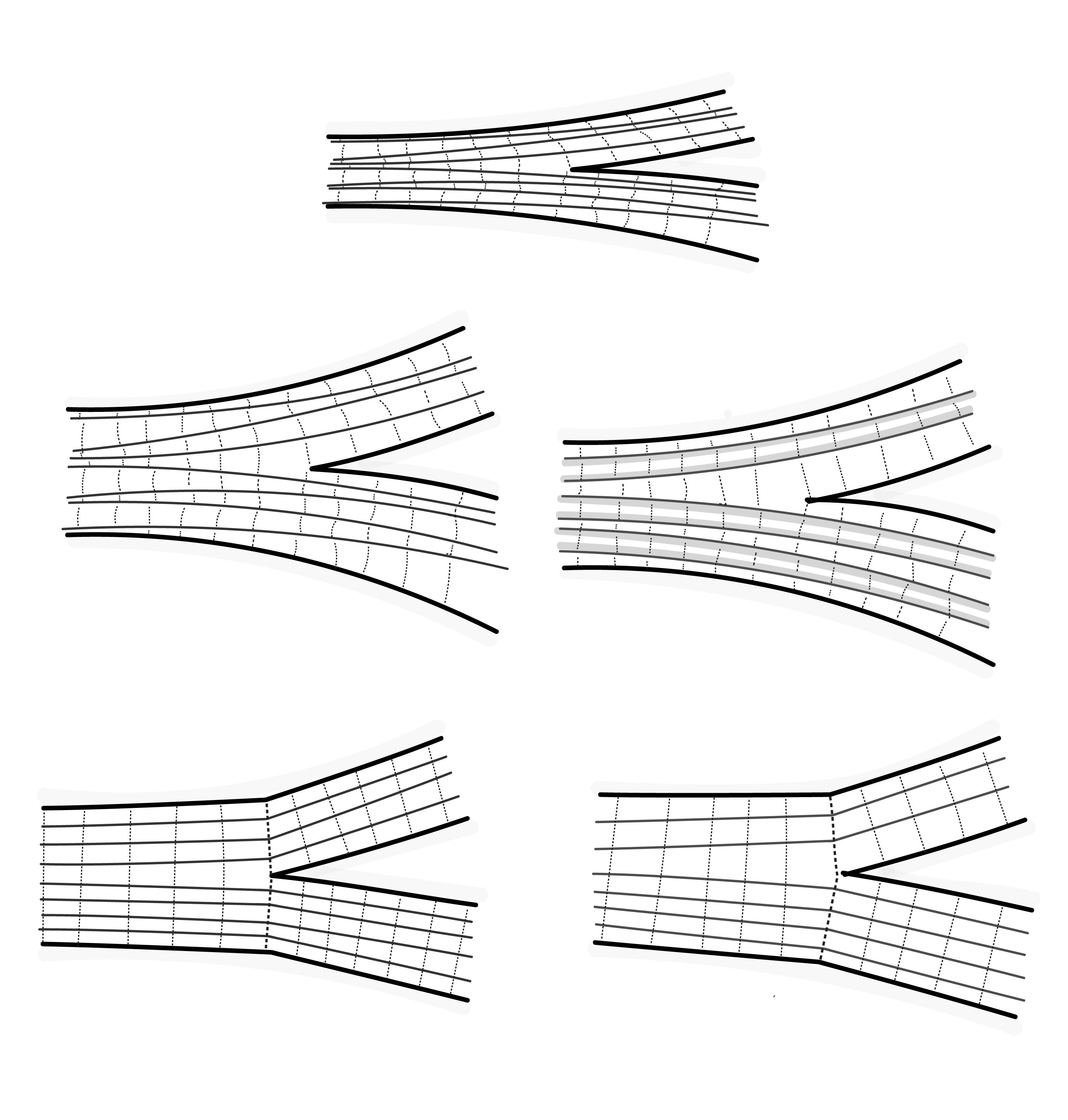} % figure file
 \put(60, 64){\textcolor{black}{\small \contour{white}{$\Gr_{M_{i, j}^s} \tau_i$}}}   
 \put(53, 10){\textcolor{Black}{\small \contour{white}{$F_{i, j}(M_{i, j}^s)$}}}   
 \put(60, 38){\textcolor{Black}{\small \contour{white}{$\xi_{M_{i, j}^s}$}}}   
 \put(46, 26){\textcolor{Black}{\small \contour{white}{$\zeta_{i, j}^s$}}}   
  \put(12, 39){\textcolor{Black}{\small \contour{white}{$\xi_{s L_i}$}}}   
  \put(10,66){\textcolor{black}{\small \contour{white}{$\Gr_{s L_i} \tau_i$}}}  
 \put(4, 11){\textcolor{Black}{\small \contour{white}{$F_{i, j}(s L_i)$}}}   
 %\put(20, 25){\textcolor{Green}{\small \contour{white}{$F_{i, j}(s L_i)$}}}   
  \put(50, 77){\color{black}\vector(1,-1.8){8}}
    \put(35,78){\color{black}\vector(-1,-1){8}}
  \put(45,22){\color{Black}\vector(1,0){7}}
\put(20, 48){\color{black}\vector(0,-1){14}}
\put(70, 43){\color{Black}\vector(0,-1){10}}
\put(45, 95){\textcolor{black}{\small \contour{white}{$\sigma_i$}}}   
\put(33, 89){\textcolor{black}{\small \contour{white}{$\tau_i$}}}   
 %\put(20, 19){\textcolor{Green}{\tiny \contour{white}{$v$}}}
% \put(, ){\textcolor{white}{\tiny \contour{}{$$}}}   
 %   \put( , ){\textcolor{}{$$}}  
 %   \put( , ){}  
      \end{overpic}
\caption{Composing branchewise $C^1$-smooth  nearly-isometric bilipschitz mappings.}\label{fPiecewieLinearExtension}
\end{figure}
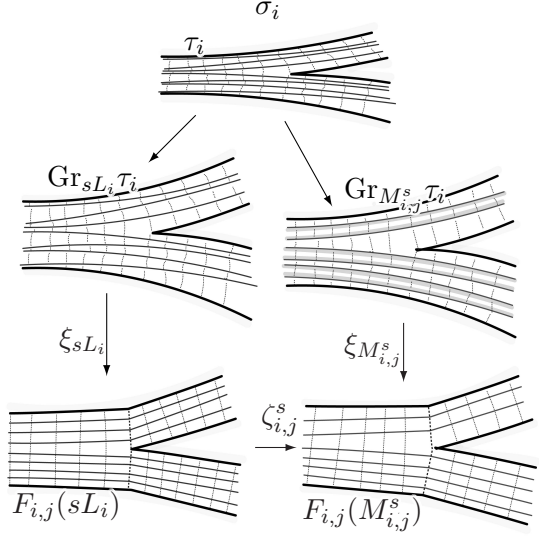
   
%\end{proof}
\Qed{AlomstConformalPLMapping}

\begin{corollary}  
 For every $\ep > 0$, if $j > 0$ is sufficient large, then  there are $I_\ep > 0, s_\ep > 0$ such that,  if $i > \ep$ and $s > s_\ep$, then
the mapping $\xi_{s M_s}^{-1} \circ \zeta_{i, j, s} \circ \xi_{s L_i}$ is a $(1 + \ep)$-quasiconformal mapping from $$\Gr_{M_s}\tau_i  \to \Gr_{s V}\tau_i$$ which is the identity on the boundary. 
\end{corollary}

\begin{proof}
By   \Cref{AlomstConformalPLMapping},  \Cref{AlmostIsometricEuclideanizationForLamination},  \Cref{AlmostIsometricEuclideanizationForMultiloop}, 
under the assumotion of the corollary, the mappings $\xi_{s M_s}^{-1}$, $\zeta_{i, j, s}$ and $\xi_{s L_i}$ are all branch-wise $C^1$-smoooth $(1 - \ep, 1 + \ep)$-bilipschitz mappings. 
Therefore, the assertion follows immediately
\end{proof}

We completed the proof of \Cref{IntegralGrafting}.
%\Qed{ApproximtingByMulitoopGraftring}
\section{Proof of the main theorem}\Label{MainProof}

In this section, we prove our main theorem. 
\begin{theorem}\Label{InfiniteIsomodromicPairs}
Let $X, Y$ be distinct Riemann surface structures in $\TT \cup \TT^\ast$.
Then, there is an infinite sequence $(C^X_j, C^Y_j)_{j = 1}^\infty \in \BB$ of distinct pairs such that $\psi(C^X_j) = X$ and $\psi(C^Y_j) = Y$ for all $j =1, 2, \dots$
\end{theorem}

We prove \Cref{InfiniteIsomodromicPairs} by induction. 
Suppose that we have $n$ pairs 
$$(C^X_1, C^Y_1), \dots, (C^X_{n}, C^Y_{n})$$ in
 $\Psi^{-1}(X, Y)$.
 Then we shall find a new pair $(C^X_{n +1}, C^Y_{n + 1})$ in  $\Psi^{-1}(X, Y)$.
Then, for each $j = 1, \dots, n$,  there are bounded open connected neighborhoods $U_j$ of $(C^X_j, C^Y_j)$ in $\BB$ and $W_j$ of $(X, Y)$ in $(\TT \cup \TT^\ast)^2 \minus \Delta$, such that 
the restriction of $\Psi$ to $U_j$ is a finite branched covering map onto $W_j$ (\cite[Theorem A]{Baba_23}). 

  Let $W$  be the (open) connected component of the intersection $ W_1 \cap W_2 \cap \dots \cap W_n$ containing $(X, Y)$.
Then it suffices to show the following.   
 \begin{proposition}\Label{ANewPair}
 There is $(C, D) \in \mathcal{B}$ such that $(\psi(C), \psi(D))$ is in $W$ and $\Hol(C) = \Hol(D) \not\in \cup_{j= 1}^n \Hol (U_j)$. 
\end{proposition}
Indeed, if we find such a pair $(C, D)$, then  we take a path $(X_t, Y_t), t \in [0,1]$ in $W$ connecting $(\psi(C), \psi(D))$ to $(X, Y)$. 
 By the completeness of $\Psi$, we can take a lift $(C_t, D_t), t \in [0,1]$ of $(X_t, Y_t)$ to $\BB$ such that
 \begin{itemize}
 \item $(C_0, D_0) = (C, D)$, and 
 \item $(\psi(C_1), \psi(D_1)) = (X, Y)$ (\Cref{fLiftingPath}). 
\end{itemize}
\begin{claim}
The pair $(C_1, D_1)$ in $\Psi^{-1}(X, Y)$ is different from all given $n$ pairs \linebreak[2] $(C^X_1, C^Y_1), \dots, (C^X_{n}, C^Y_{n})$.
\end{claim}
 \begin{proof}
Suppose, to the contaray, that $(C_1, D_1) = (C^X_j, C^Y_j)$ for some $j \in \{ 1, \dots, n\}$. 
Then, the lifted path $(C_t, D_t), t  \in [0,1]$ is entirely contained in $U_j$, since $\Psi_j\col U_j \to W_j$ is a finite branched covering map and $W (\subset W_j)$ contains the path $(X_t, Y_t),\, t \in [0,1]$.
Accordingly $\Hol (C_t) = \Hol (D_t), t \in [0,1]$  is entirely contained in $\Hol (U_j)$. 
 In particular, the initial holonomy $\Hol (C_0) = \Hol (D_0) = \Hol (C) = \Hol (D)$ is in $\Hol (U_j)$.
 This contradicts \Cref{ANewPair}. 
 Therefore, we conclude that $(C_1, D_1)$ is indeed a new pair in  $\Psi^{-1}(X, Y)$ distinct from the given n pairs.
 \begin{figure}
\begin{overpic}[scale=.23%, grid,tics=10
] {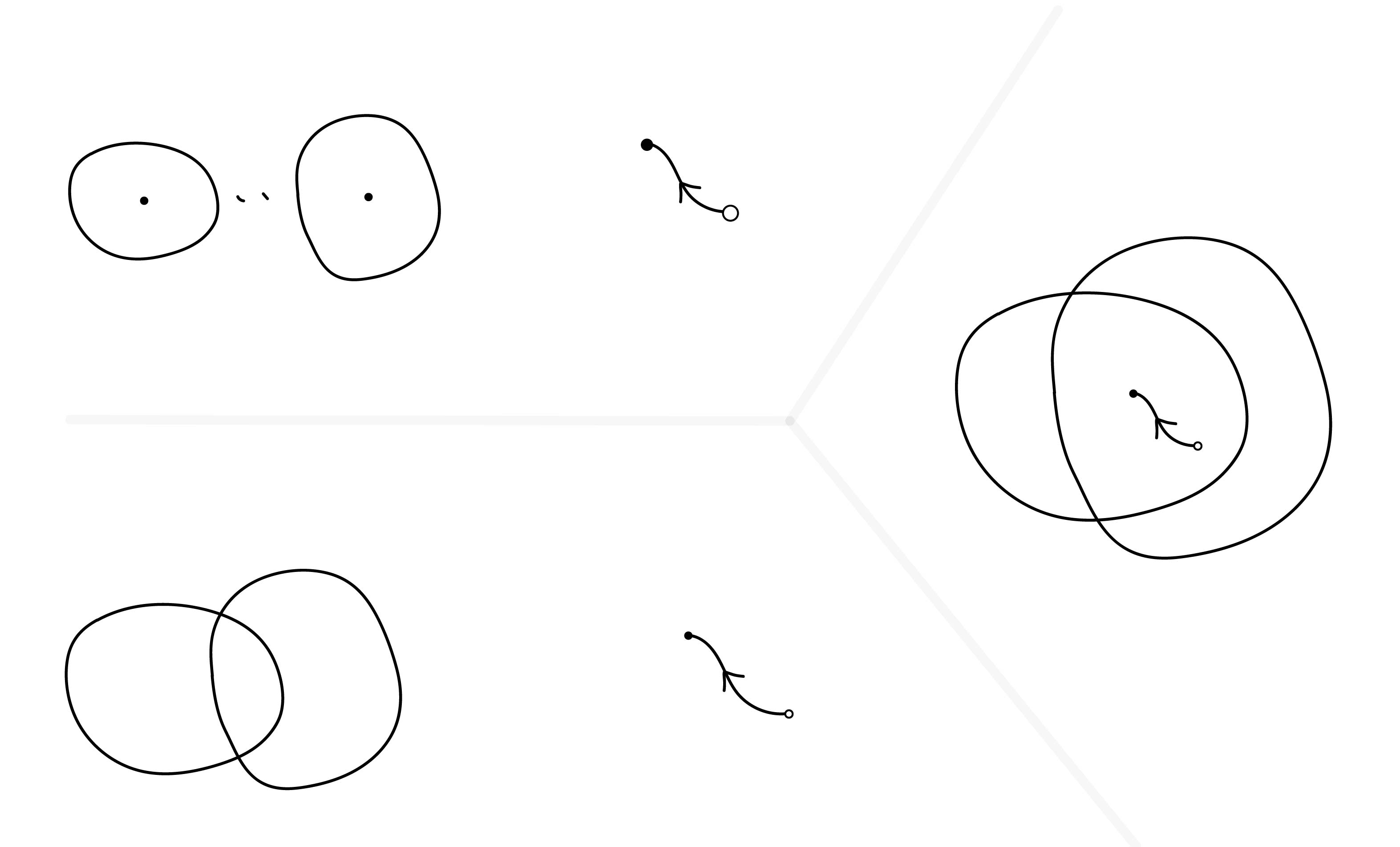} % figure file
  \put(5 , 55){\contour{white}{ $\BB$}} 
  \put(6 , 46){\contour{white}{\small $U_1$}} 
    \put(6 , 25){\contour{white}{ $\rchi$}}   
    \put(6 , 12){\contour{white}{\small $\Hol U_1$}} 
    \put(18 , 16){\contour{white}{\small $\Hol U_n$}} 
       \put(23 , 48){\contour{white}{\small $U_n$}} 
  \put(75 , 52){\contour{white}{$(\TT \cup \TT^\ast)^2 \minus \Delta$}}   
    \put(69 , 33){\contour{white}{\small $W_1$}}   
    \put(87 , 38){\contour{white}{\small $W_n$}} 
    \put(80 , 25){\contour{white}{\small $W$}} 
  \put(76 , 34){\contour{white}{\small $(X, Y)$}} 
  \put(50 , 42){\contour{white}{\small $(C, D)$}} 
  \put(40 , 52){\contour{white}{\small $(C_1, D_1)$}} 
  \put(53 , 13){\contour{white}{\small $\Hol(C_t, D_t)$}} 
     %arrow
   \put(31,35){\color{black}\vector(0,-1){10}}
 \put(56 , 36){\contour{white}{ $\Psi$}} 
     \put(55,35){\color{black}\vector(1,0){8}}
 \put(32 , 29){\contour{white}{$\Hol$}} 
 %   \put( , ){\textcolor{}{$$}}  
 %   \put( , ){}  
      \end{overpic}
\caption{Lifting the path $(X_t, Y_t)$.}\Label{fLiftingPath}
\end{figure}
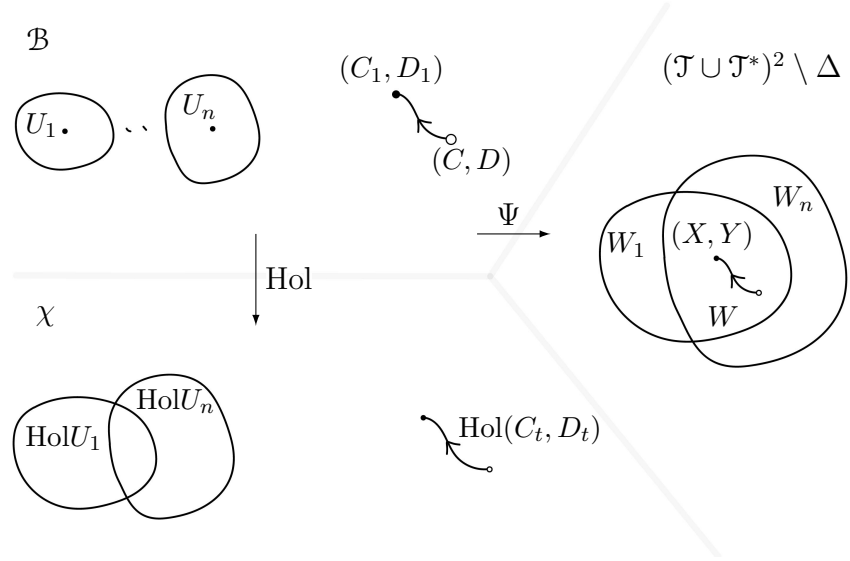
\end{proof} 

We prove \Cref{ANewPair} in the remainder of \S \ref{MainProof}.

\subsection{When the orientations of $X$ and $Y$ are the same}
\Label{sOrientationsCoincide}
In this subsection, supposing that the orientation of $X$ coincides with that of $Y$, we prove \Cref{ANewPair}. 
We, in addition, assume that  $X, Y \in \TT$, and the proof in the case  $X, Y \in \TT^\ast$ is essentially the same. 

Pick a sufficiently small $\ep> 0$ so that $W$ contains the product of the $\ep$-negihborhood of $X$ and the $\ep$-neighborhood of $Y$  in $\TT$ w.r.t. the Teichmüller metric.

There is a unique Teichmüller geodesic passing $X$ and $Y$. 
By perturbing it, we obtain  a ``generic'' Teichmüller geodesic $R\col \R \to \TT$ passing through the $\ep/3$-neighborhood of $X$ and the $\ep/3$-neighborhood of $Y$ in the Teichmüller metric, such that 
\begin{itemize}
\item its corresponding quadratic differential $q$ has only simple zeros, and
\item the projection of the ray $R(-\infty, 0]$ towrad $-\infty$ is dense in the moduli space $\MM$ of Riemann surfaces (\Cref{fTeichmuellerRayAndItsProjection}). 
\end{itemize}
Let $V$ denote the vertical (singular) measured foliation of $R$. 
Since $q$ has only simple zeros, each singular point of $V$ has three prongs.

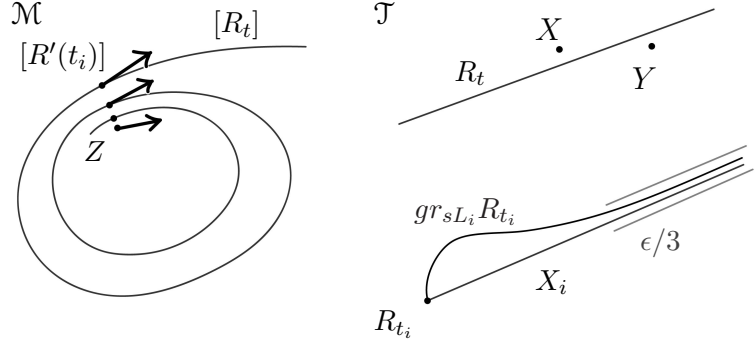
\begin{figure}
\begin{overpic}[scale=.20%, grid,tics=10
] {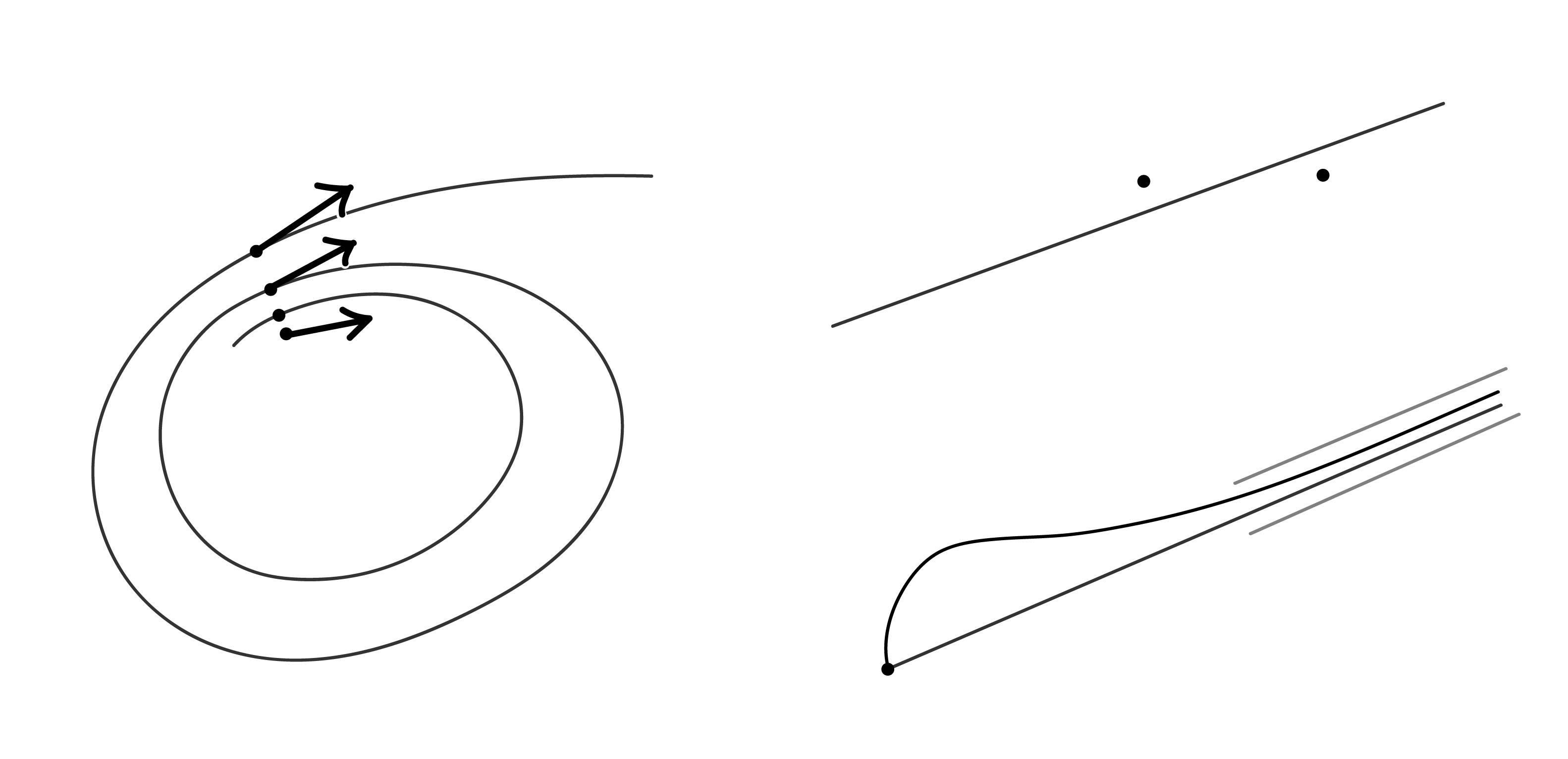} % figure file
 \put(70, 38 ){$X$}  
 \put(82, 32 ){$Y$}  
  \put(55 , 16){\textcolor{Black}{\contour{white}{\small $gr_{s L_i} R_{t_i}$} }} 
%  \put(70, 2){\textcolor{Gray}{$s_\ep$}} 
  \put(83, 12){\textcolor{darkgray}{\small $\ep/3$}} 
    \put(14, 23){\textcolor{black}{\small \contour{white}{$Z$}}} 
     \put(6, 35){\textcolor{black}{\small \contour{white}{$[R'(t_i)]$}}} 
        \put(50, 2){\textcolor{black}{\contour{white}{\small $R_{t_i}$}}} 
        \put(70, 7){\textcolor{black}{\contour{white}{\small $X_i$}}} 
      \put(30, 39){\textcolor{black}{\small $[R_t]$}} 
        \put(60, 33){\textcolor{black}{\small $R_t$}} 
    \put(5, 40){\textcolor{black}{$\mathcal{M}$}} 
      \put(50, 40){\textcolor{black}{$\TT$}} 
 %   \put( , ){\textcolor{}{$$}}  
 %   \put( , ){}  
      \end{overpic}
\caption{A generic Teichmüller geodesic passing close to $X$ and $Y$.}\Label{fTeichmuellerRayAndItsProjection}
\end{figure}

Let $[R(t)] \in \MM$ denote the unmarked Riemann surface structure of $R(t)$.
By the density of $[R (-\infty, 0] ]$, we can pick a sequence $0> t_1 > t_2 > \dots$ diverging to $-\infty$ such that 
\begin{itemize}
\item its unmarked sequence $[R(t_i)]$ converges to  $Z \in \MM$, and 
\item the tangent vector $[R'(t_i)]$ also converges in the unite tangent space $T^1 \MM$ at $Z$ as $i \to \infty$ (\Cref{fTeichmuellerRayAndItsProjection}, left).
\end{itemize}

For each $i = 1, 2,\dots$, let $\sigma_i$ denote the marked hyperbolic structure on $S$ corresponding to the marked Riemann surface $R(t_i)$ by the uniformization theorem. 
Let $L_i \in \ML$ denote the measured geodesic lamination on $\sigma_i$ representing the vertical measured foliation $V$.
 Let $\gr_{L_i}^{t} \sigma_i \in \TT \,(t  \geq 0)$ be the conformal grafting ray from $\sigma_i$ along $L_i$.

For each $i = 1, 2, \dots$, define $R_i \col \R \to \TT$ by $R_i(s) = R(t_i + s)$,  the reprametrization of the Teichmüller geodesci $R$ by shifting the base point backward to $R(t_i)$. 
 
By \Cref{UniformAsymptoticity}, for every $\ep > 0$,  there are $I_\ep > 0$ and  $s_\ep > 0$ such that, if $i > I_\ep$, then
$$d_\TT(R_i(s), \gr_{L_i}^{d_i \exp (s)} \sigma_i ) < \ep/3$$
for all $s > s_\ep$.
Since $R_i$ passes through the $\frac{\ep}{3}$-neighborhoods of $X$ and $Y$, 
we may in addtion assume that, if $i > I_\ep$, then $\gr_{V_i}^t \sigma_i$ passes through the $\frac{2}{3}\ep$-neighborhood of $X$ and  the $\frac{2}{3}\ep$-neighborhood of $Y$ in $\TT$.
Thus there are  $s_X^i, s_Y^i > s_\ep$, such that 
$$d_\TT(X, \gr_ {L_i}^{d_i \exp (s_X^i)} \sigma_i  ) < 2\ep/3,$$
and
 $$d_\TT(Y, \gr_ {L_i}^{d_i \exp (s_Y^i)} \sigma_i  ) < 2\ep/3.$$

By \Cref{IntegralGrafting},
there are $s_\ep > 0$ and $I_\ep > 0$ such that, if $s > s_\ep$ and $i > I_\ep$, then there is a geodesic multi-loop $M_{i, s}$ on $\tau_i$  with weights multiples of $2\pi$ satisfying
 $$d_\TT(\gr_{L_i}^{d_i \exp(s)}(\sigma_i),  \gr_{M_{i, s}}(\sigma_i)) < \frac{\ep}{3}.$$
 
 As $t_i \to - \infty$ as $i \to \infty$, we have  $t_i < - s_\ep$ for sufficiently large $i$.
 Thus, by the above inequality, there are multiloops $M_X = M_{X, i}$ and $M_Y = M_{Y, i}$ on $S$ with weight in $2\pi \Z_{> 0}$ such that 
 $$d_\TT(\gr_{L_i}^{d_i \exp(s_X^i)}(\sigma_i),  \gr_{M_X}(\sigma_i)) < \frac{\ep}{3}$$ and
 $$d_\TT(\gr_{L_i}^{d_i \exp(s_Y^i)}(\sigma_i),  \gr_{M_Y}(\sigma_i)) < \frac{\ep}{3}$$ 
 By combining the inequalities above,  the triangle inequality implies
 $$d_\TT(X,  \gr_{M_{X, i}}(\sigma_i)) < \ep$$ and
 $$d_\TT(Y,  \gr_{M_{Y, i}}(\sigma_i)) < \ep.$$ 
(See \Cref{fTeichmullerRayAndGrafting}.)
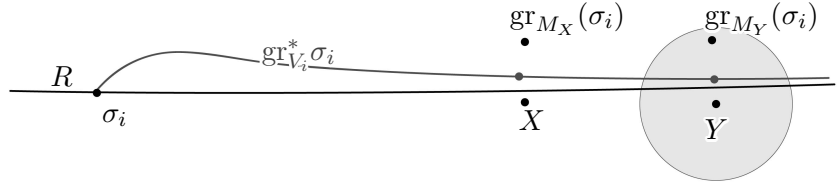
\begin{figure}
\begin{overpic}[scale=.22%, grid,tics=10
] {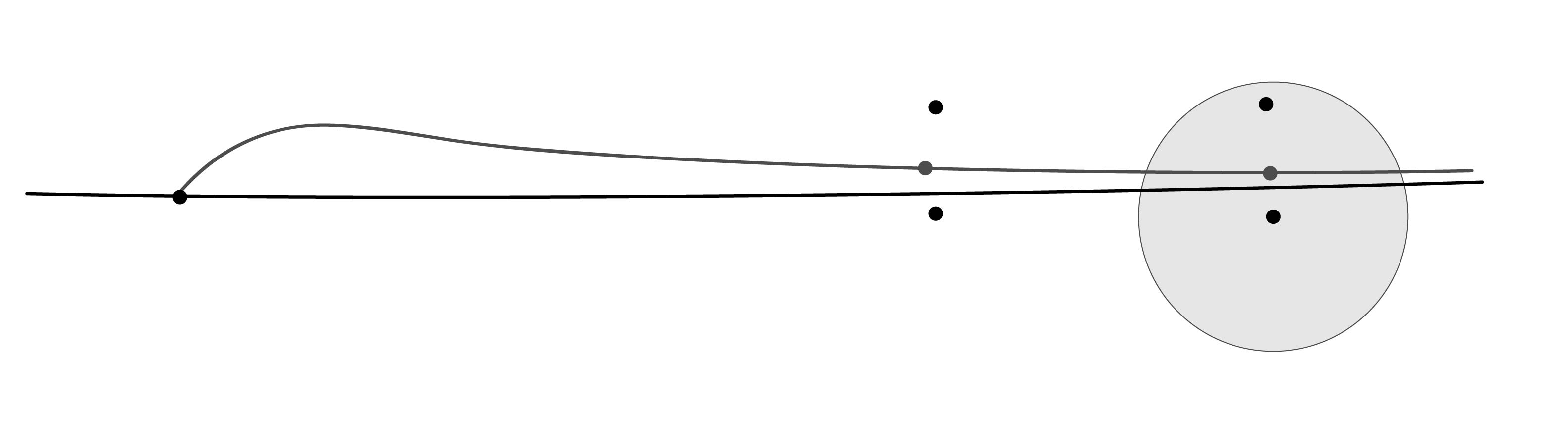} % figure file
 \put(30 , 18 ){\small \contour{white}{\textcolor{darkgray}{$ \gr_ {V_i}^{\ast} \sigma_i$} }}  
  \put(58 ,22 ){\small $\gr_{M_X}(\sigma_i)$}  
 \put(12 ,11 ){\small $\sigma_i$}  
   \put(59 ,10 ){\small $X$}  
 \put(80 ,22 ){\small \contour{white}{$\gr_{M_Y}(\sigma_i)$}} 
 \put(80 ,9 ){\small \contour{white}{$Y$}}  
  \put(5, 15){ \textcolor{black}{\small \contour{white}{$R$} } }   
% \put(, ){\textcolor{}{\tiny \contour{white}{$$}}}   
 %   \put( , ){\textcolor{}{$$}}  
 %   \put( , ){}  
%\put(, ){\color{}\vector(,){}} %{length}
      \end{overpic}
\caption{Approximating the Riemann surfaces $X$ and $Y$ by integral grafting.}\Label{fTeichmullerRayAndGrafting}
\end{figure}

 The holonomy representation of the marked hyperbolic surface $\sigma_i$ is a discrete and faithful representation $\rho_i\col \pi_1(S) \to \PSL_2\R$ unique up to conjugation by $\PSL_2\R$. 
 Since $R(t_i) = R_i(0)$ leaves every compact in $\TT$,  accordingly the hyperbolic structure $\sigma_i$ diverges to infinity  as  $i \to \infty$. 
 Thus $\rho_i$ leaves every compact subset in the character variety $\rchi$ as $i \to \infty$. 
 
Since $2\pi$-grafting does not change  holonomy,   $ {\rm Gr}_{M_X}(\sigma_i)$ is a $\CP^1$-structure with holonomy $\rho_i$ and its underlying Riemann surface structure is $\ep$-close to $X$, and 
 ${\rm Gr}_{M_Y}(\sigma_i)$ is a $\CP^1$-structure also with holonomy $\rho_i$ and its underlying Riemann surface structure is $\ep$-close to $Y$ in the Teichmüller metric. 
 Therefore, by the hypothesis of $\ep$, we have $(\gr_{M_X}(\sigma_i),  \gr_{M_Y}(\sigma_i)) \in W$. 
As $U_j$ is a bounded subset of $\BB$ for each $j = 1, \dots, n$,  accordingly $\cup_{j = 1}^n \Hol (U_j)$ is a bounded subset of $\rchi$. 
Thus, if $i$ is sufficiently large, then $\rho_i \not\in \cup_h \Hol (U_j)$. 
Therefore $( \Gr_{M_X} \sigma_i, \Gr_{M_Y} \sigma_i) \in \BB$ has holonomy outside of $\cup_j \Hol (U_j)$ and the pair of their Riemann surface structures is in $W$, as desired.

\subsection{When the orientations of $X$ and $Y$ are the opposite}\Label{sOppositeOrieantations}
We last prove \Cref{ANewPair}, supposing that the orientations of $X$ and $Y$ are opposite.
The proof is basically reduced in the previous case (\S\ref{sOrientationsCoincide}) if we appropriately reverse the orientation of the surfaces, as follows. 

First we can assume, without loss of generality, that $X \in \TT$ and $Y \in \TT^\ast$.
Let $Y^\ast$ be the complex conjugate of $Y$, so that $Y^\ast \in \TT$.

Similar to \S\ref{sOrientationsCoincide}, pick $\ep> 0$ so that the product of the $\ep$-negihborhood of $X$ in $\TT$ and the $\ep$-neighborhood of $Y$ is contained in $$W = W_1 \cap W_2 \cap \dots \cap W_n.$$
 Let $R\col \R \to \TT$ be a ``generic'' Teichmuller ray in $\TT$ passing the $\ep/3$-neighborhood of $X$ and the $\ep/3$-neighborhood of the complex conjugate $Y^\ast$ such that
 \begin{itemize}
\item its corresponding quadratic differential has only simple zeros, and
\item the unmarked Riemann surface $[R(t)]$ is dense in the moduli space $\MM$ of Riemann surfaces as $t \to -\infty$. 
\end{itemize}

Let $t_1 > t_2 > \dots$ be a sequnce such that 
\begin{itemize}
\item $t_i \to -\infty$ as $i \to \infty$; 
\item the unmarked Riemann surface $[R(t_i)]$ converges to $Z \in \MM$ as $i \to \infty$; 
\item the tangent vectore $[R'(t_i)]$ converges in the unite tangent vector of $\MM$ at $Z$ as $i \to \infty$. 
\end{itemize}
For each $i = 1, 2, \dots$, let $\sigma_i$ be the marked hyperbolic structure on $S$ uniformizaing $R(t_i)$. 
Let $\rho_i \col \pi_1(S) \to \PSL_2\R$ be the discrete faithful representation corresponding to the hyperbolic surface $\sigma_i$.

As $X, Y^\ast \in \TT$, by \Cref{sOrientationsCoincide},  for sufficiently large $i$,  
\begin{itemize}
\item  $\rho_i \not\in \cup_{j = 1}^n \Hol (U_j)$,
\item there are a multiloop $M_X$ and $M_Y$ with weighs $2\pi$-multiples on $\sigma_i$, such that  
 $$d_\TT(X, \gr_{M_X} \sigma_i) < \ep, d_\TT(Y^\ast, \gr_{M_{Y^\ast}} \sigma_i) < \ep,$$ and 
 \item  the $\CP^1$-structures $\Gr_{M_X} (\sigma_i)$  and $\Gr_{M_{Y^\ast}} (\sigma_i)$ have the smae holonomy $\rho_i$.
\end{itemize}

In the upper half-plane model of $\H^2$ in $\C$, the $\C$-conjugate $\sigma^\ast_i$ of the hyperbolic structure $\sigma_i$ on $S$ is a hyperbolic structure on $S^\ast$ with the same Fuchsian holonomy $\rho_i$. 
Let $M_Y$ denote the multiloop on $\sigma_i^\ast$ corresponding to $M_{Y^\ast}$ on $\sigma_i$ by the complex conjugation, so that $M_Y$ and $M_{Y^\ast}$ represent the same weighted multi-loop on the unoriented surface $\Sigma$. 
Therefore  $d_\TT(Y^\ast, \gr_{M_{Y^\ast}} \sigma_i) < \ep$ implies  $d_{\TT^\ast}(Y, \gr_{M_Y} \sigma_i^\ast) < \ep$.
Hence  $(\gr_{M_X} \sigma_i, \gr_{M_Y} \sigma_i^\ast) \in W$.
Therefore, if $i$ is sufficiently large, the projective grafting pair $(\Gr_{M_X} \sigma_i, \Gr_{M_Y} \sigma_i^\ast)$  in  $\BB$ has 
 holonomy outside $\cup_{i = 1}^n \Hol (U_i)$,  and the pair of their Riemann surface structures is in $W$, as desired.

  \subsubsection{Alternative proof}\Label{sAlternativeProof}
 We give a short alternative proof when the orientations of $X$ and $Y$ are opposite. 
\begin{proposition}\Label{ConjugatePairs}
For every $X \in \TT$, 
 there are infinitely many pairs of $\CP^1$-structures on the Riemann surface $X$ and on its complex conjugate $X^\ast$ which share real holonomy $\pi_1(\Sigma) \to \PSL(2, \R)$.
\end{proposition}
\begin{proof}
By  Tanigawa \cite[Theorem 3.2]{Tanigawa-97},  for every $\Z$-weighted multiloop $M$ on $S$, there is a unique (marked) hyperbolic structure $\sigma$ on $S$ such that its conformal grafting $\gr_M \sigma$ of $\sigma$ along $M$ is $X$. 
Then $\Gr_M \sigma$ is a $\CP^1$-structure on $S$ with the Fuchsian holonomy $\rho_\sigma\col \pi_1(S) \to \PSL(2, \R)$ associated with $\sigma$.
Then, by taking the complex conjugate of the developing map of $\Gr_M \sigma$, we obtain a $\CP^1$-structure on $S^\ast$ with the same Fuchsian holonomy $\rho_\sigma$, and its conformal structure is $X^\ast$. 

Since there are infinitely many $\Z$-weighted multiloops on $S$, accordingly, we have infinitely many pairs of $\CP^1$-structures on $X$ and on $X^\ast$ sharing Fuchsian holonomy. 
\end{proof}

\begin{remark}
\Cref{ConjugatePairs} provides the isomonodromic pairs in $\Psi^{-1}(X, X^\ast)$ with real Fuchsian holonomy. 
These isomonodromic $\CP^1$-structures on $X$ and on $X^\ast$ are related by complex conjugation.
Those properties are strong, and it is plausible that there are other isomodromic pairs in $\Psi^{-1}(X, X^\ast)$ having non-real holonomy. 
Seemingly, our approach in \S\ref{sOppositeOrieantations} may have more possibilities to be generalized in order to construct such isomonodromic pairs with non-real holonomy.
\end{remark}

\begin{theorem}
For each $X \in \TT$ and $Y \in \TT^\ast$, there are infinitely many pairs of $\CP^1$-structures on $X$ and on $Y$ sharing holonomy. 
\end{theorem}
\begin{proof}
Let $\BB_-$ be the union of the connected components of $\BB$ consisting of pairs of $\CP^1$-structures on $S$ and on $S^\ast$. 
Let $\Psi_-$ be the restriction of $\Psi$ to $\BB_-$. 
Clearly $\Psi_-\col \BB  \to \TT \times \TT^\ast$ is a complete local branched covering map since $\Psi$ is. 
Therefore, the covering degree of $\Psi_-$ is well-defined.
Namely, for all pairs  $(X, Y)$ in $ \TT \times \TT^\ast$, their fibers $\Psi_-^{-1}(X, Y)$ have the same cardinality counted with multiplicity. 

For every $X \in \TT$, by \Cref{ConjugatePairs}, the cartinarily of $\Psi^{-1}(X, X^\ast)$ (counted without multiplicity) is inifnite. 
Therefore, the covering degree of $\Psi | \BB_{-}$ is also infinite. 
Since the uniformization mapping $\Psi$ is a local branched covering map, for every $X \in \TT$ and $Y \in \TT^\ast$, the fiber $\Psi^{-1}(X, Y)$ is also infinite. 
\end{proof}

\bibliography{2025-8Intersection-of-Holonomy-varieties-arXiv}

\bibliographystyle{alpha}

\end{document}